\documentclass[a4paper,10pt]{amsart}

\usepackage{amssymb,latexsym}
\usepackage{amsthm}
\usepackage{amsmath}
\usepackage{amsfonts}
\usepackage{mathrsfs}
\usepackage{mathdots}
\usepackage{graphicx}
\usepackage[all]{xy}
\usepackage{fancyhdr}
\usepackage[top=1in,bottom=1in,left=1.25in,right=1.25in, marginparwidth=1in,]{geometry}

\usepackage[mathscr]{eucal}

\usepackage[title]{appendix}

\usepackage{tikz-cd}
\usepackage{stmaryrd}

\numberwithin{equation}{subsection}
\newtheorem{theorem}{Theorem}[section]
\newtheorem*{itheorem}{Theorem}

\newtheorem{definition}[theorem]{Definition}
\newtheorem{remark}[theorem]{Remark}

\newtheorem{proposition}[theorem]{Proposition}
\newtheorem{corollary}[theorem]{Corollary}
\newtheorem{lemma}[theorem]{Lemma}
\newtheorem{example}[theorem]{Example}

\def\Q{\mathbb{Q}}
\def\F{\mathbb{F}}
\def\R{\mathbb{R}}
\def\Z{\mathbb{Z}}

\def\N{\mathbb{N}}

\def\E{\mathcal{E}}

\def\Fc{\mathcal{F}}

\def\cN{\mathcal{N}}
\def\cO{\mathcal{O}}

\def\Bun{\mathrm{Bun}}
\def\Ext{\mathrm{Ext}}

\def\Fil{\mathrm{Fil}}
\def\Hom{\mathrm{Hom}}
\def\Lie{\mathrm{Lie}}

\def\Gal{\mathrm{Gal}}

\def\Proj{\mathrm{Proj}}
\def\Rep{\mathrm{Rep}}
\def\gcd{\mathrm{gcd}}

\def\Spa{\mathrm{Spa}}

\def\Spec{\mathrm{Spec}}

\def\deg{\mathrm{deg}}

\def\ra{\rightarrow}

\def\Fbar{\bar{F}}

\def\Phic{\Phi^\vee}

\def\s{\sigma}
\def\Gr{\mathrm{Gr}}

\def\ua{\underline{a}}
\def\ub{\underline{b}}
\def\uc{\underline{c}}
\def\ud{\underline{d}}
\def\ue{\underline{e}}

\newcommand{\mar}[1]{\marginpar{\tiny #1}}

\renewcommand\appendix{\par \setcounter{section}{0} \setcounter{subsection}{0} \gdef\thesection{ \Alph{section}}}

\begin{document}

\title{Weakly admissible locus and Newton stratification in $p$-adic Hodge theory}
\author{Miaofen Chen, Jilong Tong}
\date{}
\address{School of Mathematical Sciences\\
	Shanghai Key Laboratory of PMMP\\
	East China Normal University\\
	No. 500, Dong Chuan Road\\
	Shanghai 200241, China}\email{mfchen@math.ecnu.edu.cn}

\address{School of Mathematical Sciences \\ Capital Normal University \\ 105, Xi San Huan Bei Lu \\ Beijing, 100048, China}\email{jilong.tong@cnu.edu.cn}

\dedicatory{Dedicated to Luc Illusie, with our utmost gratitude} 
	
\renewcommand\thefootnote{}

\renewcommand{\thefootnote}{\arabic{footnote}}

\begin{abstract}The basic admissible locus $\mathcal{F}(G, \mu, b)^a$ inside the flag variety $\mathcal{F}(G, \mu)$, attached to a reductive group $G$ with a minuscule cocharacter $\mu$ of $G$, is a $p$-adic analogue of the complex analytic
period spaces. It has an algebraic approximation $\mathcal{F}(G, \mu, b)^{wa}$ inside the flag variety, called the weakly admissible locus. On the flag variety $\mathcal{F}(G, \mu)$, we have the Newton stratification which has the admissible locus as its unique open stratum. In this paper, we study the relation between the Newton strata and the weakly admissible locus. We show that $\mathcal{F}(G, \mu, b)^{wa}$ is maximal (in the sense that it's a union of Newton strata) is equivalent to $(G, \mu)$ weakly fully HN-decomposable, it's also equivalent to the condition that the Newton stratification is finer than the Harder-Narasimhan stratification. These equivalent conditions are generalizations of the fully HN-decomposable condition and the weakly accessible condition. Moreover, we give a criterion to determine whether a Newton stratum is completely contained in the weakly admissible locus involving $G$-bundles as extensions of $M$-bundles over the Fargues-Fontaine curve, where $M$ is a Levi subgroup of $G$. When $G=\mathrm{GL}_n$, we also give a combinatorial inductive criterion to determine whether a vector bundle over the Fargues-Fontaine curve is an extension of two given vector bundles.
\end{abstract}

\maketitle
\setcounter{tocdepth}{1}
\tableofcontents

\section*{Introduction}

Let $F$ be a finite extension of $\mathbb Q_p$, $\bar{F}$ an algebraic closure of $F$ with $C=\widehat{\bar F}$ its $p$-adic completion. Let $G$ be a connected reductive group over $F$, and $\{\mu\}$ the geometric conjugacy class of a minuscule cocharacter $\mu$. Let $\breve{F}$ be the completion of the maximal unramified extension $F^{un}$ of $F$ inside $\bar F$ , and we write $\sigma$ the Frobenius on $\breve{F}$ relative to $F$. Fix $b\in G(\breve{F})$. Let $E=E(G,\{\mu\})$ be the reflex field, that is, the field of definition of the geometric conjugacy class $\{\mu\}$, which is a finite extension of $F$ and is viewed as a subfield of $\bar F$. Attached to the pair $(G,\{\mu\})$, we have the flag variety $\mathcal F(G,\mu)$ which is a projective variety defined over the reflex field $E$. In the following, we shall consider its associated adic space, still denoted by $\mathcal F(G,\mu)$, defined over $\breve{E}$, the $p$-adic completion of the maximal unramified extension $E\cdot F^{un}\subset \bar F$ of $E$. Using the element $b\in G(\breve{F})$, Rapoport and Zink defined an open subset
\[
\mathcal F(G,\mu,b)^{wa}
\]
of the adic space $\mathcal F(G,\mu)$, called the \emph{weakly admissible locus} (historically also called the  \emph{$p$-adic period domain}, see also \cite{DOR}), as a vast generalization of the Drinfeld upper half plane. For $K/\breve{E}$ a complete field extension, the points of $\mathcal F(G,\mu,b)^{wa}(K)$ correspond to the weakly admissible filtered isocrystals equipped with a $G$-structure, whose underlying isocrystals with a $G$-structure are induced by $b$. If $K/\breve{E}$ is finite, by a fundamental result of Colmez-Fontaine in $p$-adic Hodge theory (\cite{CoFo}), such filtered isocrystals are \emph{admissible} in the sense that they come from crystalline Galois representations. If $K/\breve{E}$ is an arbitrary complete field extension, we cannot always get Galois representations attached to points in $\mathcal F(G,\mu,b)^{wa}(K)$. This phenomenon leads Rapoport and Zink to conjecture that there exists an open subspace, called \emph{admissible locus}
\[
\mathcal F(G,\mu,b)^{a}
\]
inside $\mathcal F(G,\mu,b)^{wa}$ with the same classical points, i.e., the same points with values in finite extensions of $\breve{E}$ as $\mathcal F(G,\mu,b)^{wa}$, together with a $p$-adic \'etale local system with additional structures on $\mathcal F(G,\mu,b)^a$ interpolating the crystalline Galois representations attached to its classical points. By contrast to the subspace $\mathcal F(G,\mu,b)^{wa}$ of weakly admissible locus, the construction of $\mathcal F(G,\mu,b)^{a}$ is quite mysterious, and was previously known only for certain triples $(G,\mu,b)$ by the work of Hartl (\cite{Har1} \cite{Har}) and Faltings (\cite{Fal}).

Recently, thanks to the new progress in $p$-adic Hodge theory, especially the discovery of Fargues-Fontaine curve, we can associate in a natural way to $b$ a $G$-bundle
\[
\mathcal E_b
\]
on the Fargues-Fontaine curve $X$, whose isomorphism class only depends on the $\sigma$-conjugacy class of $b$. Furthermore we can use each $C$-point $x\in\mathcal F(G,\mu)(C)$ to modify $\mathcal E_b$ \`a la Beauville-Laszlo to get another $G$-bundle
\[
\mathcal E_{b,x}
\]
on $X$. We define $\mathcal F(G,\mu,b)^a$ as the (necessarily open) subspace of $\mathcal F(G,\mu)$ stable under generalization, whose set of $C$-points is given by
\[
\mathcal F(G,\mu,b)^{a}(C)=\{x\in \mathcal F(G,\mu)(C)| \mathcal E_{b,x} \textrm{ is the trivial }G\textrm{-bundle}\}.
\]
The existence of the \'etale local system on $\mathcal F(G,\mu,b)^{a}$ as in the conjecture of Rapoport-Zink is then a consequence of the work of Fargues-Fontaine, Kedlaya-Liu, and Scholze.

As a result, we have two open adic subspaces inside the flag varieties $\mathcal F(G,\mu)$:
\[
\mathcal F(G,\mu,b)^{a}\subseteq \mathcal F(G,\mu,b)^{wa}\subseteq \mathcal F(G,\mu).
\]
It is a natural question to understand the structure of these subspaces. Compared with the weakly admissible locus $\mathcal F(G,\mu,b)^{wa}$, which has been intensively studied since the work of Rapoport-Zink, we don't know much about the admissible locus $\mathcal F(G,\mu,b)^{a}$. As the first step, we would like to see when $\mathcal F(G,\mu,b)^{a}$ coincides with $\mathcal F(G,\mu,b)^{wa}$ (\cite[Question A.20]{Ra2}). If $G=\mathrm{GL}_n$, Hartl gave a complete solution to this question. In particular, when $b\in \mathrm{GL}_n(\breve{F})$ is basic (cf. Section \ref{Section_Newton map}), up to twist by a central cocharacter, the equality
\[
\mathcal F(\mathrm{GL}_n,\mu, b)^a=\mathcal F(\mathrm{GL}_n, \mu, b)^{wa}
\]
holds if and only if
\[
\mu\in \{ (1, 0^{(n-1)}), (1^{(n-1)},0)\},\quad  \textrm{or}\quad n=4  \textrm{ and }\mu=(1,1,0,0).
\]
For a general reductive group $G$, Fargues and Rapoport conjectured that, at least when $b$ is basic, there exists a group-theoretic condition which expresses precisely the coincidence of $\mathcal F(G,\mu,b)^{a}$ and $\mathcal F(G,\mu,b)^{wa}$. By a recent joint work of the first-named author with Fargues and Shen, this conjecture is now a theorem.

\begin{itheorem}[Chen-Fargues-Shen \cite{CFS}] Suppose that $b\in G(\breve{F})$ is basic. Then $\mathcal F(G,\mu,b)^a=\mathcal F(G,\mu,b)^{wa}$ if and only if the pair $(G,\mu)$ is fully Hodge-Newton decomposable.
\end{itheorem}

Here the full Hodge-Newton decomposability condition is purely group-theoretic. It was first introduced and systematically studied by G\"ortz, He and Nie in \cite{GoHeNi}. Moreover, they also give equivalent conditions in terms of affine Deligne-Lusztig varieties for this condition. We refer to \S~\ref{sec: weakly fully HN} for the precise definition of the fully HN-decomposability condition. In her succeeding work, the first-named author proved the analogue of the above theorem  in the non-basic case (\cite{Ch1}). Furthermore, Shen generalized the Fargues-Fontaine conjecture for non-minuscule cocharacters (\cite{Sh2}).

On the other hand, Rapoport investigated in \cite{Ra2} the question that in which situations the subspace $\mathcal F(G,\mu,b)^{wa}$ (resp. $\mathcal F(G,\mu,b)^{a}$) is the whole flag variety $\mathcal F(G,\mu)$. If this is indeed the case, we call, following Rapoport, that the triple $(G,\mu, b)$ is \emph{weakly accessible} (resp. \emph{accessible}). Clearly, if $(G,\mu,b)$ is accessible, then it is weakly accessible. In \cite[A.4-A.5]{Ra2}, Rapoport gave a full classification for such triples. Indeed, if $(G, \mu, b)$ is weakly accessible, then $b$ is basic and hence determined by $(G, \mu)$. For instance, in the $\mathrm{GL}_n$-case, up to a twist by a central cocharacter, weak accessibility is equivalent to the condition that $\mu=(1^{(r)},0^{(n-r)})$ with $\mathrm{gcd}(n, r)=1$.


To go further, we consider the Newton stratification on the flag variety $\mathcal F(G,\mu)$:
\[
\mathcal F(G,\mu)=\coprod_{[b']\in B(G)}\mathcal F(G,\mu,b)^{[b']}.
\]
Here $B(G)$ is the set of $\sigma$-conjugacy classes of $G(\breve{F})$ and $\mathcal F(G,\mu,b)^{[b']}$ is a subspace of $\mathcal F(G,\mu)$ stable under generalization, whose $C$-points are those $x\in \mathcal F(G,\mu)(C)$ such that $\mathcal E_{b,x}\simeq \mathcal E_{b'}$. Rapoport determined in \cite{Ra2} which Newton strata $\mathcal F(G,\mu,b)^{[b']}$ are non-empty when $b$ is basic.  Apparently
\[
\mathcal F(G,\mu,b)^a=\mathcal F(G,\mu,b)^{[1]}
\]
is the unique open stratum of this stratification. Naturally, we want to compare the Newton stratification with the weakly admissible locus. For example, we would like to classify the strata which have a non-empty intersection with $\mathcal F(G,\mu,b)^{wa}$ (\cite[Conjecture 5.2]{Ch1}). By the work of Chen-Fargues-Shen mentioned above combined with a recent preprint of Viehmann (\cite{Vi}), we have a complete answer to this question. Suppose a Newton stratum $\mathcal F(G,\mu,b)^{[b']}$ is non-empty, then 

\[
\mathcal F(G,\mu,b)^{[b']}\cap \mathcal F(G,\mu,b)^{wa}\neq \emptyset
\]
if and only the triple $(G,\nu_{b}-\mu,b')$ is Hodge-Newton indecomposable (cf. Definition \ref{def_HN decomp}). In particular,

\[
\mathcal F(G,\mu,b)^{wa}\subseteq \coprod_{\substack{ [b']\in B(G) \textrm{ s.t. }(G,\nu_{b}-\mu,b') \\ \textrm{ is Hodge-Newton indecomposable} } } \mathcal F(G,\mu,b)^{[b']} \subseteq \mathcal F(G,\mu).
\]
We say that the weakly admissible locus $\mathcal F(G,\mu,b)^{wa}$ is \emph{maximal} if the first inclusion above is an equality, or equivalently (by \cite{Vi}), that $\mathcal F(G,\mu,b)^{wa}$ is a union of Newton strata. Note that the weakly admissible locus is maximal if $\mathcal F(G,\mu,b)^{wa}=\mathcal F(G,\mu,b)^{a}$ or $\mathcal F(G,\mu)$. Therefore the maximality condition of the weakly admissible locus can be considered as a uniform generalization of the above two extreme cases that we have discussed. Our first main result is a group theoretic characterization and a  geometric characterization of the  maximality of the weakly admissible locus when $b$ is basic under the minuscule condition.

\begin{itheorem}[Theorem \ref{thm_main}]Suppose $b$ basic and $\mu$ minuscule. Then the following three assertions are equivalent:
\begin{enumerate}
    \item[(a)] $\Fc(G, \mu, b)^{wa}$ is maximal;
    \item[(b)] $(G,\mu)$ is weakly fully Hodge-Newton decomposable (cf. Definition \ref{def:weakly-full-HN});
    \item[(c)] the Newton stratification is finer than the Harder-Narasimhan stratification in the sense that every Harder-Narasimhan stratum is a union of some Newton strata (see \S~\ref{sec:Newton-and-HN} for a brief review of the Harder-Narasimhan stratification). 
\end{enumerate}
\end{itheorem}

The weakly fully Hodge-Newton decomposability condition in assertion (b) is a group theoretic condition. This condition is a common generalization of the full Hodge-Newton decomposability and weak accessibility condition, and it includes also some new cases which are not covered by the last two conditions. For example, the following new cases arise in the $\mathrm{GL}_n$-case:
\[
n \textrm{ is even and } \mu\in \{(1^{(2)}, 0^{(n-2)}), (1^{(n-2)},0^{(2)})\}, \quad \textrm{or}\quad n=6 \textrm{ and } \mu=(1,1,1,0,0,0).
\]
In Theorem \ref{thm:classification-weak-HN-dec}, we classify all the weakly fully HN-decomposable pairs when $G$ is absolutely simple and adjoint.

We would also like to have a practical criterion to see if a single Newton stratum $\mathcal F(G,\mu,b)^{[b']}$ is entirely contained in the weakly admissible locus. Our second main result of this paper is such a criterion. 

\begin{itheorem}[Theorem \ref{thm:criterion-for-a-singule-stratum}]Let $\mu$ be a minuscule cocharacter of $G$ and $b\in G(\breve F)$ be basic. Suppose $\Fc(G, \mu, b)^{[b']}\neq \emptyset$. Then
\[
\Fc(G, \mu, b)^{[b']}\nsubseteq\Fc(G, \mu, b)^{wa}
\]
if and only if there exist some maximal proper standard Levi subgroup $M$ of $H$, the quasi-split inner form of $G$ over $F$, and an element $w$ in the Weyl group of $H$, satisfying the following two conditions:
\begin{enumerate}
\item $b$ has a reduction $b_M$ to $M$ and $w$ is $\mu$-negative for $M$ (cf. Definition \ref{def_mu negative}); and
\item $[b']$ is an extension of some $[b'_M]\in B(M, \kappa(b_M)-\mu^{w,\#}, \nu_{b_M}-\mu^{w,\diamond})$, where  $B(M, ...)$ denotes a generalized Kottwitz set (cf. \S \ref{sec:generalized-Kottwitz-sets} and definition \ref{defn_extension}).
\end{enumerate}
\end{itheorem}

If $G=\mathrm{GL}_n$, the two conditions in the theorem can be reformulated in a more down-to-earth way as follows:
\begin{itemize}
    \item[(i)] the isocrystal $(\breve F^{n},b\circ \sigma)$ has a decomposition by sub-isocrystals $(\breve F^{n},b\circ \sigma)=D_1\oplus D_2$; and 
    \item[(ii)] the vector bundle $\mathcal E_{b'}$ can be written as a extension of $\mathcal E_2$ by $\mathcal E_1$, such that the vector bundle $\mathcal E_1$ is of degree $>0$ and that each $\mathcal E_i$ can be realized as a minuscule modification of $\mathcal E(D_i)$, the vector bundle attached to the isocrystal $D_i$. 
\end{itemize}   
If the Newton stratum $\mathcal F(\mathrm{GL}_n,\mu,b)^{[b']}$ contains a point which is not weakly admissible, such a point produces a decomposition of the isocrystal $(\Breve{F}^n,b\circ \sigma)$ and an extension of vector bundles as required above. Conversely, once these two conditions are fulfilled, it is not hard to realize $\mathcal E_{b'}$ as a modification of $\mathcal E_{b}$ in the way compatible with the two modifications given in (ii). But \`a priori the modification that we get randomly may be \emph{not} minuscule. We need to adjust the initial modification properly to obtain a minuscule one, and this is done in Proposition \ref{prop:modification M to G} for general quasi-split groups. In this way, we obtain a point in $\mathcal F(\mathrm{GL}_n,\mu,b)^{[b']}$ violating the weak admissibility.

It is now obvious to the readers that, our problem of classifying the Newton strata contained in the weakly admissible locus is closely related to the difficult question to determine whether a $G$-bundle is an extension of a $M$-bundle on the Fargues-Fontaine curve, with $M$ a Levi subgroup of the quasi-split inner form of $G$. When the $M$-bundle is semi-stable, this question is studied by \cite{BFH} for $\mathrm{GL}_n$ and \cite{Vi} for arbitrary $G$.   In \S~\ref{Sec_extension}, we study this question for $\mathrm{GL}_n$ for any $M$-bundle. Inspired by the previous work of Schlesinger on the classification of extensions of vector bundles on the projective line (\cite{Schl}), we introduce in Definition \ref{Def_tildeExt}, for two given vector bundles $\mathcal E_i$ ($i=1,2$), a combinatorial constrain on the slopes of a vector bundles $\mathcal E$ which can be realized as an extension of $\mathcal E_2$ by $\mathcal E_1$. We check that this constrain is indeed necessary, and shows that our combinatorial condition coincides with that of \cite{BFH} when the two vector bundles $\mathcal E_i$ ($i=1,2$) are semistable. One might expect that this combinatorial condition is also sufficient to classify all the extensions, but unfortunately this is \emph{not} the case (cf. Example \ref{ex:counter-example-for-ext}). We  can only give an inductive criterion to see if a vector bundle $\mathcal E$ verifying our combinatorial conditions indeed comes from an extension of vector bundles. 

\vskip 2mm

After we finished this work, we noticed that Hong has posted a new preprint \cite{Ho2} on his webpage. In his new work, Hong has classified independently all the extensions of vector bundles over the Fargues-Fontaine curve in a similar way as we did in \S~\ref{Sec_extension}. 

\vskip 2mm

We briefly describe the structure of this article. In \S \ref{Sec_prelim}, we collect some preliminaries about modifications of $G$-bundles on the Fargues-Fontaine curve that we will need in the sequel. In \S \ref{sec_weakly HN decomp}, we discuss the weakly fully HN-decomposable condition and give the  minute criterion for the weakly fully HN-decomosability (Proposition \ref{prop_minute criterion}). Using the minute criterion, we give the classification of the weakly fully HN-decomposable pairs $(G,\mu)$ when $G$ is absolutely simple and adjoint in Theorem \ref{thm:classification-weak-HN-dec}.  In \S \ref{Sec_maximal wa}, we give several equivalent conditions when the weakly admissible locus is maximal in Theorem \ref{thm_main}. In \S \ref{Sec_single Newton}, we give a criterion when a single Newton stratum is contained in the weakly admissible locus in Theorem \ref{thm:criterion-for-a-singule-stratum}. In \S \ref{Sec_extension}, we introduce a combinatorial condition which is necessary for a vector bundle over the Fargues-Fontaine curve which can be realized as an extension of two given vector bundles (Proposition \ref{Prop_tildeExt is Necessary condition}). Moreover, we give a combinatorial inductive criterion to classify all extensions of vector bundles in Proposition \ref{Prop_classification of extension by Ext^1}. At the end of this section, we apply this criterion to determine all the Newton strata contained in the weakly admissible locus for $\mathrm{GL_n}$ in an explicit example.  In the appendix, we show that our combinatorial constrain above is also sufficient in some special cases, yielding hence classifications of extensions in some new cases not yet covered in the literature.
\bigskip

\textbf{Acknowledgments.} We would like to thank Eva Viehmann to share with us the main results of \cite{NGVi} so that we can use them as a tool in the proof of Theorem \ref{thm_main}. We thank Laurent Fargues, Xuhua He and Sian Nie for helpful discussions. We also thank Serin Hong, Luc Illusie and Eva Viehmann for their comments on the previous version of this paper. The first author is partially supported by NSFC grant No.11671136 and No.12071135. The second author is partially supported by One-Thousand-Talents Program of China.

\section*{Notations}
We use the following notations:
\begin{itemize}
\item $F$ is a finite degree extension of $\Q_p$ with residue field $\F_q$ and a uniformizer $\pi_F$.
\item $\Fbar$ is an algebraic closure of $F$ and $\Gamma = \Gal (\Fbar |F)$.
\item $\breve{F} =\widehat{F^{un}}$ is the completion of the maximal unramified extension $F^{un}\subset \Fbar$, with Frobenius $\s$.
\item $G$ is a connected reductive group over $F$, and $H$ is a quasi-split inner form of $G$.
\item $A\subseteq T \subseteq B$, where A is a maximal split torus, $T=Z_H (A)$ is the centralizer of $A$ in $T$, and $B$ is a Borel subgroup in $H$.
\item $(X^* (T),\Phi, X_*(T), \Phic )$ is the absolute root  of $H$ with positive roots $\Phi^+$ and simple roots $\Delta $ with respect to the choice of $B$. Write $\Delta^\vee$ for the corresponding simple coroots. 
\item $W=N_H(T)/T$ is the absolute Weyl group of $T$ in $H$, and $w_0$ is the longest length element in $W$.
\item $(X^*(A), \Phi_0,  X_*(A),\Phic_0)$ is the relative root datum of $H$ with positive roots $\Phi^+_0$ and simple (reduced) roots $\Delta_0$.
\item If $M$ is a standard Levi subgroup in $H$ we denote by $\Phi_M$ the corresponding roots or coroots showing up in $\Lie\, M$, and by $W_M$ the Weyl group of $M$. If $P$ is the standard parabolic subgroup of $H$ with Levi component $M$, sometimes we also write $W_P$ for $W_M$. For $\alpha\in \Delta_0$, let $M_{\alpha}$ be the standard Levi subgroup of $H$ with $\Delta_{M_{\alpha}, 0}=\Delta_0\backslash\{\alpha\}$. Let $P_{\alpha}$ be the standard parabolic subgroup of $H$ with Levi component $M_{\alpha}$.
\end{itemize}

\section{Preliminaries}\label{Sec_prelim}
In this section, we collect some basic definitions and results needed in the sequel. 

\subsection{Generalized Kottwitz sets} Let $B(G)$ be the set of $\sigma$-conjugacy classes of elements in $G(\breve F)$. Kottwitz has defined two maps, the \emph{Newton map} and the \emph{Kottwitz map}, on $B(G)$, which are of fundamental importance in this paper.

\subsubsection{The Newton map}\label{Section_Newton map} Let $\mathbb D$ be the pro-torus whose character group is $\mathbb Q$, and consider
\[
\mathcal N(G):=\left(\Hom(\mathbb D_{\Fbar},G_{\Fbar} )/G(\Fbar)\right)^{\Gamma},
\]
with $\Gamma=\mathrm{Gal}(\Fbar/F)$. The Newton map is a map
\begin{equation}\label{eq:Newton-map}
\nu=\nu_G:B(G)\longrightarrow\mathcal N(G)
\end{equation}
defined as follows. Let $b\in G(\breve{F})$. For $(V,\rho)$ an object in $\Rep_F(G)$, write $V_{\breve{F}}:=V\otimes_F\breve F$. The element $b$ induces an isocrystal
\begin{equation}\label{eq:isocrystal-D-b-rho}
D_{b,\rho}=(D_{b,\rho},\varphi_{b,\rho}):=(V_{\breve{F}},\rho(b)(1\otimes \sigma)),
\end{equation}
whose slope decomposition gives rise to a $\mathbb Q$-graded vector space over $\breve{F}$. In this way, we obtain an exact tensor functor
\[\begin{split}
\Rep_{F}(G)&\longrightarrow \mathbb Q-\mathrm{Grad}_{\breve F},\\ \quad (V,\rho)&\longmapsto \textrm{slope decomposition of }D_{b,\rho},
\end{split}\]
from the category $\Rep_F(G)$ of rational algebraic representations of $G$ over $F$ to the category $\mathbb Q-\mathrm{Grad}_{\breve F}$ of $\Q$-graded $\breve F$-vector spaces. This tensor functor in turn, by Tanaka duality, gives a morphism of algbraic groups
\[
\nu_{b}:\mathbb D_{\breve F}\longrightarrow G_{\breve F}.
\]
As the slope decomposition of $D_{b,\rho}$ is defined over $F^{un}$, the group $\nu_b$ is naturally defined over $F^{un}$, thus can be viewed as an element of $X_*(G)_{\mathbb Q}=\Hom(\mathbb D_{\Fbar},G_{\Fbar})$. Moreover, its $G(\Fbar)$-conjugacy class $[\nu_b]$ depends only on the $\sigma$-conjugacy of $b$, and is invariant under the action of $\Gamma$. The Newton map $\nu$ in \eqref{eq:Newton-map} is then defined by setting $\nu([b])=[\nu_b]\in \mathcal N(G)$.

Recall also that, an element $[b]\in B(G)$ is called \emph{basic} if the rational cocharacter $\nu_b$ factors through the center of $G_{\breve F}$. In this case, we will also say that $b$ is basic. 

\subsubsection{The Kottwitz map}\label{sec:Kottwitz-map} Let $\pi_1(G)$ denote the algebraic fundamental group of $G$. Indeed,
\[
\pi_1(G)=X_*(T)/\langle \Phi^{\vee}\rangle,
\]
where $\langle \Phi^{\vee}\rangle\subset X_*(T)$ is the subgroup generated by $\Phi^{\vee}$. Up to a canonical isomorphism, $\pi_1(G)$ does not depend on the choice of $T$, and is naturally equipped with an action of $\Gamma=\mathrm{Gal}(\Fbar/F)$. When $G=\mathbb {GL}_{n,F}$, we have a canonical identification $\pi_1(G)\simeq \mathbb Z$.

The Kottwitz map, written $\kappa_G$ or even $\kappa$ if there is no confusion, is a map
\[
\kappa=\kappa_G: B(G)\longrightarrow \pi_1(G)_{\Gamma}
\]
which is functorial on $G$ as an $F$-reductive group, such that the following square is commutative:
\[
\xymatrix{\breve{F}^*\ar[r]^{v_{\pi_F}(-)} \ar@{->>}[d]_{\rm can}& \mathbb Z \ar[d]^{\simeq}_{\rm can}\\
 B(\mathbb G_{m,F})\ar[r]^{\kappa_{\mathbb G_{m,F}}}&  \pi_1(\mathbb G_{m,F})_{\Gamma}},
\]where $v_{\pi_F}(-)$ is the $\pi_F$-adic valuation on $F$.
The Kottwitz map $\kappa_G$ is uniquely determined by this property. The reader can find in \cite[\S~1.5]{CFS} a more direct construction of $\kappa_G$ via the abelianization of Kottwitz set \`a la Borovo\"i.

\begin{theorem}[\cite{Kot2} 4.13]\label{thm:Kottwitz} The map
\[
(\nu,\kappa): B(G)\longrightarrow \mathcal N(G)\times \pi_1(G)_{\Gamma}, \quad [b]\mapsto ([\nu_b],\kappa([b]))
\]
is injective.
\end{theorem}

\begin{remark}\label{rem:H1-and-pi} \begin{enumerate}
\item By a classical result of Steinberg, $H^1(\breve F,G)$ is trivial, from where one deduces a natural injective map \emph(\cite[1.8.3]{Kot1}\emph):
\[
H^1(F,G)\rightarrow B(G)
\]
whose image is the set of elements $[b]\in B(G)$ with trivial Newton vector $\nu_b$.
\item Composing the natural map in (1)  with the Kottwitz map $\kappa_G$ above, we get a map
\[
H^1(F,G)\longrightarrow \pi_1(G)_{\Gamma}.
\]
It is known that this map is injective, with image the subgroup $\pi_{1}(G)_{\Gamma,{\rm tor}}$ of torsion elements in $\pi_1(G)_{\Gamma}$. If $G$ is semisimple, the group $\pi_1(G)$ is of torsion. Consequently, the map above is bijective provided $G$ semisimple.
\end{enumerate}
\end{remark}

\subsubsection{Poset structure on $\mathcal N(G)$ and on $B(G)$} 
We have the following well-known bijective maps
\[
X_*(A)_{\mathbb Q}^+\stackrel{\sim}{\longrightarrow }(X_{*}(T)_{\mathbb Q}/W)^{\Gamma}\stackrel{\sim}{\longrightarrow} (X_*(G)_{\mathbb Q}/G(\Fbar))^{\Gamma}.
\]
Here $X_*(A)^+_\Q$ is the closed Weyl chamber attached to the basis $\Delta_0$
\[
X_*(T)_{\mathbb Q}^+=\{\lambda\in X_*(A)_{\mathbb Q} | \langle \lambda, \alpha\rangle\geq 0 \textrm{ for all }\alpha\in\Delta_0\}.
\]
\begin{remark}Up to a canonical bijection, the bijective map $X_{*}(T)_{\mathbb Q}/W\longrightarrow X_*(G)/G(\Fbar)$ does not depend on the choice of the maximal torus $T$. This allows us to equip a group structure on $X_*(G)_{\mathbb Q}/G(\Fbar)$, and the latter becomes a $\Gamma$-module in this way.
\end{remark}
One define a partial order $\preceq $ on $\mathcal N(G)$ (resp. on $B(G)$): for $v,v '\in \mathcal N(G)$ (resp. $[b],[b']\in B(G)$), write $\tilde{v}$ and $\tilde{v}'$ the representatives of $v$ and $v'$ (resp. of $[\nu_b]$ and $[\nu_{b'}]$) in $X_*(A)^+$, and set
\[
v\preceq v' \quad (\textrm{resp. }[b]\preceq  [b'] ) \quad \Longleftrightarrow \quad \tilde{v}'-\tilde{v}\in \left\{\sum_{\alpha\in \Delta}n_{\alpha}\alpha^{\vee}| n_{\alpha}\in \mathbb Q_{\geq 0}\right\}\subset X_*(A)_{\mathbb Q}.
\]
One can check that, the definition of the partial order $\preceq$ only depends on $G$. Using this partial order, we can equip $B(G)$ with the topology such that
\[
[b']\in \overline{\{[b]\}}\quad  \Longleftrightarrow \quad [b]\preceq  [b'].
\]
The set $\mathcal N(G)$ can be topologized in a similar way.
\subsubsection{The generalized Kottwitz sets}\label{sec:generalized-Kottwitz-sets}
Recall the Kottwitz set
\[
B(G,\mu):=\{[b]\in B(G)| [\nu_b]\preceq  \mu^{\diamond}, \ \kappa([b])=\mu^{\#}\},
\]
where \begin{itemize}
    \item $\mu^{\diamond}$ denotes the Galois average of the cocharacter $\mu$ in $X_*(G)_{\mathbb Q}/G(\Fbar)$:
\[
\mu^{\diamond}:=\frac{1}{[\Gamma:\Gamma_{\mu}]}\sum_{\tau\in \Gamma/\Gamma_{\mu}}\mu^{\tau} \in X_*(G)_{\mathbb Q}/G(\Fbar),
\]
with $\Gamma_{\mu}$ the stabilizer of the action $\Gamma$ on $\mu\in X_*(G)$, and 
\item  $\mu^{\#}$ is the image of $\mu$ in $\pi_1(G)_{\Gamma}$ via the canonical maps below
\[
X_*(T)\longrightarrow \pi_1(G)=X_*(T)/\langle\Phi^{\vee}\rangle\longrightarrow \pi_1(G)_{\Gamma}.
\] \end{itemize}
One checks that both $\mu^{\diamond}$ and $\mu^{\#}$ depend only on the conjugacy class $\{\mu\}$ of $\mu$. Set also
\[
A(G,\mu):=\{[b]\in B(G)| [\nu_b]\preceq  \mu^{\diamond}\}.
\]
It is known that $B(G,\mu)\subset A(G,\mu)$ are both finite sets equipped with the partial order $\preceq $ induced from $B(G)$. Later we will need a generalized version of the Kottwitz sets introduced in \cite[\S~4.1]{CFS}.

\begin{definition} Let $\epsilon \in \pi_1(G)_{\Gamma}$ and $\delta\in \mathcal N(G)$. Set
\[
B(G,\epsilon,\delta):=\{[b]\in B(G)| [\nu_b]\preceq \delta, \ \kappa([b])=\epsilon\}.
\]
\end{definition}
If $\epsilon=\mu^{\#}$ and $\delta=\mu^{\diamond}$, we recover the Kottwitz set $B(G,\mu)$ above.

\begin{remark}In this paper, for the generalized Kottwitz set $B(G, \epsilon, \delta)$, $\delta$ is written additively but not multiplicatively such as in \cite{CFS} and \cite{Ch1}. For example, the notation $B(G, 0, \nu_b\mu^{-1, \diamond})$ in \cite{CFS} and \cite{Ch1} is written $B(G, 0, \nu_b-\mu^{\diamond})$ here.
\end{remark}

\subsection{Fargues-Fontaine curves and $G$-bundles}

\subsubsection{Fargues-Fontaine curves}Let $K$ be a perfectoid field over $\mathbb F_q$, and $\omega_K\in K$ with $0<|\omega_K|<1$. Let $W_{\cO_F}(\cO_K):=W(\cO_K)\otimes_{W(\mathbb F_q)}\cO_F$, and
\[
\mathcal Y_K:=\Spa(W_{\cO_F}(\cO_K))\setminus V(\pi_F[\omega_K]).
\]
The latter is an adic space over $F$ equipped with an automorphism $\varphi$ induced from the Frobenius on $K$ relative to $\mathbb F_q$. The quotient
\[
\mathcal X=\mathcal X_K:=\mathcal Y_K/\varphi^{\mathbb Z}
\]
is an adic curve which can be algebraized, and the resulting $F$-scheme, the \emph{Fargues-Fontaine curve} associated with the perfectoid field $K$, is given by
\[
X=X_{K}:=\Proj \left(\bigoplus_{i}B_K^{\varphi=\pi_F^i}\right)
\]
with $B_K:=H^0(\mathcal Y_K,\cO_{\mathcal Y_K})$. It is known that $X$ is a one-dimensional Noetherian scheme. In the following, we shall take $C$ a complete algebraically closed extension of $\Fbar$, and $K=C^{\flat}$ the tilt of $C$. In particular, the curve $X$ is equipped with a closed point $\infty\in X$ with residue field $k(\infty)=C$ and complete local ring $\hat{\cO}_{X,\infty}= B_{dR}^+(C)$. 

\subsubsection{$G$-bundles on Fargues-Fontaine curves}In the following, we shall use intensively the notion of $G$-bundles on $X=X_{C^{\flat}}$ with $C$ as above. Recall first that, a \emph{$G$-bundle} on $X$ is a (right) $G$-torsor on $X$ for the \'etale topology. From a $G$-bundle $\mathcal E$ on $X$, one can construct an exact tensor functor
\[\begin{split}
\mathrm{Rep}_F(G)&\longrightarrow \Bun_X, \\ (V,\rho)&\longmapsto \mathcal E\times^{G,\rho} (\cO_X\otimes_F V).\end{split}
\]
Here $\Bun_X$ is the category of vector bundles on $X$ and $\mathcal E\times^{G,\rho} (\cO_X\otimes_F V)$ denotes the quotient of $\mathcal E\times (\cO_X\otimes_F V)$ by the following action of $G$: for $g$ (resp. $x$) a local section of $G$ (resp. of $\mathcal E$), and $v\in V$, 
\[
g\cdot (x,1\otimes v):=(xg^{-1}, 1\otimes \rho(g)(v)). 
\]
Conversely, every exact tensor functor
\begin{equation}\label{eq:functor-Eb}
\Rep_F(G)\longrightarrow \Bun_X
\end{equation}
arises in this way (see \cite[\S~1]{F3}). Moreover, attached to a $G$-bundle $\mathcal E$, we have the Harder-Narasimhan polygon $\nu_{\mathcal E}\in \mathcal N(G)$ and the $G$-equivariant first Chern class $c_1^G(\mathcal E)\in \pi_1(G)_{\Gamma}$. See \cite[\S~1.4]{CFS} for more details.

Let $b\in G(\breve{F})$. For $(V,\rho)$ an object in $\Rep_F(G)$, consider the isocrystal $
D_{b,\rho}$ of \eqref{eq:isocrystal-D-b-rho}. It gives rise the vector bundle $\mathcal E(D_{b,\rho})$ on $X$ associated with the graded module
\[
M(D_{b,\rho}):=\bigoplus_{i\geq 0}(B_K\otimes_{\breve{F}}D_{b,\rho})^{\varphi\otimes \varphi=\pi_F^i}
\]
over the graded ring $P_K:=\bigoplus_i B_K^{\varphi=\pi_F^i}$. In this way, we obtain an exact tensor functor of the form \eqref{eq:functor-Eb}. The resulting $G$-bundle on $X$ is denoted by 
\[
\mathcal E_b
\]
in the following. The isomorphism class of the isocrystal $(D_{b,\rho},\varphi_{b,\rho})$, and hence the isomorphism class of the $G$-bundle $\mathcal E_b$, depend only on the $\sigma$-conjugacy class of $b$. So we obtain a map
\begin{equation}\label{eq:BG-and-G-bundles}
B(G)\longrightarrow H^1(X,G), \quad [b]\mapsto [\mathcal E_b].
\end{equation}
Here $B(G)$ is the set of $\sigma$-conjugacy classes $G(\breve{F})$.
\begin{theorem}[\cite{F3} Th\'eor\`eme 5.1]\label{thm:Fargues} The map \eqref{eq:BG-and-G-bundles} is bijective. Moreover, for $[b]\in B(G)$, $\nu_{\mathcal E_b}=-w_0[\nu_b]\in \mathcal N(G)$ and $c_1^G(\mathcal E_b)=-\kappa_G([b])$.
\end{theorem}

\begin{remark}\label{rem:Eb-canonically-trivialized} For $(V,\rho)$ an $F$-representation of $G$, the restriction of the vector bundle $\mathcal E(D_{b,\rho})$ to $\Spec(B_{dR}^+)$ via the map
\[
\Spec(B_{dR}^+)=\Spec(\hat{\cO}_{X,\infty})\longrightarrow X
\]
is \emph{canonically} trivialized (see also \cite[p. 278]{FF}). To show this, write $\{\infty\}=V_+(t)$ with $t\in B_K^{\varphi=\pi}$. Then $\cO_{X,\infty}$ is the homogeneous localization of the graded ring $P_K$ at the homogeneous prime ideal $(t)\subset P_K$. In other words, let 
\[
S:=\{b\in  B_K\setminus t\cdot B_K | \varphi(b)=\pi^i b\textrm{ for some }i\in \mathbb N\}, 
\]
then $\cO_{X,\infty}=(S^{-1}B_K)^{\varphi=1}$. Similarly 
\[
\mathcal E(D_{b,\rho})_{\infty}=\left(S^{-1}B_K\otimes_{\breve F}D_{b,\rho}\right)^{\varphi\otimes \varphi=1}.
\]
So it suffices to check that, for any isocrystal $D$ over $\breve F/F$, the natural morphism of $B_{dR}^+$-modules
\[
B_{dR}^+\otimes_{\cO_{X,\infty}} \left(S^{-1}B_K\otimes_{\breve F}D\right)^{\varphi\otimes\varphi =1}\longrightarrow B_{dR}^+\otimes_{\breve F}D
\]
is an isomorphism: one reduces to the case where $D$ is simple of slope $d/h$, with $d\in \mathbb Z,h\in \mathbb Z_{\geq 1}$ and $(h,d)=1$, and the result follows from the fact that the natural map below is an isomorphism:
\[
B_{dR}^+\otimes_{\cO_{X,\infty}} \left(S^{ -1}B_K\right)^{\varphi^h=\pi^d}\longrightarrow B_{dR}^{+,h}, \quad 1\otimes a\mapsto (a,\varphi(a),\ldots, \varphi^{h-1}(a)).
\]
As a corollary, the $G$-bundle $\mathcal E_b$ is \emph{canonically} trivialized over $\Spec(B_{dR}^+)$.

\if false
Using the slope decomposition of $D_{b,\rho}$, it is enough to show this when $D_{b,\rho}$ is isotypic of slope $-d/h$, with $d\in \mathbb Z,h\in \mathbb Z_{\geq 1}$ and $(d,h)=1$. So
\[
(D,\varphi)=(\breve{F}\otimes_{F_h} D^{\varphi^h=\pi^{-d}}, \sigma\otimes 1).
\]
Here $F_h/F$ denotes the unique unramified extension of degree $h$ in $\Fbar$. Write $D_0:=D^{\varphi^h=\pi^{-d}}$, which is a vector space over $F_h$ of dimension $\dim_{F}V$. It follows that the graded module defining $\mathcal E(D)$ can be written as
\[
\left(\bigoplus_i  B_{K}^{\varphi^h=\pi_F^{i+d}}\right)\otimes_{F_h}D_0.
\]
Thus $\mathcal E(D)|_{X\setminus \{\infty\}}$ is the coherent sheaf associated with
\[
B_K[1/t]^{\varphi^h=\pi_F^{d}}\otimes_{F_h}D_0.
\]
On the other hand, we have an isomorphism
\[
u:B_{dR}\otimes_{B_e}\left(B_K[1/t]^{\varphi^h=\pi_F^d}\right)\longrightarrow B_{dR}^h, \quad 1\otimes x\mapsto (x,\varphi(x),\ldots, \varphi^{h-1}(x)),
\]
and thus an isomorphism
\[
B_{dR}\otimes \mathcal E(D)_{\infty}\stackrel{\sim}{\longrightarrow} B_{dR}^h\otimes_{F_h}D_0=B_{dR}^h\otimes_{\breve F}D,
\]
here the structure of $\breve F$-algebra on $B_{dR}^h$ is given by
\[
\breve{F}\longrightarrow B_{dR}^h, \quad a\mapsto (a,\sigma(a),\ldots,\sigma^h(a)).
\]
\fi
\end{remark}

\subsection{Flag varieties and the weakly admissible locus}

Let $\{\mu\}$ be a geometric conjugacy class of a cocharacter
\[
\mu: \mathbb{G}_{m,\bar F}\longrightarrow G_{\bar F}.
\]
Let $E=E(G,\{\mu\})\subset \bar F$ be the field of definition of $\{\mu\}$: so $E/F$ is a finite subextension of $\bar F/F$ such that $\mathrm{Gal}(\bar F/E)$ is the stabilizer of
\[
\{\mu\}\in \Hom_{\bar F}(\mathbb G_{m,\bar F}, G_{\bar F})/G(\bar F)
\]
under the action of $\mathrm{Gal}(\bar F/F)$.


\subsubsection{Flag varieties} There is a projective variety $\mathcal P$ over $F$ parametrizing all parabolic subgroups of $G$, and the geometric connected component of $\mathcal P$ whose set of $\bar F$-points is given by the image of the map
\[
\{\mu\}\longrightarrow \mathcal P(\bar F), \quad \lambda\mapsto P_{\lambda}
\]
is defined over $E$. The corresponding open and closed subscheme of $\mathcal P_{E}$ will be written by
\[
\mathcal F(G,\mu).
\]
Here, for $\lambda:\mathbb G_{m,\bar F}\ra G_{\bar F}$ a cocharacter, $P_{\lambda}$ denotes its associated parabolic subgroup: for any field extension $L/\bar F$,
\[
P_{\lambda}(L)=\left\{g\in G(L)| \lim_{t\to 0}\lambda(t)g\lambda(t)^{-1}\textrm{ exists in }G_L\right\}.
\]
Recall that, for $\lambda':\mathbb G_{m,L}\ra G_L$ a second cocharacter of $G$ over $L$, $P_{\lambda}=P_{\lambda'}$ if and only if there exists $g\in P_{\lambda}(L)$ with $\lambda'=\mathrm{Int}(g)\circ \lambda$. Therefore, with a particular choice of $\mu\in \{\mu\}$, we have an identification
\[\begin{split}
G_{\bar F}/P_{\mu}&\longrightarrow \mathcal F(G,\mu)_{\bar F}, \\ gP_{\mu}&\longmapsto gP_{\mu}g^{-1}.\end{split}
\]
So $\mathcal F(G,\mu)$ is a twisted form of $G_{\bar F}/P_{\mu}$ over $E$.

\subsubsection{Weakly admissible locus in flag varieties} In the following, we shall consider $\mathcal F(G,\mu)$ as an adic space over $\breve{E}$, the completion of the composite $E\cdot F^{un}\subset \bar F$. Let $b\in G(\breve F)$. Rapoport and Zink defined a subspace, called weakly admissible locus
\[
\mathcal F(G,\mu,b)^{wa}\subset \mathcal F(G,\mu)
\]
attached to the triple $(G,\{\mu\},b)$, as a vast generalization of Drinfeld upper half plane. Let us briefly recall its definition. Let $L/\breve E$ be a complete field extension. A point $x\in \mathcal F(G,\mu)(L)$ corresponds to a parabolic subgroup of $G_L$, and thus of the form $P_{\mu_x}$ for some cocharacter $\mu_x:\mathbb{G}_{m,L}\ra G_L$ over $L$. For any finite-dimensional $F$-representation $(V,\rho)$ of $G$, we may consider the filtration on $V_L:=V\otimes_F L$ induced by
the cocharacter $\rho\circ \mu_x$, which does not depend on the auxiliary choice of $\mu_x$. We shall denote this filtration by $\Fil_{x,\rho}^{\bullet}$. Let
\[
\varphi-\mathrm{FilMod}_{L/\breve{F}}
\]
be the category of filtered isocrystals over $L/\breve{F}$. Consider the functor
\begin{equation}\label{eq:I-b-x}\begin{split}
\mathcal I_{b,x}: \Rep_F(G)&\longrightarrow \varphi-\mathrm{FilMod}_{L/\breve{F}}, \quad\\ (V,\rho)&\longmapsto (V_{\breve{F}}, \rho(b)(1\otimes \sigma), \Fil_{x,\rho}^{\bullet} ).\end{split}
\end{equation}

\begin{definition}[{\cite[1.18]{RZ}}] Let $L/\breve{E}$ be a complete field extension. Let $b\in G(\breve F)$ and $x\in \mathcal F(G,\mu)(L)$. We say that $x$ is \emph{weakly admissible} if the following equivalent conditions are satisfied:
\begin{enumerate}
\item for any object $(V,\rho)$ in $\Rep_F(G)$, the filtered isocrystal $\mathcal I_{b,x}(V,\rho)$ over $L/\breve{F}$ is weakly admissible in the sense of Fontaine;
\item there is a faithful finite dimensional $F$-representation $(V,\rho)$ of $G$ such that $\mathcal I_{b,x}(V,\rho)$ is weakly admissible.
\end{enumerate}
\end{definition}

It is a fundamental observation that the weakly admissible locus forms an open adic subspace of the flag variety:

\begin{theorem}[\cite{RZ}] There is a partially proper open subspace $
\mathcal F(G,\mu,b)^{wa}\subset \mathcal F(G,\mu)$
such that for any complete field extension $L/\breve E$,
\[
\mathcal F(G,\mu,b)^{wa}(L)=\{x\in \mathcal F(G,\mu)(L) | x \textrm{ is weakly admissible}\}.
\]
Moreover, $\mathcal F(G,\mu,b)^{wa}\neq \emptyset$ if and only if $
[b]\in A(G,\mu)$.
\end{theorem}

\begin{remark}The space $\mathcal F(G,\mu,b)^{wa}$ is obtained from $\mathcal F(G,\mu)$ by removing a family of Zariski closed subspaces of $\mathcal F(G,\mu)$. To see this, consider the reductive group $J_b$ over $F$, the $\sigma$-centralizer of $b$: so for any $F$-algebra $R$,
\[
J_b(R)=\{g\in G(\breve F\otimes_FR)|  gb\sigma(g)^{-1}=b\}.
\]
The group $J_b(F)$ acts on $\mathcal F(G,\mu)$ through the natural inclusion $J_b(F)\subset G(\breve F)\subset G(\breve{E})$. The weakly admissible locus $\mathcal F(G,\mu,b)^{wa}$ is stable under this action of $J_b(F)$, and is of the form
\[
\mathcal F(G,\mu)\setminus \bigcup_{i\in I}J_b(F)\cdot Z_i
\]
where $Z_i$, $i\in I$, is a finite collection of Zariski closed Schubert varieties.
\end{remark}

\subsection{Period domains and modifications of $G$-bundles}
Let $C/\bar F$ be a complete algebraically closed field extension, and consider the Fargues-Fontaine curve $X=X_{C^{\flat}}$, with $\infty\in X$ the associated closed point. So
\[
k(\infty)=C\quad \textrm{and}\quad \hat{\cO}_{X,\infty}=B_{dR}^+(C)=:B_{dR}^+.
\]
Let also $\mu:\mathbb{G}_{m,\Fbar}\ra G_{\Fbar}$ be a \emph{minuscule} cocharacter of $G$. Let $b\in G(\breve F)$. In this subsection, we shall construct  modifications of the $G$-bundle $\mathcal E_b$ from $C$-points of the period domain. As an application, we relate the weak admissibility to certain stability condition of $G$-bundles on the Fargues-Fontaine curve.

\begin{remark} The natural ring homomorphism $\theta:B_{dR}^+\longrightarrow C$ of Fontaine has a canonical section over $\Fbar$. Thus $B_{dR}^+$ and its fraction field $B_{dR}$ are naturally algebras over $\Fbar$. Therefore the cocharacter $\mu:\mathbb G_{m}\ra G$ yields a map $\mathbb{G}_m(B_{dR})\rightarrow G(B_{dR})$, and for $t$ a non-zero element of $B_{dR}^+$, we may consider its image $\mu(t)$ in $G(B_{dR})$.
\end{remark}

\subsubsection{Modifications of $G$-bundles}\label{sec:modification-of-G-bundles} Let $t\in B_{dR}^+$ be a uniformizer of the discrete valuation ring $B_{dR}^+$. Consider (the set of $C$-points of) the \emph{affine Schubert cell} $\mathrm{Gr}_{G,\mu}^{B_{dR}}(C)$ associated with $\mu$ (which does not depend on the choice of the uniformizer $t$ above) inside the $B_{dR}^+$-affine Grassmannian $\mathrm{Gr}_{G}^{B_{dR}}(C)$:
\[
\Gr_{G,\mu}^{B_{dR}}(C):=G(B_{dR}^+)\mu(t)^{-1}G(B_{dR}^+)/G(B_{dR}^+)\subset \Gr_G^{B_{dR}}(C):=G(B_{dR})/G(B_{dR}^+),
\]
and (the evaluation on $C$-points of) the \emph{Bialynicki-Birula map} (cf. \cite[3.4.3]{CS})
\[\begin{split}
\pi_{G,\mu}:\Gr_{G,\mu}^{B_{dR}}(C)&\longrightarrow \mathcal F(G,\mu)(C), \\ g\mu(t)^{-1}G(B_{dR}^+)&\longmapsto \theta(g)P_{\mu}\theta(g)^{-1}.\end{split}
\]

\begin{proposition}[{\cite[3.4.4]{CS}}]Assume that $\mu$ is minuscule. Then the map $\pi_{G,\mu}$ above is bijective.
\end{proposition}

In particular, for any $x\in \mathcal F(G,\mu)(C)$, one can use it to modify the $G$-bundle $\mathcal E_b$ as follows. Recall that the pullback of $\mathcal E_b$ along the natural map
\[
\Spec(B_{dR}^+)=\Spec(\hat{\cO}_{X,\infty})\longrightarrow X,
\]
can be \emph{canonically} trivialized (Remark \ref{rem:Eb-canonically-trivialized}). Let $\tilde x=gG(B_{dR}^+)\in \Gr_{G,\mu}^{B_{dR}}(C)$ be the unique element such that $\pi_{G,\mu}(\tilde x)=x$. We glue $\mathcal E|_{X\setminus \{\infty\}}$ and the trivial $G$-bundle $G\times_{\Spec(F)}\Spec(B_{dR}^+) $ over $\Spec(B_{dR}^+)$ through the isomorphism of $G$-bundles below:
\[
\left(\mathcal E|_{X\setminus \{\infty\}}\right)|_{\Spec(B_{dR})}\stackrel{\simeq}{\longrightarrow}G\times_{\Spec F}\Spec(B_{dR})\stackrel{g^{-1}\cdot }{\longrightarrow} G\times_{\Spec F}\Spec(B_{dR}).
\]
The resulting $G$-bundle on $X$ will be denoted by $\mathcal E_{b,x}$ in the sequel. 

\begin{remark}\label{rem:Ebx-depends-on-b} The isomorphism class of the $G$-bundle $\mathcal E_{b}$ depends only on the class $[b]\in B(G)$ of $b$. However, the construction of $\mathcal E_{b,x}$ does depend on $b$ as an element in $G(\breve F)$. For $b_1,g\in G(\breve F)$ with $b_1=gb\sigma(g)^{-1}$, the element $g$ induces an isomorphism $\mathcal E_{b,x}\stackrel{\sim}{\ra}\mathcal E_{b_1,gx}$. 
\end{remark}

\subsubsection{Weak admissibility and admissibility for $G$-bundles}
For $\mathcal E$ a $G$-bundle on $X$, and for $G'\subset G$ an algebraic subgroup of $G$, a \emph{reduction of $\mathcal E$ to $G'$} is a $G'$-bundle $\mathcal E'$ on $X$, together with an isomorphism of $G$-bundles
\[
\mathcal E'\times^{G'}G\stackrel{\sim}{\ra}\mathcal E.
\]
From the isomorphism above, we deduce a morphism $\mathcal E'\ra \mathcal E$ of $X$-schemes, compatible with the action of $G'$. The latter yields a section
\[
s: X\simeq \mathcal E'/G'\longrightarrow\mathcal E/G',
\]
of the map $\mathcal E/G'\ra X$ such that $\mathcal E'\stackrel{\sim}{\ra} X\times_{s,\mathcal E/G'}\mathcal E$. Conversely, every section of $\mathcal E/G'\ra X$ can be obtained in this way.

\begin{proposition}[\cite{CFS}, Lemma 2.4] Let $P\subset G$ be a parabolic subgroup of $G$. Let $\mathcal E$ and $\mathcal {\tilde{E}}$ be two $G$-bundles on $X$, equipped with an isomorphism $\mathcal E|_{U}\stackrel{\sim}{\ra}\tilde{\mathcal E}|_U$ of $G$-bundles over some non-empty open subset $U$ of $X$. Let $\mathcal E_P$ be a reduction of $\mathcal E$ to $P$. Then there is naturally a reduction $\mathcal {\tilde{E}}_P$ of $\mathcal {\tilde{E}}$ to $P$, together with an isomorphism $\mathcal E_P|_U\stackrel{\sim}{\ra}\mathcal {\tilde{E}}_P|_U$ of $P$-bundles satisfying the obvious compatibility condition.
\end{proposition}

One can reformulate the weak admissibility in terms of $G$-bundles on Fargues-Fontaine curves when $G$ is quasi-split. To see this, we need the following key definition:

\begin{definition}[\cite{CFS}, Definition 2.5] Let $b\in G(\breve F)$ be an element. For a Levi subgroup $M$ of $G$, a \emph{reduction of $b$ to $M$} is an element $b_M\in M(\breve F)$ together with an element $g\in G(\breve F)$ such that $b=gb_M\sigma(g)^{-1}$. By abuse of notation, such a reduction will simply be denoted by $b_M$ in the following if there is no confusion. 
\end{definition}

\begin{proposition}[\cite{CFS}, Proposition 2.7] \label{prop:criterion-for-wa} Assume $G$ quasi-split, and $\mu$ minuscule. Let $b\in G(\breve F)$ with $[b]\in A(G,\mu)$. Then $x\in \mathcal F(G,\mu)(C)$ is weakly admissible if and only if for any standard parabolic subgroup $P$ with associated Levi $M$, any reduction $b_M$ of $b$ to $M$, and any $\chi\in X^*(P/Z_G)^+$, we have
\[
\deg(\chi_*(\mathcal E_{b,x})_P)\leq 0,
\]
where $(\mathcal E_{b,x})|_P$ is the reduction to $P$ of $\mathcal E_{b,x}$ induced by the reduction $\mathcal E_{b_M}\times^MP$ of $\mathcal E_b$ to $P$.
\end{proposition}

As mentioned in the introduction, Rapoport and Zink has conjectured the existence of an open subspace $\mathcal F(G,\mu,b)^a\subset \mathcal F(G,\mu,b)^{wa}$ together with an \'etale $G$-local system $\mathcal L$ on $\mathcal{F}(G,\mu,b)^a$ such that these two spaces have the same classical points, that is, points with values in finite field extensions of $\breve E$, and that $\mathcal L$ interpolates the crystalline representations arising from classical points of $\mathcal F(G,\mu,b)^{wa}$. With the above construction, for $[b]\in B(G,\mu)$, we define $\mathcal F(G,\mu,b)^a$, called admissible locus, the subspace of $\mathcal F(G,\mu)$ such that, for any complete algebraically closed field extension $C$ of $\Fbar$
\[
\mathcal F(G,\mu,b)^{a}(C)=\{x\in \mathcal F(G,\mu)(C)| \mathcal E_{b,x} \textrm{ is trivial}\}
\]
The space $\mathcal F(G,\mu,b)^a$ satisfies the properties required by Rapoport-Zink conjecture follows from the work of Fargues-Fontaine, Kedlaya-Liu and Scholze.

\subsubsection{Newton and Harder-Narasimhan stratifications} \label{sec:Newton-and-HN}In the following, we review the definition of the Newton and the Harder-Narasimhan stratification on the flag variety. Recall that $b\in G(\breve F)$.

The Newton stratification on $\mathcal F(G,\mu)$ is indexed by elements in $B(G)$.
More precisely, let $C$ be a complete algebraically closed field extension of $\Fbar$. For each $x\in \mathcal F(G,\mu)(C)$, the $G$-bundle $\mathcal E_{b,x}$ corresponds to a unique element $[b_x']\in B(G)$ by Theorem \ref{thm:Fargues}. So we obtain a map $\mathcal F(G,\mu)(C)\ra B(G)$. Letting $C$ vary, we deduce a map
\begin{eqnarray}\label{eqn_Newton vector}
\mathrm{Newt}_b: |\mathcal F(G,\mu)|\longrightarrow B(G), \quad x\mapsto \mathcal [b_{x}'].
\end{eqnarray}
For $[b']\in B(G)$, the corresponding Newton strata $\mathcal F(G,\mu,b)^{[b']}$ is defined to be the preimage of $[b']\in B(G)$ via the map $\mathrm{New}_b$: so for any complete algebraically closed field extension $C$ of $\Fbar$,
\[
\mathcal F(G,\mu,b)^{[b']}(C)=\{x\in\mathcal F(G,\mu)(C)| \mathcal E_{b,x}\simeq \mathcal E_{b'}\}
\]
In particular, we have a decomposition
\begin{equation}\label{eq:Newton-stratification}
\mathcal F(G,\mu)=\coprod_{[b']\in B(G)} \mathcal F(G,\mu,b)^{[b']}
\end{equation}
of the flag variety $\mathcal F(G,\mu)$.

\begin{remark} Assume that $\mu$ is minuscule. Let $b,b'\in G(\breve F)$.
\begin{enumerate}
\item Each Newton stratum $\mathcal F(G,\mu,b)^{[b']}$ is locally closed in $\mathcal F(G,\mu)$ by the work of Kedlaya-Liu, and is stable under the action of $J_b(F)\subset G(\breve F)$ on $\mathcal F(G,\mu)$. Moreover, if $[b]\in B(G) $ is basic,
\[
\mathcal F(G,\mu,b)^{[b']}\neq \emptyset \quad \Longleftrightarrow \quad [b']\in B(G,\kappa([b])-\mu^{\#},\nu_b-\mu^{\diamond}).
\]
However, when $[b]\in B(G)$ is non-basic, such a description is still unknown.
\item  The decomposition \eqref{eq:Newton-stratification} is a stratification by the work of Viehmann \cite[Corollary 6.7]{Vi}.
\end{enumerate}
\end{remark}

Now we turn to the Harder-Narasimhan stratification on the flag variety. The formalism of Harder-Narasimhan stratification on the flag varieties (or on the $B_{dR}^+$-Grassmannian) was studied by Dat-Orlik-Rapopoport \cite{DOR}, Cornut-Peche Irissarry\cite{CPI}, Shen \cite{Sh2} and Nguyen-Viehmann\cite{NGVi}. We briefly recall the definition and main properties that we will need in the sequel. For $D=(D,\varphi,\Fil^{\bullet})$ a filtered isocrystal over $C/\breve{F}$. We call $\dim_{\breve F}(D)$ the \emph{rank} of $D$, written $\mathrm{rank}(D)$, and
\[
\deg(D):=\sum_i i\cdot \dim_{C}\mathrm{gr}^i_{\Fil^{\bullet}}(D_C)-v_{\pi}(\det(\varphi))
\]
its \emph{degree}. Using these two functions, one can develop a theory  of Harder-Narasimhan filtration on the category $\varphi-\mathrm{FilMod}_{C/\breve{F}}$ of filtered isocrystals over $C/\breve{F}$ (\cite[Proposition 8.1.10]{DOR}), and a remarkable property here is that the Harder-Narasimhan filtration for filtered isocrystals is compatible with tensor products (\cite[Theorem 8.1.9]{DOR}).

Recall that, for $b\in G(\breve F)$ and for $x\in \mathcal F(G,\mu)(C)$, we have the natural functor
\[
\mathcal I_{b,x}: \Rep_F(G)\longrightarrow \varphi-\mathrm{FilMod}_{C/\breve{F}}.
\]
of \eqref{eq:I-b-x}, or equivalently, a functor
\[
\mathcal I_b: \Rep_F(G)\longrightarrow \varphi-\mathrm{Mod}_{\breve F}
\]
together with a filtration $\Fil^{\bullet}_x$ on the fiber functor
\[
\omega^G_C:\Rep_F(G)\longrightarrow \mathrm{Vect}_{C}, \quad (V, \rho)\mapsto V_C:=V\otimes_F C.
\]
Using the formalism of Harder-Narasimhan filtration, we deduce a unique $\mathbb Q$-filtration on $\mathcal I_b$, such that for any $V=(V,\rho)\in \Rep_F(G)$, the induced filtration on the isocrystal $\mathcal I_{b}(V)$ over $\breve F$ is the Harder-Narasimhan filtration of the filtered isocrystal $\mathcal I_{b,x}(V)$ as explained in the above paragraph. In particular, this defines an element $v_{b,x}\in X_{*}(G)_{\mathbb Q}/G$.
It is known that $v_{b,x}$ is invariant under the action of $\Gamma=\mathrm{Gal}(\Fbar/F)$. So
\[
v_{b,x}\in \mathcal N(G)=\left(X_{*}(G)_{\mathbb Q}/G\right)^{\Gamma}.
\]
We call it the \emph{Harder-Narasimhan vector} of $x$. 

\begin{remark} Suppose $b=1\in G(\breve F)$. Recall that we are in the minuscule case, so we can identify the affine Schubert cell $\Gr_{G,\mu}^{B_{dR}}(C)$ with $\mathcal F(G,\mu)(C)$ via the Bialynicki-Birula map. Under this identification, the Harder-Narasimhan vector $v_{b,x}=v_{1,x}$ of $x\in \Gr_{G,\mu}^{B_{dR}}(C)=\mathcal F(G,\mu)(C)$ is denoted by $\mathrm{HN}(\mathcal E_1,x)$ by Nguyen-Viehmann in \cite{NGVi}. 

\end{remark}

Letting $C$ vary, we get the following map on topological spaces
\[
\mathrm{HN}_b: |\mathcal F(G,\mu)|\longrightarrow \mathcal N(G),\quad x\mapsto v_{b,x}^*:=-w_0v_{b,x}. 
\]
\begin{remark}We take the dual of the Harder-Narasimhan vector here is in order to compare it with $\mathrm{New}_b$ in (\ref{eqn_Newton vector}).
\end{remark}

\begin{proposition}[\cite{Sh2} Theorem 3.5 and Theorem 4.4]

The map $\mathrm{HN}_b$ above is upper-continuous. In other words, for any $v\in \mathcal N(G)$, the subset
\[
\mathcal F(G,\mu)^{\mathrm{HN}_b\geq v}:=\{x\in \mathcal F(G,\mu)| v\preceq   \mathrm{HN}_b(x)\}
\]
is closed. Consequently, the subset
\[
\mathcal F(G,\mu)^{\mathrm{HN}_b= v}:=\{x\in \mathcal F(G,\mu)| \mathrm{HN}_b(x)= v\}
\]
is locally closed, and we call it the \emph{Harder-Narasimhan stratum} associated with $v\in \mathcal N(G)$.
\end{proposition}

\begin{remark}Recall that $\mu$ is minuscule.
If $[b]\in B(G,\mu)$, then if $[b']\in B(G,0,\nu_b-\mu^{\diamond})$ is basic, then
\[
\mathcal F(G,\mu)^{\mathrm{HN}_b= [\nu_{b'}]}=\mathcal F(G,\mu,b)^{wa}
\]
is precisely the subspace of weakly admissible locus, thus is open and non-empty.
\end{remark}

In \cite{Sh2} and  \cite{NGVi}, Shen and Nguyen-Viehmann study the geometric properties of the Harder-Narasimhan stratification. For the convenience of the readers, we summarize below some of their results in the form that we will need in the sequel.

\begin{proposition}[\cite{Sh2}, \cite{NGVi}]\label{prop_NGVi}Suppose $[b]\in B(G,\mu)$ is basic and $x\in\mathcal{F}(G, \mu, b)^{[b']}(C)$. Then
\begin{enumerate}
\item  $\mathrm{HN}_b(x)\in B(G,0, \nu_b-\mu^{\diamond})$;
\item  $\mathrm{HN}_b(x)\preceq [\nu_{b'}]$. Moreover, if $(G, b', \nu_b-\mu^{\diamond})$ is Hodge-Newton decomposable with respect to some standard Levi subgroup $M$ inside the quasi-split inner form $H$ of $G$ (cf. definition \ref{def_HN decomp} below), then
\[
\mathrm{HN}_b(x)\preceq_M [\nu_{b'}],
\]
where $\preceq_M$ means that $[\nu_{b'}]-\mathrm{HN}_b(x)$ is a non-negative combination of coroots in $\Delta_{M, 0}^\vee$.
\end{enumerate}
\end{proposition}
\begin{proof} (1) and the first assertion in (2) are proved in \cite[Proposition 4.3]{Sh2}, and also in \cite[Proposition 3.13 and Theorem 6.4]{NGVi} (via inner twisting). The last assertion of (2) is stated for the modification of the trivial bundle $\E_1$ in \cite[Theorem 6.4]{NGVi}. Their results can be translated in the form mentioned above using the compatibility under inner twisting (cf. loc. cit. Section 8, or \cite[\S~4.5]{Sh2}).



\end{proof}

\section{Weakly fully HN-decomposability}\label{sec_weakly HN decomp}

In this section, we introduce the notion of \emph{weakly fully HN-decomposability} and discuss its minute criterion. We also give the complete classification of the weakly fully HN-decomposable pairs when the group $G$ is absolutely simple and adjoint. In the next section we shall use this group theoretic condition to see when the flag variety $\mathcal F(G,\mu)$ has maximal weakly admissible locus $\mathcal F(G,\mu,b)^{wa}$.

\subsection{Weakly fully HN-decomposability}\label{sec: weakly fully HN}
Recall that the reductive group $G$ is not necessarily quasi-split, and $H$ is the quasi-split inner form of $G$ over $F$. So we have natural identifications
\[
X_*(G)/G(\Fbar)=X_*(H)/H(\Fbar), \quad
\mathcal N(G)=\mathcal N(H)\quad \textrm{and}\quad \pi_1(G)=\pi_1(H).
\]
Moreover, as in Notations, let
\[
A\subset T\subset B
\]
be subgroups of $H$, where $A$ is a maximal split torus of $H$, $T=C_H(A)$ and $B$ a Borel subgroup of $H$ containing $T$. Let $(X^*(T),\Phi,X_*(T),\Phi^{\vee})$ (resp. $(X^*(A),\Phi_0,X_{*}(A),\Phi_0^{\vee})$) be the root datum given respectively by the adjoint action of $T$ (resp. of $A$), with $\Delta$ (resp. $\Delta_0$) a fixed basis. In particular, the Galois group $\Gamma$ acts on $\Phi$ and on $\Delta$, inducing a natural bijection
\[
\Phi/\Gamma\stackrel{\sim}{\longrightarrow}\Phi_0, \quad \Gamma\textrm{-orbit of }\alpha\mapsto \alpha|_{A}.
\]
Similarly, we have $\Delta/\Gamma\stackrel{\sim}{\rightarrow}\Delta_0$.

In order to discuss the concept of weakly fully Hodge-Newton decomposability for the non quasi-split reductive group $G$, we need to relate a generalized Kottwitz set for $G$ to a certain Kottwitz set of $H^{ad}$, the adjoint quotient of $H$. This part is explained in \cite[4.2]{CFS}. For the sake of completeness, let us reproduce their argument here. Since $H^{ad}$ is semisimple, we have natural identifications (see Remark \ref{rem:H1-and-pi})
\[
H^1(F, H^{ad})\stackrel{\sim}{\longrightarrow}B(H^{ad})_{basic}\stackrel{\sim}{\longrightarrow}\pi_1(H^{ad})_{\Gamma}=[\langle \Phi\rangle^{\vee}/\langle \Phi^{\vee}\rangle]_{\Gamma}.
\]
The group $G$, being an inner form of $H$, gives a class $[b_G]\in B(H^{ad})_{basic}$ and a class $[\xi]\in [\langle \Phi\rangle^{\vee}/\langle \Phi^{\vee}\rangle]_{\Gamma}$, with $b_G\in H^{ad}(\breve F)$ and $\xi\in \langle\Phi\rangle^{\vee}$. Moreover, $J_{b_G}=G^{ad}$ and there is a bijection
\begin{equation}\label{eq:BG-and-BH}
B(G^{ad})=B(J_{b_G})\longrightarrow B(H^{ad})
\end{equation}
that sends $[1]$ to $[b_G]$ (\cite[3.4]{Kot2}), which can be inserted into the following commutative diagram
\[
\xymatrix{B(G^{ad})\ar[r]^{\eqref{eq:BG-and-BH}}\ar[d]_{\kappa_{G^{ad}}} & B(H^{ad})\ar[d]^{\kappa_{H^{ad}}}\\ \pi_1(G^{ad})_{\Gamma}=\pi_1(H^{ad})_{\Gamma}\ar[r]^<<<<<{\bullet+\xi} & \pi_1(H^{ad})_{\Gamma}}
\]
As $[b_{G}]\in B(H^{ad})_{basic}$, its Newton vector is trivial. So the map \eqref{eq:BG-and-BH} gives the commutative diagram below
\[
\xymatrix{B(G^{ad})\ar[r]^{\eqref{eq:BG-and-BH}}\ar[d]_{\nu_{G^{ad}}} & B(H^{ad})\ar[d]^{\nu_{H^{ad}}}\\ \mathcal N(G^{ad})\ar@{=}[r] &\mathcal N(H^{ad})}
\]
On the other hand, for $\epsilon\in \pi_1(G)_{\Gamma}$ and $\delta\in \mathcal N(G)$, via the natural map $B(G)\ra B(G^{ad})$, the generalized Kottwitz set $B(G,\epsilon,\delta)$ can be identified with $B(G^{ad},\epsilon^{ad},\delta^{ad})\subset B(G^{ad})$, where $\epsilon^{ad}$ (resp. $\delta^{ad}$) denotes the natural image of $\epsilon$ in $\pi_1(G^{ad})_{\Gamma}$ (resp. in $\mathcal N(G^{ad})$ ): see \cite[4.11]{Kot2}. Thus, by the bijection \eqref{eq:BG-and-BH}, $B(G,\epsilon,\delta)$ can be further identified with
\[
B(H^{ad}, \epsilon^{ad}+\xi, \delta^{ad})\subset B(H^{ad}).
\]

Before giving the definition of the weakly fully HN-decomposability, let us review briefly the definition of the fully HN-decomposability  and its minute criterion.  The fully HN-decomposability condition was first introduced and systematically studied by (\cite{GoHeNi}).

\begin{definition}[full HN-decomposablity \cite{GoHeNi}]\label{def_HN decomp} Let $\{\mu\}\in X_*(G)/G(\Fbar)=X_*(H)/H(\Fbar)$ with $\mu \in X_{*}(T)^+$, $\epsilon \in \pi_1(G)_{\Gamma}=\pi_1(H)_{\Gamma}$, and $\delta\in X_*(A)_{\mathbb Q}^+=\mathcal N(H)=\mathcal N(G)$.
\begin{enumerate}
\item Suppose $[b]\in A(G, \mu)$ \emph(resp. $[b]\in B(G, \epsilon, \delta)$\emph), and view its Newton vector $[\nu_b]$ as an element in $X(A)_{\mathbb Q}^+$. We say that the triple $(G, \mu, b)$ \emph(resp. $(G, \delta, b)$\emph) is \emph{Hodge-Newton decomposable}, or \emph{HN-decomposable} for short, if there exists a standard proper Levi subgroup $M\subsetneq H$ such that:
\begin{itemize}
\item the centralizer of $[\nu_b]$ is contained in $M$; and
\item $\mu^{\diamond}-[\nu_b]\in \langle\Phi_{0, M}^\vee\rangle_{\Q}$ \emph(resp. $ \delta-[\nu_b]\in \langle\Phi_{0, M}^\vee\rangle_{\Q}$\emph).
\end{itemize}
Moreover, for a HN-decomposable triple $(G, \mu, b)$ (resp. $(G, \delta, b)$, we say that the triple is \emph{HN-decomposable with respect to a standard Levi subgroup $M$ in $H$} if $M$ is the unique minimal Levi subgroup satisfying the above conditions.
\item We say that the generalized Kottwitz set $B(G, \epsilon, \delta)$ is \emph{fully HN-decomposable} if for any non-basic $[b]\in B(G, \epsilon, \delta)$, the triple $(G, \delta, b)$ is HN-decomposable. We say that the pair $(G, \mu)$ is \emph{fully HN-decomposable} if so is $B(G, \mu)$.
\end{enumerate}
\end{definition}

\begin{remark}
\begin{enumerate}
    \item The generalized Kottwitz set $B(G,\epsilon,\delta)$ is fully HN-decomposable if and only if the corresponding generalized Kottwitz set $B(H^{ad},\epsilon^{ad}+\xi,\delta^{ad})$ for $H^{ad}$ is fully HN-decomposable.
    \item Suppose $G$ quasi-split. If $(G, \mu, b)$ (resp. $(G, \delta, b)$) is HN-decomposable with respect to $M$. Let $b_M$ be a reduction of $b$ to $M$, then $(M, \mu, b)$ (resp. $(M, \delta, b)$) is HN-indecomposable.
\end{enumerate}
\end{remark}
For $\beta\in \Delta$, let $w_{\beta}\in \langle \Phi\rangle_{\Q}$ be the corresponding fundamental weight. For $\alpha\in\Delta_0$, let
\[
\tilde{\omega}_{\alpha}=\sum_{\beta\in \Delta \ s.t. \ \beta|_A=\alpha}w_{\beta}\in X^*(T)_{\Q}^{\Gamma}=X^*(A)_{\Q}.
\]
In particular, for $\gamma\in \Delta$ we have
\[
\langle \gamma^{\vee,\diamond},\tilde{\omega}_{\alpha}\rangle=\left\{\begin{array}{ll}0, & \textrm{if }\gamma|_A\neq \alpha; \\ 1, & \textrm{otherwise.} \end{array}\right.
\]
Here $\gamma^{\vee,\diamond}$ denotes the Galois average of $\gamma^{\vee} $ in $X_*(T)$. In particular, for the element $\xi\in \langle \Phi\rangle ^{\vee}$ above, the fractional part
\[
\{\langle \xi^{\diamond}, \tilde{\omega}_\alpha\rangle\}\in [0, 1[
\]
of $\langle \xi^{\diamond}, \tilde{\omega}_\alpha\rangle$ only depends on the its $\langle\Phi^{\vee}\rangle$-coset in $\pi_1(H^{ad})=\langle \Phi\rangle^{\vee}/\langle \Phi^{\vee}\rangle$. 

\begin{proposition}[Minute criterion for the full HN-decomposability, {\cite[Theorem 3.3]{GoHeNi}}, {\cite[Proposition 4.12]{CFS}}]\label{prop:minute-for-full-HN} The set $B(G,\mu)$ is fully HN-decomposable if and only if for any $\alpha\in \Delta_0$, $\langle\mu^{\diamond}, \tilde{\omega}_{\alpha}\rangle+\{\langle\xi^{\diamond}, \tilde{\omega}_{\alpha}\rangle\}\leq 1$.
\end{proposition}


\if false
In order to introduce the concept of weakly fully Hodge-Newton decomposability for the non quasi-split reductive group $G$, we need to relate a generalized Kottwitz set for $G$ to a certain Kottwitz set of $H^{ad}$, the adjoint quotient of $H$. This part is explained in \cite[4.2]{CFS}. For the sake of completeness, let us reproduce their argument here. Since $H^{ad}$ is semisimple, we have natural identifications (see Remark \ref{rem:H1-and-pi})
\[
H^1(F, H^{ad})\stackrel{\sim}{\longrightarrow}B(H^{ad})_{basic}\stackrel{\sim}{\longrightarrow}\pi_1(H^{ad})_{\Gamma}=[\langle \Phi\rangle^{\vee}/\langle \Phi^{\vee}\rangle]_{\Gamma}.
\]
The group $G$, being an inner form of $H$, gives a class $[b_G]\in B(H^{ad})_{basic}$ and a class $[\xi]\in [\langle \Phi\rangle^{\vee}/\langle \Phi^{\vee}\rangle]_{\Gamma}$, with $b_G\in H^{ad}(\breve F)$ and $\xi\in \langle\Phi\rangle^{\vee}$. Moreover, $J_{b_G}=G^{ad}$ and there is a bijection
\begin{equation}\label{eq:BG-and-BH}
B(G^{ad})=B(J_{b_G})\longrightarrow B(H^{ad})
\end{equation}
that sends $[1]$ to $[b_G]$ (\cite[3.4]{Kot2}), which can be inserted into the following commutative diagram
\[
\xymatrix{B(G^{ad})\ar[r]^{\eqref{eq:BG-and-BH}}\ar[d]_{\kappa_{G^{ad}}} & B(H^{ad})\ar[d]^{\kappa_{H^{ad}}}\\ \pi_1(G^{ad})_{\Gamma}=\pi_1(H^{ad})_{\Gamma}\ar[r]^<<<<<{\bullet+\xi} & \pi_1(H^{ad})_{\Gamma}}
\]
As $[b_{G}]\in B(H^{ad})_{basic}$, its Newton vector is trivial. So the map \eqref{eq:BG-and-BH} gives the commutative diagram below
\[
\xymatrix{B(G^{ad})\ar[r]^{\eqref{eq:BG-and-BH}}\ar[d]_{\nu_{G^{ad}}} & B(H^{ad})\ar[d]^{\nu_{H^{ad}}}\\ \mathcal N(G^{ad})\ar@{=}[r] &\mathcal N(H^{ad})}
\]
On the other hand, for $\epsilon\in \pi_1(G)_{\Gamma}$ and $\delta\in \mathcal N(G)$, via the natural map $B(G)\ra B(G^{ad})$, the generalized Kottwitz set $B(G,\epsilon,\delta)$ can be identified with $B(G^{ad},\epsilon^{ad},\delta^{ad})\subset B(G^{ad})$, where $\epsilon^{ad}$ (resp. $\delta^{ad}$) denotes the natural image of $\epsilon$ in $\pi_1(G^{ad})_{\Gamma}$ (resp. in $\mathcal N(G^{ad})$ ): see \cite[4.11]{Kot2}. Thus, by the bijection \eqref{eq:BG-and-BH}, $B(G,\epsilon,\delta)$ can be further identified with
\[
B(H^{ad}, \epsilon^{ad}+\xi, \delta^{ad})\subset B(H^{ad}).
\]
\fi

Now we are ready to give the definition of the weakly fully HN-decomposability.

\begin{definition}[Weakly fully HN-decomposability]\label{def:weakly-full-HN}
\ 

\begin{enumerate}
\item Let $\epsilon \in \pi_1(G)_{\Gamma}$, and $\delta\in X_*(A)_{\mathbb Q}^+=\mathcal N(G)$.
We say that the generalized Kottwitz set $B(G, \epsilon, \delta)$ is \emph{weakly fully HN-decomposable} if for every non-basic $[b']\in B(G, \epsilon, \delta)$, either the triple $(G, \delta, b')$ is HN-decomposable or $[b]$ does not have reduction to $\mathrm{Cent}_{H}(\nu_{b'})$, where $[b]$ is the basic element in $B(G, \epsilon, \delta)$. 
\item  We say that the pair $(G, \mu)$ is \emph{weakly fully HN-decomposable} if the Kottwitz set $B(G,\mu)$ is weakly fully HN-decomposable.
\end{enumerate}
\end{definition}

\begin{remark}\begin{enumerate}
\item The generalized Kottwitz set $B(G, \epsilon, \delta)$ is weakly fully HN-decomposable if and only if so is $B(H^{ad}, \epsilon^{ad}+\xi, \delta^{ad})$.
\item If the generalized Kottwitz set $B(G, \epsilon, \delta)$ is full HN-decomposable, then it is weakly fully HN-decomposable.
\item By the uniqueness of superbasic Levi subgroup (cf. \cite{Nie} Lemma 1.5), $[b]$ has a reduction to $\mathrm{Cent}_H(\nu_{b'})$ if and only if $[b]$ has a reduction to all the maximal proper standard Levi subgroup of $H$ containing $\mathrm{Cent}_H(\nu_{b'})$.
\end{enumerate}
\end{remark}

Suppose that $G=H$ is quasi-split. In the following, we will give a criterion when $[b]$ has reduction to a maximal proper standard Levi subgroup. Recall that for $\alpha\in \Delta_0$, $M_{\alpha}$ is the standard Levi subgroup of $H$ with $\Delta_{M_{\alpha}, 0}=\Delta_0\backslash\{\alpha\}$.

\begin{lemma}\label{lemma_b has reduction to Levi}Suppose that $G=H$ is quasi-split. Let $\mu\in X_*(T)^+$ and $[b]\in B(G)_{basic}$ such that $\kappa_G(b)=\mu^{\#}$. For any $\alpha\in \Delta_0$, the element $[b]\in B(G)$ has a reduction to $M_{\alpha}$ if and only if $\langle\mu^{\diamond}, \tilde{\omega}_{\alpha}\rangle\in\Z$.
\end{lemma}
\begin{proof}We may assume that the identity component of the center $Z$ of $G$ is trivial. To see this, set $G'=G/Z^0$. As $Z^0$ is connected, we obtain from \cite[1.9]{Kot1} an exact sequence of pointed sets
\[
B(Z^0)\longrightarrow B(G)\longrightarrow B(G')\longrightarrow 0.
\]
Let $M_{\alpha}'=M_{\alpha}/Z^0$ be the corresponding standard Levi subgroup of $G'$. We have a similar exact sequence of pointed sets
\[
B(Z^0)\longrightarrow B(M_{\alpha})\longrightarrow B(M_{\alpha}')\longrightarrow 0.
\]
As a result, one checks that the square below is cartesian
\[
\xymatrix{B(G)\ar[r] & B(G')\\ B(M_{\alpha}) \ar[u]\ar[r]& B(M_{\alpha}')\ar[u]}.
\]
Therefore, $[b]\in B(G)$ has a reduction to $M_{\alpha}$ precisely when its image in $B(G')$ has a reduction to $M_{\alpha}'$. Hence, for the remaining part of the proof, we assume that the center of $G$ is finite. In particular, since $[b]\in B(G)_{basic}$, its Newton vector $\nu_b$ is trivial.

We next claim that $[b]$ has a reduction to $M_{\alpha}$ if and only if there exists a torsion element in $ \pi_1(M_{\alpha})_{\Gamma}$ which is mapped to $\kappa_G(b)$ via $\pi_1(M_{\alpha})_{\Gamma}\rightarrow\pi_1(G)_{\Gamma}$. If $b_{\alpha}\in B(M_{\alpha})$ is a reduction of $b$ to $M_{\alpha}$, the image of $[\nu_{b_{\alpha}}]\in \mathcal N(M_{\alpha})$ in $\mathcal N(G)$ is $[\nu_b]$, hence is trivial. Therefore, $\nu_{b_{\alpha}}$ itself is trivial, and $\kappa_{M_{\alpha}}(b_{\alpha})\in \pi_1(M_{\alpha})_{\Gamma}$ is a torsion element, which is sent to $\kappa_G(b)=\mu^{\#}$. Conversely, if there exists $a\in\pi_1(M_{\alpha})_{\Gamma,tor} $ in the preimage of $\kappa_G(b)=\mu^{\#}$ via $\pi_1(M_{\alpha})_{\Gamma}\ra \pi_1(G)_{\Gamma}$, it corresponds to a unique element in $H^1(F,M_{\alpha})$ via the composed map below (see Remark \ref{rem:H1-and-pi})
\[
H^1(F,M_{\alpha})\longrightarrow B(M_{\alpha})\stackrel{\kappa_{M_{\alpha}}}{\longrightarrow} \pi_1(M_{\alpha})_{\Gamma}.
\]
The latter gives rise to an element of $B(M_{\alpha})$, which is a reduction of $b$ to $M_{\alpha}$, proving our claim.

Now, write $\mu=\sum_{\beta\in\Delta}c_{\beta}\beta^{\vee}\in X_*(T)$ with $c_{\beta}\in\Q$. The preimage of $\mu^{\sharp}$ via $\pi_1(M_{\alpha})_{\Gamma}\rightarrow \pi_1(G)_{\Gamma}$ are the images in $\pi_1(M_{\alpha})_{\Gamma}$ of the elements in $X_*(T)$ of the form
\[
\sum_{\beta\in\Delta}c_{\beta}\beta^{\vee}+\sum_{\beta\in\Delta\atop \beta|_A=\alpha}\lambda_\beta \beta^{\vee},
\]
with $\lambda_{\beta}\in\Z$ for any $\beta\in \Delta$ such that $\beta|_A=\alpha$. On the other hand, the image in $\pi_1(M_{\alpha})_{\Gamma}$ of such an element is a torsion if and only if
\[
\sum_{\beta\in\Delta\atop \beta|_A=\alpha}(\lambda_{\beta}+c_{\beta})=0.
\]
Therefore, there exists a torsion element in $\pi_1(M_{\alpha})_{\Gamma}$ which is mapped to $\kappa_G(b)=\mu^{\sharp}$ via the natural map $\pi_1(M_{\alpha})_{\Gamma}\rightarrow \pi_1(G)_{\Gamma}$ if and only if
\[
\langle\mu^{\diamond}, \tilde{\omega}_{\alpha}\rangle=\sum_{\beta\in\Delta\atop \beta|_A=\alpha}c_{\beta}\in\Z,
\]
as required.
\end{proof}

\begin{lemma}\label{lemma_reduction b dual}Suppose that $G$ is quasi-split. Let $\mu\in X_*(T)^+$ and let $[b]\in B(G)_{basic}$ such that $\kappa_G(b)=\mu^{\#}$. Then for any $\alpha\in \Delta_0$, the element $[b]$ has a reduction to $M_{\alpha}$ if and only if $[b]$ has a reduction to $M_{\alpha^*}$, where $\alpha^{*}=-w_0\alpha$.
\end{lemma}
\begin{proof}By Lemma \ref{lemma_b has reduction to Levi}, it suffices to show
\[\langle \mu^{\diamond}, \tilde{\omega}_{\alpha}\rangle +\langle \mu^{\diamond}, \tilde{\omega}_{\alpha^*}\rangle\in\Z.\] Indeed, the involution $*$ is equal to  $-1$ on $\pi_1(G)_{\Gamma}$ since $(w_0\mu)^{\sharp}=\mu^{\sharp}$. As the fractional part of $\langle \mu^{\diamond}, \tilde{\omega}_{\alpha*}\rangle=\langle (-w_0\mu)^{\diamond}, \tilde{\omega}_{\alpha}\rangle$ only depends on the image $(-w_0\mu)^{\sharp}$, it follows that
\[
\langle \mu^{\diamond}, \tilde{\omega}_{\alpha}\rangle +\langle \mu^{\diamond}, \tilde{\omega}_{\alpha^*}\rangle
\in \langle \mu^{\diamond}, \tilde{\omega}_{\alpha}\rangle +\langle -\mu^{\diamond}, \tilde{\omega}_{\alpha}\rangle+\Z=\Z.
\]
This completes the proof of our lemma.
\end{proof}


For the weakly fully HN-decomposibility, we have a similar minute criterion as that for the full HN-decomposability (Proposition \ref{prop:minute-for-full-HN}). To see this, recall that, by Theorem \ref{thm:Kottwitz}, the generalized Kottwitz set $B(G,\epsilon,\delta)$ can be viewed as a subset of $\mathcal N(G)$ through the Newton map $\nu_G$. Hence we can describe the elements in $B(G,\epsilon,\delta)$ in terms of their Newton vectors in $X_*(A)_{\mathbb Q}^+=\mathcal N(G)$.

\begin{proposition}[\cite{CFS}, Proposition 4.6, Corollary 4.7 and Corollary 4.8]\label{prop:Kottwitz-set-in-NG}Let $\epsilon\in\pi_1(G)_{\Gamma}$ and $\delta\in X_*(A)_{\mathbb Q}^+$. Suppose $\epsilon=\mu_1^{\sharp}$ with $\mu_1\in X_*(T)^+$ not necessarily minuscule. Then as a subset of $\mathcal N(G)$, the generalized Kottwitz set $B(G,\epsilon, \delta)$ consists of the vectors $v\in X_*(A)_{\mathbb Q}^+$ such that
\begin{enumerate}
\item $\delta-v\in \langle \Phi^{\vee}_0\rangle_{\mathbb Q}$;
\item for all $\alpha\in \Delta_0$ with $\langle v,\alpha\rangle\neq 0$, one has $\langle \delta-v,\tilde{\omega}_{\alpha}\rangle \geq 0$, and $\langle \mu_1^{\diamond}+\xi^{\diamond}-v,\tilde{\omega}_{\alpha}\rangle\in \mathbb Z$.
\end{enumerate}

In particular, let $[b]\in B(G,\mu_1)$ be the basic element, then
\begin{enumerate}
\item the Kottwitz set $B(G,\mu_1)$ consists of the vectors $v\in X_*(A)_{\mathbb Q}^+$ such that
\begin{enumerate}
\item $\mu_1^{\diamond}-v\in \langle \Phi^{\vee}_0\rangle_{\mathbb Q}$;
\item for all $\alpha\in \Delta_0$ with $\langle v,\alpha\rangle\neq 0$, one has $\langle \mu_1^{\diamond}-v,\tilde{\omega}_{\alpha}\rangle \geq 0$, and $
\langle \mu_1^{\diamond}+\xi^{\diamond}-v,\tilde{\omega}_{\alpha}\rangle\in \mathbb Z$.
\end{enumerate}
\item the Kottwitz set $B(G,0,\nu_b-\mu_1^{\diamond})$ consists of $v\in X_*(A)^+_{\mathbb Q}$ such that
\begin{enumerate}
\item $v\in \langle \Phi_0^{\vee}\rangle_{\mathbb Q}$;
\item for all $\alpha\in \Delta_0$ with $\langle v,\alpha\rangle \neq 0$, one has $\langle \nu_b-w_0\mu_1^{\diamond}-v,\tilde{\omega}_{\alpha}\rangle \geq 0$ and $\langle v-\xi^{\diamond},\tilde{\omega}_{\alpha}\rangle \in \mathbb Z$.
\end{enumerate}
\end{enumerate}
\end{proposition}

As a direct application of the previous proposition, we get the following corollary.
\begin{corollary}\label{coro_involution_BG}Suppose $G=H$ is quasi-split. Then the bijection
\[
\begin{split}X_*(A)_{\mathbb Q}^+&\simeq X_*(A)_{\mathbb Q}^+\\
v&\mapsto v^*:=-w_0v\end{split}
\]
induces a bijection between generalized Kottwitz sets
\[
B(G, \epsilon, \delta)\simeq B(G, -\epsilon, -w_0\delta).
\]
\end{corollary}

\begin{remark}If $G$ is not quasi-split, then in general, $B(G,\epsilon, \delta)$ is not in bijection with $B(G, -\epsilon, -w_0\delta)$. For example, let $H=\mathrm{PGL}_6$, $\mu=(1^{(2)}, 0^{(4)})$ and $\xi=\mu^{\sharp}$. Then 
\[
\begin{split}B(G, \mu)&\simeq B(H ,\mu^{\sharp}+\xi^{\sharp}, \mu)\\
&=\left\{\left(\frac{2}{3}, \frac{4}{15}^{(5)}\right), \left(\frac{2}{3}, \frac{1}{3}^{(3)}, \frac{1}{6}^{(2)}\right), \left(\frac{2}{3}^{(2)}, \frac{1}{6}^{(4)}\right), \left(\frac{2}{3}^{(3)}, 0^{(3)}\right), \left(\frac{5}{12}^{(4)}, \frac{1}{6}^{(2)}\right), \left(\frac{1}{3}^{(6)}\right)\right \}\end{split}
\]
consists of $6$ elements, while
\[
\begin{split}B(G, -\mu)&\simeq B(H,0, -\mu)\\ &=\left\{\left(\frac{2}{3}^{(6)}\right), \left(1^{(3)}, \frac{1}{3}^{(3)}\right), \left(1^{(3)}, \frac{2}{3}, \frac{1}{6}^{(2)}\right), \left(\frac{11}{12}^{(4)}, \frac{1}{6}^{(2)}\right) \right\}\end{split}
\]
consists of $4$ elements.

\end{remark}

\begin{proposition}[Minute criterion for weakly fully HN-decomposability]\label{prop_minute criterion} Let $\mu\in X_*(T)^+$.
The pair $(G, \mu)$ is weakly fully HN-decomposable if for any $\alpha\in \Delta_0$ with $\langle\mu^{\diamond}+\xi^{\diamond}, \tilde{\omega}_{\alpha}\rangle\in\Z$, we have
\begin{equation}\label{eq:inequality-minute}
\langle\mu^{\diamond}, \tilde{\omega}_{\alpha}\rangle+\{\langle \xi^{\diamond}, \tilde{\omega}_{\alpha}\rangle\} \leq 1.
\end{equation}
\end{proposition}

\begin{proof} Without loss of generality, assume $G=G^{ad}$ and thus $H$ are adjoint. Recall that there is a natural identification between the generalized Kottwitz set $B(G,\mu)=B(G,\mu^{\#},\mu^{\diamond})$ for $G$, and the generalized Kottwitz set $B(H, \mu^{\#}+\xi^{\#}, \mu^{\diamond})$ for $H$. 

\underline{Necessity}. Suppose that $\langle\mu^{\diamond}, \tilde{\omega}_\alpha\rangle+\{\langle \xi^{\diamond}, \tilde{\omega}_\alpha\rangle\}>1$ is an integer for some $\alpha\in\Delta_0$. Let $v\in \langle \Phi^{\vee}_0\rangle_{\Q}=X_*(A)_{\mathbb Q}$ such that $\langle v, \alpha'\rangle=0$ for all $\alpha'\in\Delta_0\setminus \{\alpha\}$ and
\[
\langle v, \tilde{\omega}_{\alpha}\rangle= \langle\mu^{\diamond}, \tilde{\omega}_\alpha\rangle+\{\langle \xi^{\diamond}, \tilde{\omega}_\alpha\rangle\}-1\in \mathbb Z_{>0}.
\]
As $\tilde{\omega}_{\alpha}\in \mathbb Q_{\geq 0}\Delta_0$, $v$ is dominant. Moreover, $\langle \mu^{\diamond}-v,\omega_{\alpha}\rangle =1-\{\langle \xi^{\diamond}, \tilde{\omega}_\alpha\rangle\} \geq 0$, and
\[
\langle \mu^{\diamond}+\xi^{\diamond}-v,\tilde{\omega}_{\alpha}\rangle=1+ \langle \xi^{\diamond},\tilde{\omega}_{\alpha}\rangle-\{\langle \xi^{\diamond},\tilde{\omega}_{\alpha}\rangle\} \in \mathbb Z
\]
So by Proposition \ref{prop:Kottwitz-set-in-NG}, there exists  $[b']\in B(G,\mu)$ such that $\nu_{b'}=v$. Since $\langle v,\tilde{\omega}_{\alpha}\rangle =\langle\nu_{b'},\tilde{\omega}_{\alpha}\rangle\neq 0$, $v\neq 0$ and the element $[b']$ is not basic. Moreover, $M_{\alpha}$ is the centralizer of $v=\nu_{b'}$ in $H$, and
\[
\langle \mu^{\diamond}-v,\tilde{\omega}_{\alpha}\rangle =1-\{\langle \xi^{\diamond}, \tilde{\omega}_\alpha\rangle\}>0.
\]
Hence $\mu^{\diamond}-v\notin \langle \Phi^{\vee}_{0,M_{\alpha}}\rangle_{\mathbb Q}$, and $(G, \mu,b')$ is not HN-decomposable. Moreover, as \[
\langle \mu^{\diamond}+\xi^{\diamond},\tilde{\omega}_{\alpha}\rangle \in \mathbb Z,\]
according to Lemma \ref{lemma_b has reduction to Levi}, the basic element $[b]$ in $B(G,\mu)=B(H,\mu^{\#}+\xi,\mu^{\diamond})$ has a reduction to $\mathrm{Cent}_{H}(\nu_{b'})=M_{\alpha}$. This contradicts the weakly fully HN-decomposability condition.

\underline{Sufficiency}. Suppose that $[b']\in B(G,\mu)=B(H, \mu^{\#}+\xi, \mu^{\diamond})$ is a non-basic element which is HN-indecomposable with respect to $\mu$. We want to show that $[b]$ does not have reduction to the maximal standard Levi $M_{\alpha}\supset \mathrm{Cent}_{H}([\nu_{b'}])$ of $H$ for some $\alpha\in \Delta_0$, or equivalently, that for some $\langle \alpha, \nu_{b'}\rangle\neq 0$ we have $\langle \mu^{\diamond}+\xi^{\diamond}, \tilde{\omega}_{\alpha}\rangle \notin \Z$: see Lemma \ref{lemma_b has reduction to Levi}. Take $\alpha\in \Delta_0$ with $\langle \nu_{b'},\alpha\rangle \neq 0$, such that $\langle \mu^{\diamond}+\xi^{\diamond}, \tilde{\omega}_{\alpha}\rangle \in \Z$. Then, $M_{\alpha}$ contains the centralizer of $\nu_{b'}$ and $\langle \nu_{b'}, \tilde{\omega}_{\alpha}\rangle>0$. Moreover, the inequality \eqref{eq:inequality-minute} gives $\langle \mu^{\diamond},\tilde{\omega}_{\alpha}\rangle +\{\langle\xi^{\diamond}, \tilde{\omega}_{\alpha}\rangle\}\leq 1$.  As $[b']\in B(G,\mu)$, again by Proposition \ref{prop:Kottwitz-set-in-NG}, we find
\[
\underbrace{\{\langle \xi^{\diamond}, \tilde{\omega}_{\alpha}\rangle\}}_{\geq 0}+\underbrace{\langle \mu^{\diamond}, \tilde{\omega}_{\alpha}\rangle- \langle \nu_{b'}, \tilde{\omega}_{\alpha}\rangle}_{\geq 0}=\underbrace{\{\langle \xi^{\diamond}, \tilde{\omega}_{\alpha}\rangle\}+\langle \mu^{\diamond}, \tilde{\omega}_{\alpha}\rangle}_{\leq 1}- \underbrace{\langle \nu_{b'}, \tilde{\omega}_{\alpha}\rangle}_{> 0}\in\Z.
\]
This implies that $\langle \mu^{\diamond}-\nu_{b'}, \tilde{\omega}_{\alpha}\rangle= 0$ since $\nu_{b'}\preceq  \mu^{\diamond}$. In other words, $\mu^{\diamond}-\nu_{b'}\in \langle\Phi_{0,M_{\alpha}}^{\vee}\rangle_{\mathbb Q}$. This contradicts to the condition that $(G, \mu, b')$ is HN-indecomposable.
\end{proof}
\label{Miaofen: characterization of weakly accessiblity in term of minute criterion}

\begin{remark}\label{rem_minute criterion}Since $-w_0$ induce a bijection on $\Delta_0$ and $w_0\xi-\xi\in \langle \Phi^{\vee}\rangle$, the inequality \eqref{eq:inequality-minute} given by the minute criterion (Proposition \ref{prop_minute criterion}) is equivalent to the following: for $\alpha\in \Delta_0$ with $\langle-w_0\mu^{\diamond}, \tilde{\omega}_{\alpha}\rangle+\{\langle -\xi^{\diamond}, \tilde{\omega}_{\alpha}\rangle\}\in\Z$, we have
\[
\langle-w_0\mu^{\diamond}, \tilde{\omega}_{\alpha}\rangle+ \{\langle -\xi^{\diamond}, \tilde{\omega}_{\alpha}\rangle\}\leq 1.
\]
In particular, when $G$ is quasi-split, $(G,\mu)$ is weakly fully HN-decomposable if and only if
\[
\langle -w_0\mu^{\diamond},\tilde w_{\alpha}\rangle \notin \mathbb Z_{>1}, \quad \forall \ \alpha\in \Delta_0.
\]
So $(G,\mu)$ is weakly fully HN-decomposable if and only if this is the case for $(G,\mu^{-1})$.
\end{remark}

\begin{remark}\label{rem:example-of-weakly-full-HN} Let $G=\mathrm{GL}_n$, and $\mu$ a minuscule cocharacter of $G$. Replacing $\mu$ by $-w_0\mu$ if needed, we assume further that $\mu=(1^{(r)}, 0^{(n-r)})$. Then, using the minute criterion, it is easy to see that, the pair $(G, \mu)$ is weakly fully HN-decomposable if it is fully HN-decomposable, or is one of the following forms:
\begin{enumerate}
\item \emph(weakly accessible case, cf. Definition \ref{def:maximal-wa} below\emph), $\mu$ is central or $\mu=(1^{(r)},0^{(n-r)})$ with $\mathrm{gcd}(n, r)=1$;
\item $n$ even and $\mu=(1,1, 0,\cdots,0)$;
\item $n$ even and $\mu=(1,\cdots, 1, 0, 0)$;
\item $n=6$ and $\mu=(1,1,1,0,0,0)$.
\end{enumerate}
In the next subsection, we shall give the complete classification of weakly fully HN-decomposable pairs when the group $G$ is absolutely simple and adjoint. 
\end{remark}

\subsection{Classification of weakly fully HN-decomposable pairs}
Notations are the same as the last subsection. By definition, we know that the weakly fully HN-decomposability condition (Definition \ref{def:weakly-full-HN}) implies the fully HN-decomposable case (Definition \ref{def_HN decomp}). G\"ortz, He and Nie classify in \cite{GoHeNi} the fully HN-decomposable pairs. In this section, we classify the weakly fully HN-decomposable pairs $(G, \mu)$ when $G$ is absolutely simple and adjoint.

The main tool for us to do the classification is the minute criterion for weakly admissibility (Proposition \ref{prop_minute criterion}): the pair $(G, \mu)$ is weakly fully HN-decomposable if and only if for any $\alpha\in \Delta_0$ with $\langle\mu^{\diamond}+\xi^{\diamond}, \tilde{\omega}_{\alpha}\rangle\in\Z$, we have
\[
\langle\mu^{\diamond}, \tilde{\omega}_{\alpha}\rangle+ \{\langle \xi^{\diamond}, \tilde{\omega}_{\alpha}\rangle\}\leq 1.
\]

Suppose $G$ and hence $H$ are absolutely simple and adjoint. Note that in the minute criterion, whether $(G, \mu)$ is weakly fully HN-decomposable only depend on the quadruple $(H, |\mathrm{Im}\Gamma|, \mu, \xi)$, where $|\mathrm{Im}\Gamma|$ is the order of the image of the natural homomorphism $\Gamma\rightarrow \mathrm{Aut}(\Phi, \Delta)$ for $H$. So in order to classify the weakly fully HN-decomposable pairs $(G, \mu)$, it suffices classify the possible quadruples $(H, |\mathrm{Im}\Gamma|, \mu, \xi)$.

We will write the type of Dynkin diagram to represent $H$. Here the labeling of the Dynkin diagram are as in \cite{Bo}.
Moreover, as $H$ is adjoint, there is a bijection
\[\langle\Phi\rangle^\vee/\langle\Phi^\vee\rangle \simeq \{\text{minuscule dominant cocharacter of } T\}. \]Via this identification, we will take minuscule dominant cocharacter of $T$ as a representative of $\xi\in (\langle\Phi\rangle^\vee/\langle\Phi^\vee\rangle)_{\Gamma}$. We also write $w_0^\vee:=0$ for $\xi$ as a convention.

\begin{proposition}\label{prop_classification}Suppose $G$ is absolutely simple and adjoint. Let $\mu$ be a minuscule cocharacter and $\mu\neq 0$. Then $(G, \mu)$ is weakly fully HN-decomposable if and only if  $(G, \mu)$ is fully HN-decomposable or the associated quadruple $(H, |\mathrm{Im}\Gamma|, \mu, \xi)$ is one of the following up to isomorphism:
 \begin{center}
  \begin{tabular}{|c|c|c|c|}
\hline\ \ \  $H$\ \ \  & \  $|\mathrm{Im}\Gamma|$\  &\ \ \  $\mu$ \ \ \ & $\xi$\\
\hline $A_n$ & 1 & $\omega_i^\vee$ &$\omega_{i'}^\vee$, s.t. $\gcd(i+i', n+1)=1$ \\
\hline $A_n$ &1   &$\omega_1^\vee$ & arbitrary\\
\hline $A_n$ &1  &$\omega_2^\vee$  & $\omega_{i'}^\vee$, s.t. $\gcd(2+i', n+1)=2$\\
\hline $A_5$ &1  &$\omega_3^{\vee}$  & 0\\
\hline $A_3$ &2 &$\omega_2^\vee$ &$\omega_1^\vee$\\
\hline $C_3$ &1 &$\omega_3^\vee$ & $0$\\
\hline $D_n$ &1 &$\omega_1^\vee$ & $\omega_n^\vee$\\
\hline $D_5$ &1 &$\omega_5^\vee$  &$0$, $\omega_1^\vee$\\
\hline $D_n$ &2 &$\omega_1^\vee$ & $\omega_n^\vee$\\
\hline $D_4$ &2 &$\omega_4^\vee$ & $0$ \\
\hline
\end{tabular}
\end{center}
Moreover, $(G, \mu)$ is weakly accessible if and only if the quadruple is  the first row in the above table.
\end{proposition}

As a corollary of this proposition, we can list all the weakly fully HN-decomposable pairs using the notation of Tits' table (\cite{Ti}).
\begin{theorem}\label{thm:classification-weak-HN-dec}Suppose $G$ is absolutely simple and adjoint. Let $\mu$ be a minusucle cocharacter and $\mu\neq 0$. Then $(G, \mu)$ is weakly fully HN-decomposable if and only if  $(G, \mu)$ is fully HN-decomposable or weakly accessible or $(G, \mu)$ is one of the following up to isomorphism:
\begin{tabular}{|c|c|c|c|}
\hline $( ^dA_n, \omega_1^\vee)$ with $d|n+1$     &$( ^dA_n, \omega_2^\vee)$ with $n$ odd, $d|n+1$  (*) & $(A_5, \omega_3^\vee)$ & $( ^4A_3, \omega_2^\vee)$\\
\hline $( ^2C_n, \omega_1^\vee)$ & $(^2C-B_n, \omega_1^{\vee})$ & $(C_3, \omega_2^\vee)$  &   \\
\hline $(^2 D''_n, \omega_1^{\vee})$ & $(D_5, \omega_5^\vee)$ & $(^2D'_5, \omega_5^{\vee})$ & $( ^2D_4, \omega_4^\vee)$ \\
\hline
\end{tabular}

where (*) after $( ^dA_n, \omega_2^\vee)$ means that not all groups of type $^dA_n$ are allowed, but only the ones with Frobenius acting on the $n+1$-cycle by a clockwise rotation of $i'$-steps with $\gcd(2+i', n+1)=2$ are allowed.
\end{theorem}

\begin{remark}In \cite{GoHeNi}, G\"ortz, He and Nie give characterizations of basic affine Delgine-Lusztig varieties associated to $(G, \mu)$ when it is fully HN-decomposable. For example, under this assumption, the affine Deligne-Lusztig variety is a union of Deligne-Lusztig varieties. It will also be an interesting  question to investigate the basic affine Delgine-Lusztig varieties associated to a weakly fully HN-decomposable pair.  For the case $(C_3, \omega_2^\vee)$, it is studied in \cite{Ri}. In a joint work with Viehmann in preparation, we study the case $(A_n, \omega_2^\vee)$.
\end{remark}

\begin{proof}[Proof of Proposition \ref{prop_classification}]
Note that by minute criterion (Proposition \ref{prop_minute criterion}), the condition that $(H, \mu)$ is fully HN-decomposable implies that $(G, \mu)$ is weakly fully HN-decomposable for any inner form $G$ of $H$.  In the following, we only consider $(H, \mu)$ which is NOT fully HN-decomposable. We use the same notation as in Bourbaki \cite{Bo}. We will discuss by the Dynkin diagram of $H$. Sometimes for simplicity,  we will also write $\tilde{\omega}_i$ for $\tilde{\omega}_{\alpha_i}=\omega_{i}$ when $H$ is split. 

\emph{Case $A_n$}: Note that for $1\leq i, j\leq n$,
\begin{eqnarray}\label{eqn_computat An}\langle \omega_i^\vee, \omega_j\rangle=\min\{i, j\}-\frac{ij}{n+1}.\end{eqnarray} Consider
\[\theta_i(j):=\begin{cases} j-\frac{ij}{n+1},  &j\leq i\\ i-\frac{ij}{n+1},  &j>i,\end{cases}\] as a function on $\mathbb{R}\cap [0, n+1]$. Then $\theta_i(j)=\langle \omega_i^\vee, \omega_j\rangle$ for $1\leq i, j\leq n$. Moreover, as a function on $j$, $\theta_i(j)$ is strictly increasing for $j\leq i$ and strictly decreasing for $j\geq i$.

\emph{Subcase $^1A_n$}: $|\mathrm{Im}\Gamma|=1$.

By (\ref{eqn_computat An}),
\[\langle \omega_i^\vee, \tilde{\omega}_j\rangle+\{\langle \omega_{i'}^\vee, \tilde{\omega}_j\rangle\}\equiv -\frac{(i+i')j}{n+1} \mod \mathbb{Z}.\] It follows that
\[\langle \omega_i^\vee, \tilde{\omega}_j\rangle+\{\langle \omega_{i'}^\vee, \tilde{\omega}_j\rangle\}\in\Z \Longleftrightarrow n+1| (i+i')j.\] Suppose $(A_n, 1, \omega_i^\vee, \omega_{i'}^\vee)$ is weakly fully HN-decomposable.  By minute criterion (Proposition \ref{prop_minute criterion}),  it's equivalent to say
\begin{eqnarray}\label{eqn_A_1}\forall 1\leq j\leq n, \frac{(i+i')j}{n+1}\in\Z\Longrightarrow \theta_i(j)=\langle \omega_i^\vee, \tilde{\omega}_j\rangle\leq 1.\end{eqnarray} Obviously, this condition is satisfied if $\gcd(i+i', n+1)=1$.  Now assume $d:=\gcd(i+i', n+1)>1$.
By symmetry, we may further assume $i\leq \frac{n+1}{2}$. Moreover $(A_n, \omega_1^\vee)$ is fully HN-decomposable by \cite{GoHeNi}. It remains to consider  $2\leq i\leq \frac{n+1}{2}$.

\emph{Claim:  $i=3$ with $n$=5 or $i=2$.}

If $i\geq 3$, then
\[\begin{split}\theta_i(\frac{n+1}{2})=\frac{i}{2}&>1 \\
\theta_i(2)=2-\frac{2i}{n+1}&\geq 1 \end{split}\]
where the second inequality becomes an equality if and only if $i=\frac{n+1}{2}$. The weakly fully HN-decomposability condition (\ref{eqn_A_1}) for $(A_n, 1, \omega_i^\vee, \omega_{i'}^\vee)$ is equivalent to
\begin{eqnarray}\label{eqn_A_2}\theta_i(\frac{n+1}{d}r)\leq 1,\  \forall 1\leq r\leq d-1.
\end{eqnarray}
If $i\neq \frac{n+1}{2}$, then $\theta_i(j)>1$ for all $j\in [2, \frac{n+1}{2}]$. On the other hand, note that
\[\underbrace{\{\frac{n+1}{d}r| r\in\N\cap [1,d-1]\}}_{\text{value of }\theta_i\leq 1 \text{ by } (\ref{eqn_A_2})}\cap \underbrace{[2, \frac{n+1}{2}]}_{\text{value of } \theta_i>1}\neq \emptyset\]
which leads to a contradiction.

If $i=\frac{n+1}{2}$, then $\theta_i(j)>1$ for all $j\in]2, \frac{n+1}{2}]$. Again by (\ref{eqn_A_2}),
\[\{\frac{n+1}{d}r| r\in\N\cap [1,d-1]\}\cap [2, \frac{n+1}{2}]=\{2\}\]
which implies that $\frac{n+1}{d}=2$ and $4>\frac{n+1}{2}$. Hence $n=5$ and $i=3$ as we ignore the fully HN-decomposable pairs. Now the Claim follows.

For $i=2$, $\theta_2(j)=2-\frac{2j}{n+1}$ for $j\geq 2$. By (\ref{eqn_A_2}), it follows that
\[j\in \{\frac{n+1}{d}r|r\in\N\cap [1,d-1]\}\Longrightarrow \theta_2(j)\leq 1\Longleftrightarrow j\geq \frac{n+1}{2}\text{ or }j=1.\] This implies $d=2$ or $n=3$.

For $i=3$ with $n=5$,  $\theta_3(j)\leq 1\Leftrightarrow j\neq 3$. Hence (\ref{eqn_A_2}) implies that $\frac{n+1}{d}\nmid 3$. It follows that $d=3$.

\emph{Subcase $^2A_n$}: $|\mathrm{Im}\Gamma|=2$.
If $n$ is even, then $\langle\mu, \tilde{\omega}_{\alpha}\rangle\in\Z$ for any $\mu$ and any $\alpha\in\Delta_0$. Hence $(G, \mu)$ is weakly fully HN-decomposable if and only if it's fully HN-decomposable.

If $n=2m-1$ is odd. Note that \[\langle \omega_i^{\vee}, \tilde{\omega}_{\alpha_j}\rangle\in\Z, \ \forall 1\leq j\leq m-1, \forall 1\leq i\leq m\] where
$\tilde{\omega}_{\alpha_j}=\omega_j+\omega_{n-j}=(1^{(j)}, 0^{(2m-2j)}, -1^{(j)})$.
Suppose $(A_n, 2, \omega_i^\vee, \omega_{i'}^\vee)$ is weakly fully HN-decomposable. It follows that
\[\langle \omega_i^\vee, \tilde{\omega}_{\alpha_j}\rangle \leq 1,\  \forall 1\leq j\leq m-1.\] It follows that $m=2$ if $i=2$. We can easily verify $(A_3, 2, \omega_2^\vee, \omega_1^\vee)$ is the only new case up to isomorphism.

\emph{Case $B_n$}: $\mu=\omega_1^{\vee}$ is the only non-trivial minuscule cocharacter. We can check that the weakly fully HN-decomposable pairs are all fully HN-decomposable.

\emph{Case $C_n$}: $\mu=\omega_n^{\vee}$ is the only non-trivial minuscule cocharacter. Note that \[\langle \omega_n^\vee, \tilde{\omega}_j\rangle =\frac{j}{2}, \forall 1\leq j\leq n.\] So $(C_3, 1, \omega_3^\vee, 0)$ is the only new case.

\emph{Case $D_n$}: $\omega_{1}^\vee$, $\omega_{n-1}^\vee$ and $\omega_n^{\vee}$ are the non-trivial minuscule cocharacters.
Note that
\[\langle\omega_{1}^{\vee}, \tilde{\omega}_j\rangle=\begin{cases}1, &j\leq n-2\\ \frac{1}{2}, &j=n-1, n\end{cases},\ \ \   \langle \omega_{n}^\vee, \tilde{\omega}_j\rangle=\begin{cases} \frac{j}{2},  & j\leq n-2\\ \frac{n-2}{4}, &j=n-1\\ \frac{n}{4}, &j=n. \end{cases}  \]

\emph{Subcase $^1D_n$}: $|\mathrm{Im}(\Gamma)|=1$. If $\mu=\omega_1^{\vee}$, then $\xi$ can be taken arbitrarily. If $n=5$, we can also take $\mu=\omega_n^\vee$ with $\xi=0$ or $\omega_1^{\vee}$.

\emph{Subcase $^2D_n$}: $|\mathrm{Im}(\Gamma)|=2$. Note that
\[\langle \omega_1^{\vee}, \tilde{\omega}_{\alpha_j}\rangle=1,  \langle \omega_n^{\vee}, \tilde{\omega}_{\alpha_j}\rangle=\frac{j}{2}, \ \forall 1\leq j\leq n-1.\] It follows that if $\mu=\omega_1^{\vee}$, then $\xi$ can be chosen arbitrarily. If $\mu=\omega_n^{\vee}$, then $n=4$ and $\xi=0$ or $\omega_1^{\vee}$ (which in fact have the same image in $(\langle \Phi\rangle^\vee/\langle\Phi^\vee\rangle)_{\Gamma}$)

\emph{Subcase $^3D_4$}: $|\mathrm{Im}(\Gamma)|=3$.
\[ \langle \omega_1^{\vee}, \tilde{\omega}_{\alpha_i}\rangle=\begin{cases}2, & i=1,\\ 1 &i=2. \end{cases}\]No new case is possible.

\emph{Case $E_6$, $E_7$}: Same computation shows there are NO new cases.

\emph{Case $E_8$, $F_4$, $G_2$}: There is no non-trivial minuscule cocharacter.

The last assertion of the proposition follows directly from \cite[A.13]{Ra2}.

\end{proof}

\if false

\begin{proposition}\label{prop_weakly HN-decomp}The following three conditions are equivalent:
\begin{enumerate}
\item $(G,\mu)$ is weakly fully HN-decomposable; and
\item $B(G, 0, \nu_b\mu^{-1})$ is weakly fully HN-decomposable.
\end{enumerate}
\end{proposition}
\begin{proof} By the minute criterion for weakly fully HN-decomposability above, the statement (1) is equivalent to the condition that, for any $\alpha\in \Delta_0$ with $\langle\mu^{\diamond}+\xi^{\diamond}, \tilde{\omega}_{\alpha}\rangle\in\Z$, we have
\[
\langle\mu^{\diamond}, \tilde{\omega}_{\alpha}\rangle+\{\langle \xi^{\diamond}, \tilde{\omega}_{\alpha}\rangle\}\leq 1.
\]
On the other hand, since $-w_0$ induce a bijection on $\Delta_0$ and $w_0\xi-\xi\in \langle \Phi^{\vee}\rangle$, the condition above given by the minute criterion (Proposition \ref{prop_minute criterion}) is equivalent to the following: for $\alpha\in \Delta_0$ with $\langle-w_0\mu^{\diamond}, \tilde{\omega}_{\alpha}\rangle+\{\langle -\xi^{\diamond}, \tilde{\omega}_{\alpha}\rangle\}\in\Z$, we have
\[
\langle-w_0\mu^{\diamond}, \tilde{\omega}_{\alpha}\rangle+ \{\langle -\xi^{\diamond}, \tilde{\omega}_{\alpha}\rangle\}\leq 1.
\]
So we only need to check the equivalence between this last condition and (2). The proof of this equivalence is similar to that of Proposition \ref{prop_minute criterion}. See also the proof of \cite[Proposition 4.14]{CFS}.\end{proof}\mar{Miaofen: check details}
\fi

\section{Maximal weakly admissible locus}\label{Sec_maximal wa}
In this section, we address the question when the weakly admissible locus $\mathcal F(G,\mu,b)^{wa}$ is maximal (in the sense that it is a union of Newton strata, see Definition \ref{def:maximal-wa}). We shall show in Theorem \ref{thm_main} that the previous condition is fulfilled if and only if the pair $(G,\mu)$ is weakly fully HN-decomposable, which is further equivalent to the fact that the Newton stratification is finer than the Harder-Narasimhan stratification.

Recall that there is a Newton decomposition for the flag variety $\mathcal F(G,\mu)$:
\[
\Fc(G, \mu)=\coprod_{[b']\in B(G, 0, \nu_b-\mu^{\diamond})}\Fc(G, \mu, b)^{[b']}
\]
where $\mathcal F(G,\mu,b)^{[b']}$ is a locally closed adic subspace of $\Fc(G, \mu)$, such that
\[
\Fc(G, \mu, b)^{[b']}(C):=\{x\in \Fc(G, \mu)(C)| \E_{b, x}\simeq \E_{b'}\}
\]
for any complete algebraically closed field extension $C/\Fbar$. Moreover, $\Fc(G,\mu,b)^{[b']}$ defines a locally spatial subdiamond of $\Fc(G, \mu)^{\diamond}$.

\begin{proposition}\label{prop_Newton and wa} Let $\{\mu\}$ be the geometric conjugacy class of a minuscule cocharacter $\mu$, and $[b]\in A(G,\mu)$ which is basic. Suppose $[b']\in B(G, 0, \nu_b-\mu^{\diamond})$.\begin{enumerate}
\item \emph(\cite[Theorem 5.1]{CFS}\emph) If $[b']$ is HN-decomposable with respect to $\nu_b-\mu^{\diamond}$,
then \[\Fc(G, \mu, b)^{wa}\cap \Fc(G, \mu, b)^{[b']}=\emptyset;\]

\item \emph(\cite[Theorem 1.3]{Vi}\emph) If $[b']$ is HN-indecomposable with respect to $\nu_b-\mu^{\diamond}$,
then 
\[
\Fc(G, \mu, b)^{wa}\cap \Fc(G, \mu, b)^{[b']}\neq \emptyset.
\]
\end{enumerate}
\end{proposition}

As a corollary, we deduce
\begin{equation}\label{eq:maximal-wa}
\mathcal F(G,\mu,b)^{wa}\subseteq \coprod_{\substack{ [b']\in B(G, 0, \nu_b-\mu^{\diamond})\textrm{ such that } \\ (b', \nu_b-\mu^{\diamond}) \textrm{ is HN-indecomposable} } } \mathcal F(G,\mu,b)^{[b']} \subseteq \mathcal F(G,\mu).
\end{equation}

\begin{definition}\label{def:maximal-wa} Let $\{\mu\}$ a geometric conjugacy class of a cocharacter of $G$. Let $[b]\in A(G,\mu)$, so that the period domain $\mathcal F(G,\mu,b)^{wa}$ is non-empty.
\begin{enumerate}
\item \emph(\cite[Definition A1]{Ra2}\emph) We say that the triple $(G, b, \mu)$ is \emph{weakly accessible} if
\[
\Fc(G, \mu, b)^{wa}=\Fc(G, \mu).
\]
We say that $(G, \mu)$ is \emph{weakly accessible} if this is the case for $(G, b', \mu)$, where $[b']\in B(G,\mu)$ is the unique basic element.
\item We say that the weakly admissible locus $\Fc(G, \mu, b)^{wa}$ is \emph{maximal} if the first inclusion in \eqref{eq:maximal-wa} is an equality, or equivalently,
\[
\Fc(G, \mu, b)^{wa}(C)=\coprod_{\substack{ [b']\in B(G, 0, \nu_b-\mu^{\diamond})\textrm{ s.t. } \\ (b', \nu_b-\mu^{\diamond}) \textrm{ is HN-indecomposable} } }\Fc(G, \mu, b)^{[b']}(C)
\]
for any complete algebraically field extension $C/\Fbar$.
\end{enumerate}
\end{definition}

If $G=\mathrm{GL}_n$, the weakly accessible pairs $(G,\mu)$ are given as in Remark \ref{rem:example-of-weakly-full-HN} (1). In the general case,  after some reduction steps (\cite[Lemma A.11]{Ra2}), the following proposition gives the complete classification of all weakly accessible triples $(G,\mu,b)$.

\begin{proposition}[\cite{Ra2} Proposition A.12]Suppose that $(G, \mu, b)$ defines a non-empty period domain where $G$ is $F$-simple and $\mu$ is non-trivial. Then $(G, \mu, b)$ is weakly accessible if and only if the $F$-group $J_b$ is anisotropic, in which case $b$ is basic.
\end{proposition}


\begin{remark}Clearly, the fully HN-decomposability or the weak accessibility implies the maximality of the weakly admissible locus.
\end{remark}



The following theorem gives several characterizations for the maximality of the weakly admissible locus $\mathcal F(G,\mu,b)^{wa}$, which is the first main result of our paper.

\begin{theorem}\label{thm_main}Suppose that $b\in G(\breve F)$ is basic and that $\mu\in X_*(T)^{+}$ is minuscule. Then the following assertions are equivalent:
\begin{enumerate}
\item $\Fc(G, \mu, b)^{wa}$ is maximal;

\item $(G,\mu)$ is weakly fully HN-decomposable;

\item the Newton stratification is finer than the Harder-Narasimhan stratification in the sense that every Harder-Narasimhan stratum is union of some Newton strata. \end{enumerate}

\end{theorem}

\begin{proof}We may assume $G$ and thus $H$ adjoint. Indeed, this follows from the observation that, if we denote by
\[
\pi:\mathcal F(G,\mu)\longrightarrow \mathcal F(G^{ad},\mu^{ad})
\]
the natural map, then the preimage of a Newton stratum (resp. Harder-Narasimhan stratification) for the triple $(G^{ad}, \mu^{ad}, b^{ad})$ is the corresponding stratum for the triple $(G, \mu, b)$. Therefore, for the rest of the proof, $G$ is supposed to be adjoint. In particular, $\nu_b=0$. Furthermore, we suppose $\mu\neq 0$: otherwise our theorem is trivially true.

Recall that we have the identification \eqref{eq:BG-and-BH} between $B(G)$ and $B(H)$ induced by the class
\[
\xi\in H^1(F,H)\simeq \pi_1(H)_{\Gamma, tor}
\]
of $G$ as an inner form of $H$, under which $B(G,\mu)$ is identified with $B(H,\mu^{\#}+\xi,\mu^{\diamond})$. Moreover, using a fixed inner twisting
\[
G_{\breve F}\stackrel{\sim}{\longrightarrow} H_{\breve F},
\]
we have an identification
\[
\mathcal F(G,\mu)\stackrel{\sim}{\longrightarrow}\mathcal F(H,\mu)
\]
of flag varieties over $\breve E$, under which the Newton stratification (resp. the Harder-Narasimhan stratification) on $\mathcal F(G,\mu)$ for the triple $(G, \mu, b)$ is identified with the corresponding stratification on $\mathcal F(H,\mu)$ for the triple $(H, \mu, b^H)$, where $[b^H]\in B(H)$ is the image of $[b]$ via $B(G)\simeq B(H)$. In the following,  we will deal with $(H, \mu, b^H)$ instead of $(G, \mu, ,b)$.

Moreover, we may assume that $H$ is simple. Indeed, if $H=H_1\times H_2$ with $\mu=(\mu_1, \mu_2)$ and $b^H=(b_1^H, b_2^H)$. Then we have a natural isomorphism
\[\mathcal{F}(H, \mu, b^{H})\simeq \mathcal{F}(H_1, \mu_1, b_1^{H})\times \mathcal{F}(H_2, \mu_2, b_2^{H}),\]under which the Newton stratification and Harder-Narasimhan stratification on both sides are compatible.

\underline{$(1)\Rightarrow (2)$}. Suppose that $(G, \mu)$ is not weakly fully HN-decomposable. By the minute criterion (Proposition \ref{prop_minute criterion} and Remark \ref{rem_minute criterion}), there exists $\alpha\in\Delta_0$ such that
\[
\langle \mu^{\diamond},\tilde{\omega}_{-w_0\alpha}\rangle +\{\langle \xi^{\diamond}, \tilde{\omega}_{-w_0\alpha}\rangle\}=\langle-w_0\mu^{\diamond}, \tilde{\omega}_{\alpha}\rangle+\{\langle -w_0 \xi^{\diamond},\tilde w_{\alpha}\rangle \} \in\Z_{>1}.
\]
We want to show that there exists a point in a HN-indecomposable Newton stratum that is not weakly admissible.
Let $\alpha^* =-w_0\alpha$ and $M=M_{\alpha^*}$ the standard Levi such that $\Delta_{0,M}=\Delta_0\setminus \{\alpha^*\}$. Let $P$ be the standard parabolic subgroup of $H$ corresponding to $M$. Since  $\langle \mu^{\diamond}+\xi^{\diamond},\tilde{\omega}_{\alpha*}\rangle \in \mathbb Z$ and $[b^H]\in B(H,\mu^{\#}+\xi,\mu^{\diamond})$, the element
 $[b^H]$ has a reduction $[b_M]$ to $M$ (Lemma \ref{lemma_b has reduction to Levi}).  By \cite[Lemma 1.8]{Ch1},  \[
\xi\in \mathrm{Im}(\pi_1(M)_{\Gamma,tor}\ra \pi_1(G)_{\Gamma,tor}),
\]
and there exists $\mu_1\in W\mu\subset X_*(T)$ which is $M$-dominant, such that
\[
[b_M]\in B(M, \mu_1^{\#}+\xi, \mu_1).
\]
Here, $\xi$ is view as an element in $\pi_1(M)_{\Gamma, tor}$. As $\langle \mu^{\diamond},\tilde w_{\alpha^*}\rangle >0$ while $\mu_1^{\diamond}=\mu_1^{\diamond}-\nu_{b_M}\in \langle \Phi_{M,0}^{\vee}\rangle_{\mathbb Q}$, it follows that $\mu\neq\mu_1$, and $\mu_1$ is $M$-dominant but not $H$-dominant. So there exists $\beta\in\Delta$ with $\beta|_A=\alpha^*$, such that
\[
\langle\mu_1, \beta\rangle<0.
\]
Therefore $\langle \mu_1,\beta\rangle=-1$ as $\mu$ and hence $\mu_1$ are minuscule. Let $[b'_M]\in B(M)_{basic}$ such that
\[
\kappa_M(b'_M)=\textrm{image of }-\beta^{\vee}+\xi \text{ in }\pi_1(M)_{\Gamma}.
\]
Then
\[
\nu_{b_M'}=\mathrm{pr}_{M}(-\beta^{\vee}+\xi)^{\diamond}=\mathrm{pr}_{M}(-\beta^{\vee})^{\diamond},
\]
with $\mathrm{pr}_{M}$ the projection of $X_*(T)=\langle \Phi_{M}^{\vee}\rangle_{\mathbb Q}\oplus \langle \Phi_{M}\rangle_{\mathbb Q}^{\perp}$ to the direct factor $\langle \Phi_{M}\rangle_{\mathbb Q}^{\perp}$. Let
\[
\mu_2=(s_{\beta}\mu_1)_{M\textrm{-dom}}=(\mu_1+\beta^{\vee})_{M\textrm{-dom}}.
\]
We check easily that
\[
[b'_M]\in B(M, (\mu_1-\mu_2)^\#+\xi, \nu_{b_M}-\mu_2^{\diamond}).
\]
Hence there exists a point $x_M\in\Fc(M, \mu_2)(C)$ such that $\E_{b'_M}\simeq \E_{b_M, x_M}$.  Let $[b'^{H}]$ be the image of $[b'_M]$ in $B(H)$. It follows that
\[[b'^{H}]\in B(H, \xi, \nu_b-\mu^{\diamond}).\]  In particular, $\mathcal{F}(H, \mu, b^H)^{[b'^{H}]}$ is a non-empty Newton stratum of the flag variety. Let $x\in\Fc(H, \mu)(C)$ the image of $x_M$ via the natural map $\Fc(M, \mu_2)\rightarrow \Fc(H, \mu)$.
In order to show the weakly admissible locus is not maximal, it suffices to prove the following Claim.

\textit{Claim: $(H, b'^H, -w_0\mu)$ is HN-indecomposable and $x\in \Fc(H, \mu, b^H)^{[b'^{H}]}(C)$ is not weakly admissible.}

 As $\nu_{b_M'}$ is $H$-antidominant, $\nu_{b'^H}=(\nu_{b_M'})_{H\textrm{-dom}}=w_0\nu_{b_M'}$. It follows that  $(H,b'^H,\mu^{-1})$ is HN-indecomposable as
\begin{eqnarray*}
\langle -w_0\mu^{\diamond}-w_0\nu_{b_M'},\tilde{\omega}_{\alpha}\rangle & =& \langle -w_0\mu^{\diamond},\tilde{\omega}_{\alpha}\rangle +\langle \nu_{b_M'},\tilde{\omega}_{\alpha^*}\rangle \\& = &  \langle -w_0\mu^{\diamond},\tilde{\omega}_{\alpha}\rangle -1 >0.
\end{eqnarray*}since $H$ is simple.
Moreover, $\E_{b, x}\simeq \E_{b'}$, and $x$ is not weakly admissible by definition: for the dominant root $\chi=mw_{\beta}\in X^*(P/Z_H)^+$ with $m>0$, we have \[\deg\chi_{*}((\mathcal E_{b,x})_P)=\deg(\chi_*(\mathcal{E}_{b'_M}))=\langle -w_0\nu_{b'_M}, mw_\beta\rangle>0,\]since $\nu_{b'_M}$ is $H$-anti-dominant and non trivial.

\underline{$(2)\Rightarrow (1)$}. Suppose that $(G, \mu)$ is weakly fully HN-decomposable. Let
\[
x\in \Fc(H,\mu,b^H)^{[b'^{H}]}(C)
\]
with $(H, b'^{H}, -w_0\mu)$ HN-indecomposable. So $\mathcal E_{b^H,x}=\mathcal E_{b'^H}$. We want to show that $x$ is weakly admissible. Suppose that $x$ is not weakly admissible, then there exist a maximal standard Levi subgroup $M_{\alpha}$, a reduction of $b^H$ to $M_{\alpha}$, and thus a reduction $(\E_{b^H, x})_{P_{\alpha}}$ of $\mathcal E_{b^H,x}=\mathcal E_{b'^H}$ to the corresponding standard parabolic $P_{\alpha}$, such that for some $\chi\in X^*(P_{\alpha}/Z_H)^+=X^*(P_{\alpha})^+$
\begin{equation}\label{eq:inequality-for-slope-vector}
\mathrm{deg}\chi_*((\E_{b, x})_{P_{\alpha}})>0.
\end{equation}
In particular, $\langle \mu^{\diamond}+\xi^{\diamond},\tilde w_{\alpha}\rangle\in \mathbb Z$ by Lemma \ref{lemma_b has reduction to Levi}. Let $v\in X_*(A)_{\mathbb Q}$ be the slope vector for the reduction $(\E_{b, x})_{P_{\alpha}}$ of $\mathcal E_{b,x}=\mathcal E_{b'}$ to $P_{\alpha}$. Indeed,
\[
v\in \Hom(X^*(P_{\alpha}),\mathbb Z)^{\Gamma}= \Hom(X^*(M_{\alpha}^{ab}),\mathbb Z)^{\Gamma}\stackrel{\sim}{\longrightarrow} X_*(M^{ab}_{\alpha})^{\Gamma}\subseteq X_*(A_{M_{\alpha}})_{\mathbb Q}\subset X_*(A)_{\mathbb Q}
\]
where $M^{ab}_{\alpha}$ is the cocenter of $M_{\alpha}$ and $A_{M_{\alpha}}\subseteq A$ is a maximal split central torus of $M_{\alpha}$.


 In other words, let $\mathcal E_{b'_{M_{\alpha}}}:=(\mathcal E_{b,x})_{P_{\alpha}}\times^{P_{\alpha}} M_{\alpha}$ with $[b'_{M_\alpha}]\in B(M_{\alpha})$, then \[[b'_{M_\alpha}]\in B(M_\alpha, \kappa_{M_\alpha}(b'_{M_\alpha}), \nu_{b'_{M_\alpha}}).\] Moreover, as \[
v=\mathrm{pr}_{M_{\alpha}}(\nu_{\mathcal E_{b'_{M_{\alpha}}}})^{\diamond}=-\mathrm{pr}_{M_{\alpha}}(\nu_{b'_{M_\alpha}})^{\diamond},
\]
with $\mathrm{pr}_{M_{\alpha}}$ the projection of $X_*(T)=\langle \Phi_{M_{\alpha}}^{\vee}\rangle_{\mathbb Q}\oplus \langle \Phi_{M_{\alpha}}\rangle_{\mathbb Q}^{\perp}$ to the direct factor $\langle \Phi_{M_{\alpha}}\rangle_{\mathbb Q}^{\perp}$, It follows that \[-v\in B(M_\alpha, \kappa_{M_\alpha}(b'_{M_\alpha}), \nu_{b'_{M_\alpha}}).\]

The fact that $M_{\alpha}=\mathrm{Cent}_{G}(v)$ and the inequality \eqref{eq:inequality-for-slope-vector} implies
\begin{eqnarray}\label{eqn_thm 2.4}\langle v,\tilde w_{\alpha}\rangle =\langle \nu_{\mathcal E_{b'_{M_{\alpha}}}},\tilde w_{\alpha}\rangle>0.
\end{eqnarray}
Hence $v$ is $H$-dominant. Furthermore, as
\[
\kappa_H(b'_{M_{\alpha}})=-c_1^H(\mathcal E_{b'_{M_{\alpha}}}\times^{M_{\alpha}}H)=-c_1^H(\mathcal E_{b,x})=\kappa_H(b')=\xi\in\pi_1(H)_{\Gamma} 
\]
we deduce that $-w_0v\in B(H, \xi, (\nu_{b'_{M_\alpha}})_{H-dom})$. In particular, $-w_0v\in B(H, \xi, -w_0v)$.
On the other hand, $v\in X_*(A)_{\mathbb Q}^+$ and $v\preceq \nu_{\E_{b, x}}=-w_0\nu_{b'}$ by \cite[Theorem 1.8]{CFS}. Hence,
\[
-w_0v\preceq \nu_{b'}\preceq -w_0\mu^{\diamond},
\]
and then \[-w_0v\in B(H,\xi,-w_0\mu^{\diamond}).\] By proposition \ref{prop:Kottwitz-set-in-NG}, it follows that
\begin{eqnarray}\label{eqn_thm 2.1}\langle -w_0v, \tilde{\omega}_{\alpha^*}\rangle+\{\langle -\xi^{\diamond}, \tilde{\omega}_{\alpha^*}\rangle\}\in\mathbb{Z}_{\geq 0}
\end{eqnarray}

As $(G, \mu)$ is weakly fully HN-decomposable, by Remark \ref{rem_minute criterion}, and the fact that $b^H$ has reduction to $M_{\alpha^*}$, we have
\begin{eqnarray}\label{eqn_thm 2.2}\langle -w_0\mu^{\diamond}, \tilde{\omega}_{\alpha^*}\rangle+\{\langle -\xi^{\diamond}, \tilde{\omega}_{\alpha^\vee}\rangle\}= 0 \text{ or }1
\end{eqnarray}

Moreover,
\begin{eqnarray}\label{eqn_thm 2.3}
\langle -w_0\mu^{\diamond}-(-w_0v),\tilde{\omega}_{\alpha^*}\rangle \geq \langle -w_0\mu-\nu_{b'},\tilde w_{\alpha^*}\rangle >0.
\end{eqnarray}
Here we have the last inequality because the centralizer of $\nu_{b'}=-w_0\nu_{\mathcal E_{b'}}$ is contained in $M_{\alpha^*}$ and $\nu_{b'}$ is HN-indecomposable relative to $-w_0\mu$. Combined with (\ref{eqn_thm 2.1}), (\ref{eqn_thm 2.2}), (\ref{eqn_thm 2.3}),
it follows that  $\langle -w_0v,\tilde w_{\alpha^*}\rangle =0$ which contradicts with (\ref{eqn_thm 2.4}).

\underline{$(3)\Rightarrow (1)$}. It follows directly from a result of Viehmann (cf. Proposition \ref{prop_Newton and wa} (2)).

\underline{$(2)\Rightarrow (3)$}. Suppose $(G, \mu)$ is weakly fully HN-decomposable. 
As we have proved $(2)\Rightarrow (1)$, for any $[b']\in B(G, 0, -w_0\mu)$ which is HN-indecomposable with respect to $-w_0\mu$, the Newton stratum $\mathcal{F}(G, \mu, b)^{[b']}$ is contained in the weakly admissible locus $\mathcal{F}(G, \mu, b)^{wa}$. Therefore it suffices to consider $[b']\in B(G, 0, -w_0\mu)$ which are HN-decomposable. Let $[b'^{H}]\in B(H)$ be the image of $[b']$ via inner twist. Suppose that $[b'^{H}]\in B(H, \xi^{\sharp}, -w_0\mu)$ is HN-decomposable with respect to a standard Levi subgroup $M$ of $H$ (i.e. $\nu_{b'^H}\preceq_{M}-w_0\mu^{\diamond}$ and $M$ is the smallest standard Levi subgroup with this property). For any $x\in \mathcal{F}(H, \mu, b^H)^{[b'^{H}]}(C)$, we want to show that the Harder-Narasimhan vector $\mathrm{HN}_{b^H}(x)$ does not depend on $x$. By Proposition \ref{prop_NGVi} (2), $\mathrm{HN}_{b^H}(x)\preceq_M \nu_{b'^H}$. We claim that 
\[
\mathrm{HN}_{b^H}(x)=\mathrm{pr}_M\nu_{b'^H},
\]
or equivalently, that $\mathrm{HN}_{b^H}(x)$ is central in $M$. Write $v:=\mathrm{HN}_{b^H}(x)$.

Suppose $v$ that is NOT central in $M$. Then there exists $\alpha\in\Delta_{M, 0}$ such that $\langle v, \alpha\rangle>0$. As $-w_0\mu^{\diamond}-v$ is a nonnegative combination of coroots in $\Delta_{M,0}^{\vee}$ and $\alpha\in \Delta_{M, 0}$, we have
\begin{eqnarray}\label{eqn_thm_1}
\langle -w_0\mu^{\diamond}-v, \tilde{\omega}_{\alpha}\rangle>0.
\end{eqnarray}
Moreover, $\langle \mathrm{HN}_{b^H}(x), \alpha\rangle=\langle v, \alpha\rangle\neq 0$ implies that $b_H$ has reduction to $M_{\alpha}$. Therefore, by Lemma \ref{lemma_b has reduction to Levi} and Lemma \ref{lemma_reduction b dual}, we have
\begin{eqnarray*}\langle\mu^{\diamond}+\xi^{\diamond}, \tilde{\omega}_\alpha\rangle\in\mathbb{Z}.\end{eqnarray*}
Then by the minute criterion for the weakly fully HN-decomposability (Proposition \ref{prop_minute criterion} and Remark \ref{rem_minute criterion}), \begin{eqnarray}\label{eqn_thm_2}\langle -w_0\mu^{\diamond}, \tilde{\omega}_\alpha\rangle+\{\langle-\xi^{\diamond}, \tilde{\omega}_\alpha\rangle\}=0 \text{ or }1. \end{eqnarray}

On the other hand, by Proposition \ref{prop_NGVi} (1) and Corollary \ref{coro_involution_BG}, $v\in B(H, \xi^\sharp, -w_0\mu^{\diamond})$. In particular, by Proposition \ref{prop:Kottwitz-set-in-NG},
\begin{eqnarray}\label{eqn_thm_3}\langle \xi^{\diamond}-v, \tilde{\omega}_{\alpha}\rangle\in\mathbb{Z}.\end{eqnarray} Combined with (\ref{eqn_thm_2}) and (\ref{eqn_thm_3}), we deduce
\begin{eqnarray}\label{eqn_thm_4}
\langle -w_0\mu^{\diamond}, \tilde{\omega}_\alpha\rangle+\{\langle-v, \tilde{\omega}_\alpha\rangle\}=0 \text{ or }1. 
\end{eqnarray}
As $v$ is dominant, $\langle v, \tilde{\omega}_{\alpha}\rangle\geq 0$. It follows that
\begin{eqnarray*}
\underbrace{\langle -w_0\mu^{\diamond}, \tilde{\omega}_\alpha\rangle+\{\langle-v, \tilde{\omega}_\alpha\rangle\}}_{=0\text{ or }1}-\underbrace{\langle -w_0\mu^{\diamond}-v, \tilde{\omega}_{\alpha}\rangle}_{>0 \textrm{ by }\eqref{eqn_thm_1}}\in\N.
\end{eqnarray*}
Therefore the left hand side is 0 and it follows that $\langle v, \tilde{\omega}_\alpha\rangle=0$ which implies $v=0$ since $H$ is simple. This contradicts to the fact that $v$ is not central in $M$.
\end{proof}

\section{Newton strata completely contained in the weakly admissible locus}\label{Sec_single Newton}
In this section, we study the question when a single Newton stratum is contained in the weakly admissible locus of the flag variety. We shall start with some general observation which works for a quasi-split reductive group, and give an answer to the above question by establishing a criterion for a Newton stratum completely contained in the weakly admissible locus. At the end of this section, we illustrate our criterion by an explicit example in the $\mathrm{GL}_n$-case.

\subsection{Extensions and weakly admissible locus} Recall that $H$ is the quasi-split inner form of $G$ over $F$, and $T$ is a maximal $F$-torus of $H$, with
\[
W=\left(N_{H}(T)/T\right)(\Fbar)
\]
its Weyl group. Let $\mu\in X_*(T)$ be a minuscule cocharacter. Let $M$ be a standard Levi subgroup of $H$, with corresponding parabolic subgroup $P$. We have the decomposition in Schubert cells of $\mathcal F(H,\mu)_{\Fbar}$ according to the $P_{\Fbar}$-orbits given as follows. For $w\in W$, let $\mathcal F(H,\mu)_{P}^{w}$ be the schematic image of the map
\[
P_{\Fbar}\longrightarrow \mathcal F(H,\mu)_{\bar F}, \quad g\mapsto gP_{\mu^{w}}g^{-1}.
\]
So $\mathcal F(H,\mu)_{P}^w$ is the $P_{\Fbar}$-orbit of $P_{\mu^{w}}\in \mathcal F(H,\mu)(\Fbar)$, and
\[
\mathcal F(H,\mu)_P^w=P_{\bar F}/P_{\Fbar}\cap P_{\mu^{w}}.
\]
Moreover, every $P_{\Fbar}$-orbit of $\mathcal F(H,\mu)_{\bar F}$ is obtained in this way: this follows from the relative Bruhat decomposition
\[
H_{\Fbar}=\coprod_{[w]\in W_{P}\backslash W/W_{P_{\mu}} }P_{\Fbar}wP_{\mu}
\]
Here $W_{P}$ denotes the stabilizer of the standard parabolic subgroup $P_{\Fbar}\subset H$ under the action of Weyl group $W$ (on the set of standard parabolic subgroups of $H_{\Fbar}$), and $[w]$ is the double coset in $W_P\backslash W\slash W_{P_{\mu}}$ of an element $w$ in the Weyl group $W$. By taking the projection to the Levi quotient $M$, we obtain a map
\begin{equation}\label{eq:affine-fibration}
\mathrm{pr}_{P,w}: \Fc(H, \mu)_P^w\rightarrow \Fc(M, \mu^w),
\end{equation}
It is known that the above map $\mathrm{pr}_{P,w}$ above is an affine fibration. 

\begin{definition}\label{def_mu negative}Let $\mu$ be a minuscule cocharacter of $H$, $[b]\in A(H,\mu)_{basic}$. Let $M$ be a standard Levi subgroup of $H$. An element $w\in W$ is called \emph{$\mu$-negative for $M$} if $\langle \nu_{b}-w\mu, \chi\rangle <0$ for some $\chi\in X^*(P/Z_H)^+$.
\end{definition}

\begin{proposition}\label{prop:non-wa-locus}Let $\mu$ be a minuscule cocharacter of $H$, $[b]\in A(H,\mu)_{basic}$. 
\begin{enumerate}
\item Let $M$ be a standard Levi subgroup of $H$, with $P$ its corresponding standard parabolic subgroup. Assume that $w\in W$ is $\mu$-negative and that $b$ has a reduction $(b_M,h)$ to $M$ with $b_M\in M(\breve{F})$ and $h\in H(\breve{F})$, then
\[
h\cdot\mathcal F(H,b)_{P}^w\subset  \mathcal F(H,\mu)_{\Fbar}\setminus \mathcal F(H,\mu,b)_{\Fbar}^{wa}.
\]
\item The complement in $\mathcal F(H,\mu)$ of the weakly admissible locus can be described as follows: for $C/\Fbar$ a complete algebraically closed field extension of $\Fbar$,
\[
\left(\Fc(H, \mu)\backslash \Fc(H, \mu, b)^{wa}\right)(C)=\bigcup_{\alpha\in \Delta_0}\left(\bigcup_{\substack{(b_M,h) \textrm{ is a reduction} \\ \textrm{of } b \textrm{ to }M_{\alpha} }}\left(\bigcup_{\substack{ w\in W\text{ is } \mu\text{-negative}\\ \textrm{for } M_{\alpha}}}h\cdot \Fc(H, \mu)_{P_{\alpha}}^w(C)\right)\right).
\]
\end{enumerate}
\end{proposition}

\begin{remark}Viehmann told us that this proposition is a reformulation of \cite[Corollary 4.5]{NGVi}. 
\end{remark}

\begin{proof} (1) Since $b=hb_M\sigma(h)^{-1}$, the element $h\in H(\breve F)$ induces an isomorphism of $H$-bundles 
\[
\iota=\iota_h:\mathcal E_{b_M}\times^{M}H\stackrel{\sim}{\longrightarrow}\mathcal E_{b}, 
\]  
such that, under the natural identification in Remark \ref{rem:Eb-canonically-trivialized}, the restriction of $\iota$ to $\Spec(\hat{\cO}_{X,\infty})=\Spec(B_{dR}^+)$ is the right-multiplication by $h$. 
So for $C$ a complete algebraically closed field extension of $\Fbar$ and for $x\in \mathcal F(H,\mu)(C)$, the isomorphism $\iota$ above induces a compatible isomorphism between the modifications of $H$-bundles (see also Remark \ref{rem:Ebx-depends-on-b}): 
\[
\left(\mathcal E_{b_M}\times^{M}H\right)_{h^{-1}x}\stackrel{\sim}{\longrightarrow}\mathcal E_{b,x}. 
\]
If furthermore $x\in h\cdot \mathcal F(H,\mu)_{P}^{w}$, and by \cite[Lemma 2.6]{CFS} there is an isomorphism
\[
(\E_{b,x})_P\times ^P M\simeq \E_{b_M, \mathrm{pr}_{P,w}(h^{-1}x)}.
\]
On the other hand, $w$ is $\mu$-negative for $M$, so $\langle w\mu-\nu_b,\chi\rangle >0$ for some $\chi\in X(P/Z_{G})^+$. Thus
\[
\deg(\chi_*(\mathcal E_{b,x})_P)=\deg(\chi_*(\mathcal E_{b_M,\mathrm{pr}_{P,w}(h^{-1}x)}))=\langle w\mu-\nu_{b_M},\chi\rangle >0.
\]
Therefore $x$ is not weakly admissible by Proposition \ref{prop:criterion-for-wa}.

(2) In view of (1) and the criterion for the weak admissibility in the quasi-split case (Proposition \ref{prop:criterion-for-wa}), this follows from the the following two facts: once an element $b$ has a reduction to a proper standard Levi $M$, it has a reduction to every maximal standard Levi containing $M$; moreover, if $\langle \nu_b-w\mu,\chi\rangle <0$ for some $\chi\in X^*(P/Z_H)^+$, then the same holds  with $\chi=\tilde w_{\alpha}$ for a certain $\alpha\notin \Delta_{0,M}$.
\end{proof}

Suppose $[b]\in A(H,\mu)_{basic}$ with $b_M$ a reduction of $b$ to a standard Levi $M$ of $H$. Recall that we have Newton stratifications
\[
\begin{split}\Fc(H, \mu)=&\coprod_{[b']\in B(H, \kappa(b)-\mu^{\#}, \nu_b-\mu^{\diamond})} \Fc(H, \mu, b)^{[b']}, \quad \textrm{and}\\
\Fc(M, \mu^w)=&\coprod_{[b'_M]\in B(M, \kappa_M(b_M)-\mu^{w,\#}, \nu_{b_M}-\mu^{w, \diamond})}\Fc(M, \mu^w, b_M)^{[b'_M]}.
\end{split}
\]
Let $P$ be the standard parabolic subgroup corresponding to $M$. We would like to compare these two Newton stratifications via the affine fibration $\mathrm{pr}_{P,w}$ in \eqref{eq:affine-fibration} above.

\begin{definition}\label{defn_extension}Let $M$ be a standard Levi subgroup of $H$, with $P$ the corresponding standard parabolic subgroup. Let $[b'^H]\in B(H)$ and $[b'_M]\in B(M)$. We say that $[b'^H]$ is an \emph{extension} of $[b'_M]$ if $\E_{b'^H}$ has a reduction $(\E_{b'^H})_P$ to $P$ such that $(\E_{b'^H})_P\times^P M\simeq \E_{b'_M}$. Let $[b']\in B(G)$ be the preimage of $[b'^H]$ via the identification $B(G)\simeq B(H)$. We also say that $[b']$ is an \emph{extension} of $[b'_M]$ if $[b'^H]$ is an extension of $[b'_M]$. 
\end{definition}

\begin{remark} \begin{enumerate}
   \item Suppose $H=\mathrm{GL}_n$ and $M=\mathrm{GL}_r\times \mathrm{GL}_{n-r}$ a standard Levi of $H$. An element $[b']\in B(H)$ is an extension of
\[
[b'_M]=([b_1''],[b_2''])\in B(M)=B(\mathrm{GL}_r)\times B(\mathrm{GL}_{n-r})
\]
if and only if the rank $n$ vector bundle $\mathcal E_{b'}$ on the Fargues-Fontaine curve $X$ is an extension of the vector bundle $\mathcal E_{b_2''}$ of rank $n-r$ by the vector bundle $\mathcal E_{b_1''}$ of rank $r$. 

 \item When $[b'_M]\in B(M)$ is basic with $\nu_{b'_M}$ anti-$H$-dominant, then the extensions $[b']\in B(G)$ of $[b'_M]$ are classified by \cite{BFH} for $\mathrm{GL}_n$ and by \cite{Vi} for arbitrary $G$. 
\end{enumerate}
\end{remark}

Suppose that $[b]\in A(H,\mu)_{basic}$ has a reduction $(b_M,h)$ to a proper standard Levi subgroup $M$ of $H$. Let $P$ be the standard parabolic subgroup corresponding to $M$. Then, for every $w\in W$ and for every $[b']\in B(H,\kappa(b)-\mu^{\sharp},\nu_b-w_0\mu^{\diamond})$, we have

\begin{equation}\label{eqn_comparison Newton stratification}
\mathrm{pr}_{P,w}(\Fc(H, \mu)_P^w\cap h^{-1}\Fc(H, \mu, b)^{[b']}) \subseteq  \coprod_{\substack{[b'_M]\in B(M, \kappa(b_M)-\mu^{w,\#}, \nu_{b_M}-\mu^{w,\diamond})\\ \textrm{such that } [b']\text{ is an  extension of }[b_M'] \\}}\Fc(M, \mu^w, b_M)^{[b'_M]}.
\end{equation}

\begin{proposition}\label{prop:modification M to G} The inclusion in \eqref{eqn_comparison Newton stratification} is an equality.
\end{proposition}

\begin{proof} Since $(b_M,h)$ is a reduction of $b$ to $M$, we have $b=hb_M\sigma(h)^{-1}$. So by Remark \ref{rem:Ebx-depends-on-b}, we have \[
h^{-1}\cdot \mathcal F(G,\mu,b)^{[b']}=\mathcal F(G,\mu,b_M)^{[b']}\subset \mathcal F(G,\mu). 
\]
Here for the second term above, we view $b_M\in M(\breve F)$ as an element of $G(\breve{F})$. Therefore, replacing $b\in H(\breve F)$ by $b_M\in M(\breve F)\subset H(\breve F)$, we shall assume $h=1$ for the remaining part of the proof.  

Let $[b'_M]\in B(M, \kappa(b_M)-\mu^{w,\#}, \nu_{b_M}-\mu^{w,\diamond})$ such that $[b']$ is an  extension of $[b_M']$, and let $x'\in \mathcal F(M,\mu^{w},b_M)^{[b_M']}$. We need to find $x\in \mathcal F(H,\mu)^{w}_P\cap \mathcal F(H,\mu,b)^{[b']}$ such that $\mathrm{pr}_{P,w}(x)=x'$. Write $U=X\setminus\{\infty\}$. Recall that $\mathcal E_{b_M'}=\mathcal E_{b_M,x'}$ is a modification of $\mathcal E_{b_M}$: so there is an isomorphism
\begin{equation}\label{eq:iso-M-bundles-over-U}
\mathcal E_{b_M}|_U\stackrel{\sim}{\longrightarrow} \mathcal E_{b_M,x'}|_U=\mathcal E_{b'_M}|_U\simeq\mathcal E_{b',P}|_U\times^P M
\end{equation}
of $M$-bundles over $U$, where $\mathcal E_{b',P}$ is defined in Definition \ref{defn_extension} as $[b']$ is an extension of $[b'_M]$. We claim that the isomorphism \eqref{eq:iso-M-bundles-over-U} can be extended to a modification of $P$-bundles over $U$
\begin{equation}\label{eq:iso-P-bundles-over-U}
\mathcal E_{b_M}|_U\times^M P \stackrel{\sim}{\longrightarrow} \mathcal E_{b',P}|_U.
\end{equation}
Indeed, by a result of Ansch\"utz (\cite[Theorem 4]{Ans}), the $M$-bundle $\mathcal E_{b_M}$ is trivial over $U$. So the isomorphism \eqref{eq:iso-M-bundles-over-U} corresponds to a section $s$ of $\mathcal E_{b',P}\times^PM$ over $U$. We want to lift the section $s$ to a section of $\mathcal E_{b',P}$ over $U$ through the canonical map
\[
\mathcal E_{b',P}\longrightarrow \mathcal E_{b',P}\times^{P}M.
\]
The latter makes $\mathcal E_{b',P}$ a $R_u(P)$-bundle over $\mathcal E_{b',P}\times^{P}M$, where $R_u(P)$ denotes the unipotent radical of $P$. Consider the cartesian diagram below
\[
\xymatrix{\mathcal U\ar[rr]\ar[d] & & \mathcal E_{b',P}\ar[d] \\ U\ar[rr]^s &  & \mathcal E_{b',P}\times^PM}.
\]
In particular, $\mathcal U$ is a $R_u(P)$-bundle over $U$. Since $U$ is affine and $R_u(P)$ is unipotent over $F$ (thus a successive extension of $\mathbb G_a$), every $R_u(P)$-bundle must be trivial.  So $\mathcal U$ has a section over $U$, giving a section of $\mathcal E_{b',P}$ over $U$ lying above $s$, as claimed.

On the other hand, the restriction $
\mathcal E_{b_M}|_{\Spec(\hat{\cO}_{X,\infty})}$ is canonically trivialized (Remark \ref{rem:Eb-canonically-trivialized}), so the modification \eqref{eq:iso-P-bundles-over-U} corresponds to a coset
\[
x_PP(B_{dR}^+)\in \mathrm{Gr}^{B_{dR}}_{P}(C)=P(B_{dR})/P(B_{dR}^+),
\]
whose image in $\mathrm{Gr}_{M}^{B_{dR}}(C)=M(B_{dR})/M(B_{dR}^+)$ is
\[
\pi_{M,\mu^{w}}^{-1}(x')=m\mu^{w,-1}(t)M(B_{dR}^+)\in M(B_{dR}^+)\mu^{w,-1}(t)M(B_{dR}^+)/M(B_{dR}^+)=\mathrm{Gr}_{M,\mu^{w}}^{B_{dR}}(C).
\]
Here
\[
\pi_{M,\mu}:\mathrm{Gr}_{M,\mu^{w}}^{B_{dR}}(C)\stackrel{\sim}{\longrightarrow} \mathcal F(M,\mu^{w})(C)
\]
is the Bialynicki-Birula map (for $M$) recalled in \S~\ref{sec:modification-of-G-bundles}. Hence the element $x_P$ is of the form
\[
um\mu^{w,-1}(t)m'\in P(B_{dR}),
\]
with $u\in R_u(P)(B_{dR})$, $m,m'\in M(B_{dR}^+)$. But we can further write $u$ as
\[
u=u 'u ''
\]
with $u'\in R_u(P)(U)$ and $u''\in P(B_{dR}^+)$: indeed, using the fact that $R_u(P)$ is a successive extension of $\mathbb G_a$, by a standard d\'evissage, one reduces to the similar assertion for the addition group $\mathbb{G}_a$, which is clear. Now, composing \eqref{eq:iso-P-bundles-over-U} with ${u'}
^{-1}\in R_u(P)(U)\subset G(U)$, we get a new modification $\mathcal E_{b_M}|_U\times^M P \stackrel{\sim}{\rightarrow} \mathcal E_{b',P}|_U$ compatible with \eqref{eq:iso-M-bundles-over-U}, which is now minuscule of type $\mu^{w}$. From this new modification, we deduce a minuscule modification of type $\mu$
\[
\mathcal E_{b}|_U\stackrel{\sim}{\longrightarrow} \mathcal E_{b'},
\]
or equivalently, a point $x\in \mathcal F(H,\mu,b)_P^{[b']}$ lying above $x'$. This completes the proof of our proposition.
\end{proof}

\begin{corollary}\label{cor:modification M to G} Let $\mu$ be a minuscule cocharacter of $H$, and $[b]\in A(H,\mu)_{basic}$. Assume that $[b]$ has a reduction $(b_M,h)$ to some standard Levi subgroup $M$ of $H$, with $b_M\in M(\breve F)$ and $h\in H(\breve F)$. Let $[b']\in B(H, \kappa(b)-\mu^{\#}, \nu_{b}-\mu^{\diamond})$ and $w\in W$. Then
\[
\left(h\cdot \Fc(H, \mu)_P^{w}\right)\cap \Fc(H, \mu, b)^{[b']}\neq \emptyset.
\]
if and only if $[b']$ is an extension of a certain $[b_M']\in B(M,\kappa(b_M)-\mu^{w,\#},\nu_{b_M}-\mu^{w,\diamond})$.
\end{corollary}
\begin{proof} This follows directly from Proposition \ref{prop:modification M to G}.
\end{proof}

This corollary allows us to have a criterion about whether a single Newton stratum is completely contained in the weakly admissible locus for a general connected reductive group.
\begin{theorem} \label{thm:criterion-for-a-singule-stratum}Let $\mu$ be a minuscule cocharacter of the \emph(not necessarily quasi-split\emph) connected reductive group $G$, and $b\in G(\breve F)$ such that $[b]\in B(G,\mu)_{basic}$. Let $[b']\in B(G, 0, \nu_b-\mu)$. Then
\[
\Fc(G, \mu, b)^{[b']}\nsubseteq\Fc(G, \mu, b)^{wa}
\]
if and only if there exists some maximal proper standard Levi subgroup $M$ of $H$, the quasi-split inner form of $G$ over $F$, and $w\in W$ satisfying the following two properties:
\begin{enumerate}
\item $b$ has a reduction $b_M$ to $M$ and $w$ is $\mu$-negative for $M$; and
\item $[b']$ is an extension of some $[b'_M]\in B(M, \kappa(b_M)-\mu^{w,\#}, \nu_{b_M}-\mu^{w,\diamond})$.
\end{enumerate}
\end{theorem}

\begin{proof} Recall the identification \eqref{eq:BG-and-BH} between $B(G)$ and $B(H)$ induced by the class
\[
\xi\in H^1(F,H)\simeq \pi_1(H)_{\Gamma, tor}
\]
of $G$ as an inner form of $H$, under which $B(G,\mu)$ is identified with $B(H,\mu^{\#}+\xi,\mu^{\diamond})$. Moreover, using a fixed inner twisting
$G_{\breve F}\stackrel{\sim}{\longrightarrow} H_{\breve F}$, we have an identification
\[
\mathcal F(G,\mu)\stackrel{\sim}{\longrightarrow}\mathcal F(H,\mu)
\]
of flag varieties over $\breve E$, under which the Newton stratification (resp. the Harder-Narasimhan stratification) on $\mathcal F(G,\mu)$ for the triple $(G, \mu, b)$ is identified with the corresponding stratification on $\mathcal F(H,\mu)$ for the triple $(H, \mu, b^H)$, where $[b^H]\in B(H)$ is the image of $[b]$ via $B(G)\simeq B(H)$. Therefore the theorem follows immediately from Corollary \ref{cor:modification M to G} and Proposition \ref{prop:non-wa-locus} (2). 
\end{proof}

\subsection{An examples in the $\mathrm{GL}_n$-case}  In this \S, we illustrate an applications of Theorem \ref{thm:criterion-for-a-singule-stratum}, and we refer to the next section for some similar but more complicated applications of the same result.

Consider the case where $G=\mathrm{GL}_{10}$ and $\mu=(1^{(4)},0^{(6)})$. Let $b\in B(G)_{basic}$ with $\kappa(b)=\mu^{\#}$. In particular,
\[
\nu_b=\left(\frac{2}{5}^{(10)}\right)=\left(\underbrace{\frac{2}{5},\frac{2}{5},\cdots,\frac{2}{5}}_{10}\right)\in \mathcal N(\mathrm{GL}_{10}).
\]
The element $b$ has a reduction to only one proper maximal standard Levi
\[
M=\mathrm{GL}_5\times \mathrm{GL}_5\hookrightarrow G= \mathrm{GL}_{10},
\]
and if $w\in W$ is such that $\langle \nu_b-w\mu,\chi\rangle <0$ for some $\chi\in X^*(P/Z_G)^+$, then
\[
w\mu \in \left\{\left((1^{(3)},0^{(2)}),(1,0^{(4)})\right),\left((1^{(4)},0),(0^{(5)})\right)\right\}\in \mathcal N(M)=\mathcal N(\mathrm{GL}_5)\times \mathcal N(\mathrm{GL}_5).
\]
So
\[
\nu_{b_M}\mu^{w,-1}\in \left\{\left(\left(\frac{2}{5}^{(2)},-\frac{3}{5}^{(3)}\right),\left(\frac{2}{5}^{(4)},-\frac{3}{5}\right)\right),\left(\left(\frac{2}{5},-\frac{3}{5}^{(4)}\right),\left(\frac{2}{5}^{(5)}\right)\right)\right\}\in \mathcal N(M).
\]
Therefore, by Theorem \ref{thm:criterion-for-a-singule-stratum}, for $[b']\in B(G,0,\nu_b\mu^{-1})$, the Newton stratum $\mathcal F(G,\mu,b)^{[b']}$ contains a point that is not weakly admissible if and only if $\mathcal E_{b'}$ is an extension of $\mathcal E'$ by $\mathcal E''$, where the pair $(\mathcal E',\mathcal E'')$ is either contained in
\[
\left\{\cO\left(-\frac{1}{3}\right)\oplus \cO^2, \cO\left(-\frac{1}{4}\right)\oplus \cO, \cO\left(-\frac{1}{5}\right)\right\}\times \left\{\cO\left(\frac{1}{5}\right),\cO\oplus \left(\frac{1}{4}\right), \cO^2\oplus \cO\left(\frac{1}{3}\right), \cO^3\oplus \cO\left(\frac{1}{2}\right)\right\},
\]
or contained in
\[
\left\{\cO\left(-\frac{2}{5}\right)\right\}\times \left\{\cO\left(\frac{2}{5}\right),\cO\oplus \cO\left(\frac{1}{2}\right)^2,\cO\left(\frac{1}{3}\right)\oplus \cO\left(\frac{1}{2}\right)\right\}.
\]
Observe that, for each choice of $(\mathcal E',\mathcal E'')$ above, the slopes of $\mathcal E'$ are less or equal to those of $\mathcal E''$. In particular, every extension of $\mathcal E'$ by $\mathcal E''$ splits. On the other hand,
\[
\nu_b\mu^{-1}=\left(\frac{2}{5}^{(6)},\frac{-3}{5}^{(4)}\right),
\]
so there are $26$ non-empty Newton strata, listed as follows:
\begin{itemize}
\item $\nu_{b'}=\left(\frac{2}{5}^{(5)}, 0^{(i)}, -\frac{2}{5-i}^{(5-i)}\right)$ for $i=0,1$; or $\left(\frac{2}{5}^{(5)}, -\frac{1}{3}^{(3)},-\frac{1}{2}^{(2)}\right)$. We have $3$ possibilities in this case, and the polygons are HN-decomposable, in the sense that they all touch the polygon for $\nu_b\mu^{-1}$. As a result, the corresponding Newton strata do not contain any weakly admissible point.

\item $\nu_{b'}=\left(\frac{1}{3}^{(2)},-\frac{1}{2}^{(2)}\right)$. So $\mathcal E_{b'}$ is not a direct sum of $\mathcal E'\oplus \mathcal E'' $ for all possible choices $(\mathcal E',\mathcal E'')$, and the corresponding Newton stratum is completely contained in $\mathcal F(G,\mu,b)^{wa}$.

\item $\nu_{b'}=\left(\frac{1}{i}^{(i)},0^{(j)},-\frac{1}{10-i-j}^{(10-i-j)}\right)$ for $3\leq i\leq 8$ and $0\leq j\leq 8-i$. We have $21$ possibilities for this case. It is easy to see that, $\mathcal E_{b'}=\mathcal E'\oplus \mathcal E''$ for some choice $(\mathcal E',\mathcal E'')$ if and only if $i\leq 5$ and $i+j\geq 5$. So in this case, we get $9$ Newton strata that are completely contained in the weakly admissible locus.

\item $\nu_{b'}=\left(0^{(6)}\right)$, which corresponds to the admissible locus $\mathcal F(G,\mu,b)^{a}$ thus is completely contained in $\mathcal F(G,\mu,b)^{wa}$.
\end{itemize}
In summary, there are $26$ non-empty Newton strata in the Newton stratification of $\mathcal F(G,\mu)$, and $11$ of them are completely contained in the weakly admissible locus.

\begin{remark} The reason why we can determine all the Newton strata completely contained in the weakly admissible locus in this example is that all the involved extensions of $M$-bundles to $G$-bundles are trivial. In general, it's a difficult question to determine whether a $G$-bundle is an extension of a $M$-bundle. In the next section, we will study this question for $\mathrm{GL}_n$.
\end{remark}

\section{Extensions of Vector bundles over Fargues-Fontaine curve}\label{Sec_extension}

As seen in the previous section, to determine if a single Newton stratum $\mathcal F(G,\mu,b)^{[b']}$ is completely contained in the weakly admissible locus $\mathcal F(G,\mu,b)^{wa}$, we need to have a classification of extensions of $G$-bundles over the Fargues-Fontaine curve. In this section, we give such a classification for $\mathrm{GL}_n$ in an inductive way: see also the appendix below for a direct classification in some special cases. As an application this result, this gives an algorithm to determine which Newton strata are completed contained in the weakly addmissible locus for $G=\mathrm{GL}_n$.

\subsection{Inductive classification of extensions of vectors bundles}

For $n\in\N$, let
\[
\cN(n):=\cN(\mathrm{GL}_n)\subset \mathbb Q^n.
\]
Let $T\subset \mathrm{GL}_n$ be the diagonal torus, and $B\subset \mathrm{GL}_n$ the Borel subgroup of upper-triangular matrices. So $\mathcal N(n)$ can be identified with the set of rational cocharacters of $T$ which are dominant relative to the positive coroots defined by $B$. In other words,
\[
\mathcal N(n)=X_*(T)_{\mathbb Q}^+=\{(a_1,\ldots, a_n)\in \mathbb Q^n| a_1\geq a_2\geq \cdots \geq a_n\}.
\]
For $\ua\in \cN(n)$, let $\cO(\ua)$ be the corresponding vector bundle of rank $n$ over the Fargues-Fontaine curve $X$. We say that $\ua$ is semistable (resp. stable) if the corresponding vector bundle $\cO(\ua)$ is semistable (resp. stable). In general, the stable (resp. semi-stable) blocks in $\cO(\ua)$ are also called the stable (resp. semistable) blocks in $\ua$.
For $\ua\in \cN(n)$ and $\ub\in \cN(m)$, their direct sum
\[
\ua\oplus \ub\in \cN(n+m)
\]
is defined in such a way that $\cO(\ua\oplus \ub)=\cO(\ua)\oplus \cO(\ub)$. For $\ua=(a_1,\ldots, a_n)\in \cN(n)$, we define its dual $\ua^{\vee}$ by
\[
\ua^{\vee}:=(-a_n,\ldots, -a_1)\in \cN(n).
\]
In other words, $\ua^{\vee}=-w_0\ua$. Clearly $\cO(\ua^{\vee})=\cO(\ua)^{\vee}$. Finally, for $\uc\in \mathbb Q^n$, set
\[
S_n\uc:=\{(c_{\sigma(1)},\ldots, c_{\sigma(n)})\in \mathbb Q^n | \sigma\in S_n\}.
\]

\begin{definition}\label{Def_tildeExt} Let $r,s\in \N$, $\uc\in \cN(r)$ and $\ud\in \cN(s)$. Set $n=r+s$ and define the following two subsets $\Ext^1(\uc,\ud)$ and $\widetilde{\Ext}^1(\uc,\ud)$ of $\cN(n)$:
\begin{enumerate}
\item $\Ext^1(\uc, \ud)$ is the set of $\ua\in \cN(n)$ such that there exists a short exact sequence
\[
0\longrightarrow \cO(\ud)\longrightarrow \cO(\ua)\longrightarrow \cO(\uc)\longrightarrow 0,
\]
or equivalently, that $\cO(\ua)$ is an extension of $\cO(\uc)$ by $\cO(\ud)$;
\item $
\widetilde{\Ext}^1(\uc, \ud)$ is the set of $\ua\in \cN(n)$ satisfying the following condition: there exists a partition
\[
\{1,\ldots, n\}=H\coprod K
\]
of $\{1,\ldots, n\}$ with
\[
H=\{h_1<h_2<\ldots <h_r\} \quad \textrm{and}\quad K=\{k_1<k_2\ldots<k_s\},
\]
and $\ub\in \mathbb Q^n$ with
\[
(b_{h_1},\ldots, b_{h_r})\in S_r\uc\quad  \textrm{and} \quad (b_{k_1},\ldots,b_{k_s})\in S_s\ud,
\]
such that the following properties hold:
\begin{enumerate}
\item for any $i\in H$, $b_i\geq a_i$;
\item for any $i\in K$, $b_i\leq a_i$;
\item for any $1\leq l\leq n$, $\sum_{i=1}^{l}b_i\geq \sum_{i=1}^{l}a_i$, with equality if $l=n$.
\end{enumerate}
\end{enumerate}
\end{definition}

\begin{remark}\begin{enumerate}
\item The combinatorial definition of $\widetilde{\Ext}^1$ is motivated by \cite{Schl} in which the extensions of vector bundles on $\mathbb{P}^1$ over an algebraically closed field are classified in a similar way. After we finish our work, we noticed that similar condition is also considered independently by Hong in \cite{Ho2}.
\item In the definition of $\widetilde{\Ext}^1$, if an element $\ub\in \mathbb Q^n$ satisfies the conditions (2.a)-(2.c) above and if we take $\tilde{\ub}\in \mathbb Q^n$ such that
\[
(\tilde{\ub}_{h_1},\ldots,\tilde{\ub}_{h_r})=(c_1,\ldots,c_r), \quad \textrm{and}\quad  (\tilde{\ub}_{k_1},\ldots,\tilde{\ub}_{k_s})=(d_1,\ldots,d_s),
\]
then $\tilde{\ub}$ satisfies equally the conditions (2.a)-(2.c) above.
\end{enumerate}
\end{remark}

These two subsets of $\cN(n)$ are closely related by the following proposition:

\begin{proposition}\label{Prop_tildeExt is Necessary condition} For $r,s\in \mathbb N$, $\uc\in \cN(r)$ and $\ud\in \cN(s)$, we have $\Ext^1(\uc,\ud)\subset \widetilde{\Ext}^1(\uc,\ud)$.
\end{proposition}

\begin{proof}Take $\ua\in \Ext^1(\uc,\ud)$. Let $N\in \mathbb N$ such that all the components of $N\ua, N\uc,N\ud$ are integers. Consider the cyclic covering $f:X_N\longrightarrow X$ of order $N$ between Fargues-Fontaine curves.
It follows that $f^*\cO(\ua),f^*\cO(\uc)$ and $f^*\cO(\ud)$ are sum of line bundles: for example, if $\ua=(a_1,\ldots,a_n)\in \cN(n)\subset \mathbb Q^n$, $f^*\cO(\ua)$ is the direct sum of line bundles $\cO(Na_i)$, $1\leq i\leq n$. Then the same argument as in the proof of \cite[Proposition 3.1]{Schl} shows that $N\ua\in \widetilde{\Ext}^1(N\uc,N\ud)$, which implies $\ua\in \widetilde{\Ext}^1(\uc,\ud)$.
\end{proof}

Let $\uc\in \cN(r)$, $\ud\in \cN(s)$. Proposition \ref{Prop_tildeExt is Necessary condition} gives a necessary condition for the property that $\cO(\ua)$ can be realized as an extension of $\cO(\uc)$ by $\cO(\ud)$. One could find some other conditions in the literature for the last property. For example, if $\uc$ and $\ud$ are semistable, one can reformulate the conditions in Definition \ref{Def_tildeExt} (2) in a more direct way, which allows us to relate the set $\widetilde{\Ext}^1(\uc,\ud)$ with the description of extensions in \cite{BFH}.

\begin{lemma} Let $\uc\in \cN(r)$, $\ud\in \cN(s)$ and $\ua\in \cN(n)$. Assume $n=r+s$, and that $\uc$ and $\ud$ are both semistable. Set $c:=c_1=\cdots=c_r$ and $d:=d_1=\ldots=d_s$. Then $\ua\in \widetilde{\Ext}^1(\uc,\ud)$ if and only if one of the following two conditions is verified:
\begin{enumerate}
\item $c\leq d$ and $\ua=\uc\oplus \ud$; or
\item $c>d$, and $\ua\leq \uc\oplus \ud$.
\end{enumerate}
\end{lemma}

\begin{proof} If one of the two conditions is verified, it is clear that $\ua\in \widetilde{\Ext}^1(\uc,\ud)$: we can take $H=\{1<\ldots <r\}$ and $K=\{r+1<\ldots <r+s=n\}$ in both cases.

Conversely, suppose that $\ua\in \widetilde{\Ext}^1(\uc,\ud)$, with $H=\{h_1<\ldots<h_r\}$ and $K=\{k_1<\ldots<k_s\}$ a partition of $\{1,\ldots, n\}$ given in Definition \ref{Def_tildeExt}.  We shall distinguish two different cases:
\begin{itemize}
\item Assume first $c\leq d$. If $h_1<k_s$, by the conditions (2.a) and (2.b) of Definition \ref{Def_tildeExt} we have
\[
c=b_{h_1}\geq a_{h_1}\geq a_{k_s}\geq b_{k_s}=d.
\]
So we get $c=d$. Combing (2.c) in Definition \ref{Def_tildeExt}, we deduce moreover $a_1=\ldots=a_n$, and thus $\ua=\uc\oplus \ud$ as asserted in (1). If $h_1>k_s$, then $H=\{s+1<\ldots< n\}$ and $K=\{1<\ldots<s\}$. By (2.c) of Definition \ref{Def_tildeExt}, for all $1\leq l\leq s$, we have
\[
\sum_{i=1}^ld_i=\sum_{i=1}^{l}b_i\geq \sum_{i=1}^la_i
\]
while $d_i=b_i\leq a_i$ for all $1\leq i\leq s$ by (2.b) of Definition \ref{Def_tildeExt}, therefore $a_1=\ldots=a_s=d$. On the other hand, as
\[
\sum_{i=1}^sd_i+\sum_{i=s+1}^nc_i=\sum_{i=1}^{n}b_i\geq \sum_{i=1}^la_i=\sum_{i=1}^sa_i+\sum_{i=s+1}^na_i,
\]
we find
\[
\sum_{i=s+1}^nc_i\geq \sum_{i=s+1}^na_i.
\]
But $c_i=b_i\geq a_i$ for every $s+1\leq i\leq n$ by (2.a) of Definition \ref{Def_tildeExt}. Thus $a_i=c_i=c$ for all $s+1\leq i\leq n$. Consequently we still obtain $\ua=\uc\oplus \ud$ as asserted in (1).

\item It remains to consider the case where $c>d$. Necessarily $a_1\leq c$: otherwise $1\in K$ and we obtain $d=b_1\leq a_1\leq c$ which is absurd. Similarly $a_n\geq d$. Since $\sum_{i=1}^na_i=\sum_{i=1}^nb_i$ by (2.c) of Definition \ref{Def_tildeExt}, it follows that $\uc\oplus \ud\geq \ua$, as claimed by (b).
\end{itemize}
This completes the proof of our lemma.
\end{proof}

\begin{corollary}\label{coro_Hansen} Let $\uc\in \cN(r)$ and $\ud\in \cN(s)$. Assume that $\uc$ and $\ud$ are both semistable. Then $\Ext^1(\uc,\ud)=\widetilde{\Ext}^1(\uc,\ud)$.
\end{corollary}

\begin{proof} Let $\ua\in \widetilde{\Ext}^1(\uc,\ud)$. We must check that $\cO(\ua)$ is an extension of $\cO(\uc)$ by $\cO(\ud)$. Set $c:=c_1=\ldots=c_r$ and $d:=d_1=\ldots=d_s$. According to the lemma above, we only need to consider the following two cases.
\begin{itemize}
\item If $c\leq d$ and $\ua=\uc\oplus \ud$, then $\cO(\ua)=\cO(\uc)\oplus \cO(\ud)$, which is the split extension of $\cO(\uc)$ by $\cO(\ud)$.
\item If $c>d$ and $\ua\leq \uc\oplus \ud$, then the vector bundles $\mathcal F_1:=\cO(\ud)$, $\mathcal F_2:=\cO(\uc)$, and $\mathcal E:=\cO(\ua)$ satisfy the condition of \cite[Theorem 1.1.2]{BFH}. So according to loc. cit., $\cO(\ua)$ is an extension of $\cO(\uc)$ by $\cO(\ud)$, as required.
\end{itemize}
\end{proof}

We have the following property of $\widetilde{\mathrm{Ext}}^1$:

\begin{lemma}\label{Lem_Hong}Let $\uc\in \cN(r)$, $\ud\in \cN(s)$ and $\ua\in \cN(n)$. Assume $n=r+s$, and $\ua\in \widetilde{\Ext}^1(\uc,\ud)$. Then the following assertions hold.
\begin{enumerate}
\item[(i)] $\ua$ strongly slopewise dominates $\ud$: for every $\mu\in \mathbb Q$,
\[
n_{\mu}:=\# \{a_i |a_i\geq \mu\} \geq \#\{d_i | d_i\geq \mu \},
\]
with equality if and only if $(a_1,\ldots, a_{n_{\mu}})=(d_1,\ldots, d_{n_{\mu}})$;
\item[(ii)] $\ua^{\vee}$ strongly slopewise dominates $\uc^{\vee}$, or equivalently, for every $\mu\in \mathbb Q$,
\[
m_{\mu}:=\# \{a_i |a_i\leq \mu\} \geq \#\{c_i | c_i\leq\mu \},
\]
with equality if and only if $(a_1,\ldots, a_{m_{\mu}})=(c_1,\ldots, c_{m_{\mu}})$; and
\item[(iii)] $\ua\leq \uc\oplus \ud$.
\end{enumerate}
\end{lemma}

\begin{proof} Suppose for some $\mu\in\mathbb Q$,
\[
n_{\mu}=\# \{a_i |a_i\geq \mu\} \leq \#\{d_i | d_i\geq\mu \}.
\]We may assume without lossing generality that $\mu=a_i$ or $d_i$ for some $i$. By (2.b) of Definition \ref{Def_tildeExt}, the inequality above must be an equality and $\{1,\ldots, n_{\mu}\}\subset K$. Therefore $(a_1,\ldots, a_{n_{\mu}})=(d_1,\ldots, d_{n_{\mu}})$, showing (i) above. For (ii), note that
\[
\ua\in \widetilde{\Ext}^1(\uc, \ud)\Longleftrightarrow \ua^\vee\in \widetilde{\Ext}^1(\ud^\vee, \uc^{\vee}).
\]
So (ii) is just the dual version of (i).

It remains to check (iii). Write $\uc\oplus \ud=(b_1',\ldots, b_n')\in \cN(n)$. So the $b_1'\geq \ldots\geq b_n'$ is just the permutation by order of $b_1,\ldots, b_n$. So, combing (2.c) of Definition \ref{Def_tildeExt}, we obtain
\[
\sum_{i=1}^lb_i'\geq \sum_{i=1}^{l}b_i\geq \sum_{i=1}^la_i, \quad \textrm{for all } 1\leq l\leq n.
\]
In other words, $\ua\leq \uc\oplus \ud$, as claimed by (iii).
\end{proof}

\begin{remark}\label{rem:compared-with-Hong} We keep the notations of Lemma \ref{Lem_Hong}.
\begin{enumerate}
\item According to \cite{Ho3},  the condition that $\ua$ strongly slopewise dominates $\ud$ is equivalent to the fact that $\cO(\ud)$ is a subbundle (i.e. locally direct factor) of $\cO(\ua)$ and the condition that $\ua^{\vee}$ strongly slopewise dominates $\uc^{\vee}$ is equivalent to the fact that $\cO(\uc)$ is a quotient of the vector bundle $\cO(\ua)$.
\item  In general, as shown by the following example, the combination of the above conditions (i)-(iii) is weaker than the conditions defining $\widetilde{\Ext}^1(\uc,\ud)$ even when one of  $\ua$ and $\ub$ is semistable. Consider
\[
\ua=(6,5,2,1)\in \cN(4), \quad \uc=(10,4)\in \cN(2), \quad \textrm{and} \quad \ud=(0,0)\in \cN(2).
\]
Then the triple $(\ua,\uc,\ud)$ does verify the conditions (i)-(iii) above. But $\ua\notin \widetilde{\Ext}^1(\uc,\ud)$. Otherwise, let $\{1,2,3,4\}=H\coprod K$ be the partition given in Definition \ref{Def_tildeExt}. Since $0<6$, we must have $1\in H$ and thus $b_1=c_1=10$. Then as
\[
a_1+a_2=11 >b_1+0,
\]
by (2.c) of Definition \ref{Def_tildeExt}, $2\in H$ and thus $b_2=c_2=4$. But this contradicts to the fact that $b_2\geq c_2$. In particular, by Proposition \ref{Prop_tildeExt is Necessary condition}, $\cO(\ua)=\cO(6)\oplus \cO(5)\oplus \cO(2)\oplus \cO(1)$ is not an extension of $\cO(\uc)=\cO(10)\oplus \cO(4)$ by $\cO(\ud)=\cO\oplus \cO$.
\end{enumerate}
\end{remark}

One would probably expect that
\[
\Ext^1(\uc,\ud)=\widetilde{\Ext}^1(\uc,\ud),
\]
or equivalently, for $\ua\in \widetilde{\Ext}^1(\uc,\ud)$, there exists a short exact sequence of the form
\[
0\longrightarrow \cO(\ud)\longrightarrow \cO(\ua)\longrightarrow \cO(\uc)\longrightarrow 0.
\]
However this fails in general by the following example.

\begin{example}\label{ex:counter-example-for-ext} Let
\[
\ua=\left(1,\frac{5}{7}^{(7)}, \frac{4}{7}^{(7)},0\right), \quad \uc=\left(3,\frac{3}{5}^{(5)}\right), \quad \textrm{and}\quad \ud=\left(\frac{5}{9}^{(9)}, -1\right).
\]
In particular, $\ua\in \cN(16)$, $\uc\in \cN(6)$ and $\ud\in \cN(10)$. Then we have $\ua\in \widetilde{\Ext}^1(\uc,\ud)$, with
\[
H=\{1,9,10,11,12,13\}\subset \{1,2,\ldots,16\}.
\]
However, $\cO(\ua)$ is not an extension of $\cO(\uc)$ by $\cO(\ud)$, or equivalently, there does not exist a short exact sequence as follows
\[
0\longrightarrow \cO\left(\frac{5}{9}\right)\oplus \cO(-1)\stackrel{\phi}{\longrightarrow} \cO(1)\oplus \cO\left(\frac{5}{7}\right)\oplus \cO\left(\frac{4}{7}\right)\oplus \cO\stackrel{\psi}{\longrightarrow}\cO(3)\oplus \cO\left(\frac{3}{5}\right)\longrightarrow 0
\]
Suppose that such an extension exists. Consider the subbundle $\cO(\ua')\subset \cO(\ua)$, with
\[
\ua'=\left(1,\frac{5}{7}^{(7)}, \frac{4}{7}^{(7)}\right)\in \cN(15).
\]
Write $\cO(\uc'):=\psi(\cO(\ua'))\subset \cO(\ua)$ and $\cO(\ud')=\phi^{-1}(\cO(\ua'))$. So we have the following commutative diagram with exact rows
\[
\xymatrix{0\ar[r]& \mathrm{Coker}(\alpha')\ar[r] & \cO\ar[r] & \mathrm{Coker}(\alpha'')\ar[r] & 0 \\
0\ar[r] & \cO\left(\frac{5}{9}\right)\oplus \cO(-1)\ar[r]^<<<<<{\phi} \ar[u]& \cO(1)\oplus \cO\left(\frac{5}{7}\right)\oplus \cO\left(\frac{4}{7}\right)\oplus \cO\ar[r]^<<<<<{\psi} \ar[u]& \cO(3)\oplus \cO\left(\frac{3}{5}\right)\ar[r] \ar[u]&  0 \\ 0\ar[r] & \cO(\ud')\ar[r] \ar[u]^{\alpha'}& \cO(\ua')\ar[r]\ar[u]^{\alpha} & \cO(\uc')\ar[r]\ar[u]^{\alpha''} & 0}
\]
We claim that $
\cO(\ud')=\cO(\frac{5}{9})$. To see this, observe first that $\alpha'$ is not surjective: otherwise $\cO\stackrel{\simeq}{\ra}\mathrm{Coker}(\alpha'')$ and thus $\cO$ would be a direct factor of $\cO(3)\oplus \cO(\frac{3}{5})$, which is absurd. So $\mathrm{Coker}(\alpha')$ is a line bundle contained in $\cO$, and its degree $\leq 0$. In particular, there is no non-zero morphism $\cO(\frac{5}{9})\ra \mathrm{Coker}(\alpha')$. As a result, being a quotient of $\cO(\frac{5}{9})\oplus \cO(-1)$, $\mathrm{Coker}(\alpha')\simeq \cO(-1)$, and thus $\cO(\ud')=\cO(\frac{5}{9})\subset \cO(\ud)$. Furthermore, we claim that $\cO(\uc')=\cO(2)\oplus \cO(\frac{3}{5})$. As $\mathrm{Coker}(\alpha'')$ is torsion, $\alpha''$ is generically an isomorphism. It follows that the induced map
\[
\cO(\ua')\stackrel{\psi}{\longrightarrow} \cO(3)\oplus \cO\left(\frac{3}{5}\right)\longrightarrow \cO\left(\frac{3}{5}\right)
\]
is generically an epimorphism, hence must be an epimorphism since $\frac{2}{5}<\frac{4}{7}$. In particular, the composition
\[
\beta:\cO(\uc')\stackrel{\alpha''}{\longrightarrow} \cO(3)\oplus \cO\left(\frac{3}{5}\right)\longrightarrow \cO\left(\frac{3}{5}\right)
\]
must be surjective. On the other hand, $\cO(\uc')$ is a vector bundle of rank $6$ and of degree $5$, so $\ker(\beta)\simeq \cO(2)$. As $2>\frac{3}{5}$, we get $\cO(\uc')=\cO(2)\oplus \cO(\frac{3}{5})$. So
\[
\uc'=\left(2,\frac{3}{5}^{(5)}\right), \quad \textrm{and}\quad \ud'=\left(\frac{5}{9}^{(9)}\right).
\]
Moreover, $\cO(\ua')$ is an extension of $\cO(\uc')$ by $\cO(\ud')$, so $\ua'\in \widetilde{\Ext}^1(\uc',\ud')$. But one could check directly that this is not the case. In fact, since
\[
2>1>\frac{5}{9}>\frac{5}{7}>\frac{3}{5}>\frac{4}{7}
\]
if $\ua'\in \widetilde{\Ext}^1(\uc',\ud')$, the corresponding partition $\{1,\ldots, 15\}=H'\coprod K'$ would satisfy $1\in H'$ and $\{2,3,\ldots,8\}\subset K'$. But this is impossible as
\[
2+\frac{5}{9}\times 7 < 1+ \frac{5}{7}\times 7=6.
\]
\end{example}

Next we want to give an inductive criterion for an element $\ua\in \widetilde{\Ext}^1(\uc,\ud)$ to be contained in $\Ext^1(\uc,\ud)$.

\begin{proposition} \label{Prop_classification of extension by Ext^1} Let $\uc\in \cN(r)$, $\ud\in \cN(s)$. Let $n:=r+s$, and $\ua\in \widetilde{\Ext}^1(\uc,\ud)$. Write $c_r=q/p$ with $p\in \mathbb Z_{\geq 1}$ and $q \in \mathbb Z$ such that $(p,q)=1$, so $\uc=\left(c_1,\ldots, c_{r-p}, \frac{q}{p},\ldots,\frac{q}{p}\right)$. Set
\[
\uc':=(c_1,\ldots, c_{r-p})\in \cN(r-p),  \quad \textrm{and}\quad
\uc''=c_r^{(p)}:=\left(\frac{q}{p},\ldots,\frac{q}{p}\right)\in \cN(p).
\]
Then $\ua\in \Ext^1(\uc,\ud)$ if and only if there exists some $\ue\in \cN(n-p)$ such that
\begin{itemize}
\item $\ue\in \Ext^1(\uc',\ud)$; and
\item $\ua\in \Ext^1(\uc'',\ue)$, or equivalently, $\ua\in \widetilde{\Ext}^1(\uc'',\ue)$: see Proposition \ref{Prop_one_semistable} below or \cite[Theorem 1.1]{Ho2}.
\end{itemize}
\end{proposition}

\begin{remark}Proposition \ref{Prop_classification of extension by Ext^1} is also proved independently by Hong in \cite[Theorem 1.2]{Ho2}.
\end{remark}

\begin{proof} Assume first that $\ua\in \Ext^{1}(\uc,\ud)$. So we have a short exact sequence of vector bundles over the Fargues-Fontaine curve $X$:
\[
0\longrightarrow \cO(\ud)\longrightarrow \cO(\ua)\longrightarrow \cO(\uc)\longrightarrow 0.
\]
Let $\E\subset \cO(\ua)$ be the inverse image of the subbundle $\cO(\uc')\subset \cO(\uc)$ by the surjective morphism $\cO(\ua)\ra \cO(\uc)$, and write $\E=\cO(\ue)$. Then $\E=\cO(\ue)$ is an extension of $\cO(\uc')$ by $\cO(\ud)$, or equivalently, $\ue\in \Ext^1(\uc',\ud)$. Furthermore, by Snake Lemma,
\[
\cO(\ua)/\E\stackrel{\simeq}{\longrightarrow}\cO(\uc)/\cO(\uc')\simeq \cO(\uc'').
\]
So $\cO(\ua)$ is an extension of $\cO(\uc'')$ by $\E=\cO(\ue)$. In other words, $\ua\in \Ext^1(\uc'',\ue)$.

Conversely, suppose that there exists $\ue\in \cN(n-p)$ such that  $\ue\in \Ext^1(\uc',\ud)$, and that $\ua\in \Ext^1(\uc'',\ue)$. Then we get two short exact sequences of vector bundles
\[
0\longrightarrow \cO(\ud)\longrightarrow \cO(\ue)\longrightarrow \cO(\uc')\longrightarrow 0,
\]
and
\[
0\longrightarrow \cO(\ue)\longrightarrow \cO(\ua)\longrightarrow \cO(\uc'')\longrightarrow 0.
\]
We can insert them into the following commutative diagram
\[
\xymatrix{& & 0& &  \\ & & \cO(\uc'')\ar[u] & &  \\  0\ar[r] & \cO(\ud)\ar[r] & \cO(\ua)\ar[u] \ar@{.>}[r]& \mathcal G \ar@{.>}[r]& 0  \\ 0\ar[r] & \cO(\ud)\ar[r]\ar@{=}[u] & \cO(\ue)\ar[r] \ar[u]& \cO(\uc')\ar[r]\ar@{.>}[u] & 0 \\ & & 0\ar[u]& &   },
\]
where $\mathcal G$ is the cokernel of the injective morphism $\cO(\ud)\ra \cO(\ua)$. By Snake Lemma, the coherent sheaf $\mathcal G$ on the Fargues-Fontaine curve $X$ can be put into the following short exact sequence
\[
0\longrightarrow \cO(\uc')\longrightarrow \mathcal G\longrightarrow \cO(\uc'')\longrightarrow 0.
\]
As $c_{r-p}\geq c_r$, all the extension of $\cO(\uc'')$ by $\cO(\uc')$ splits. In particular, $\mathcal G$ is again a vector bundle, and $\mathcal G=\cO(\uc')\oplus \cO(\uc'')=\cO(\uc)$. It follows that $\ua\in \Ext^1(\uc,\ud)$.
\end{proof}

Proposition \ref{Prop_classification of extension by Ext^1} can be used as an algorithm to compute $\Ext^1(\uc, \ud)$ inductively on the number of stable blocks in $\uc\oplus\ud$ using duality combined with Corollary \ref{coro_Hansen}. 

\begin{example}\label{ex:inductive-computation-of-ext0} We have 
\[
\Ext^1\left(\left(0,-\frac{1}{6}^{(6)}\right),\left(-\frac{1}{3}^{(3)}\right)\right)=\left\{\substack{\left(-\frac{1}{5}^{(10)}\right),\left(-\frac{1}{6}^{(6)},-\frac{1}{4}^{(4)}\right), \left(0,-\frac{2}{9}^{(9)}\right)\\ \left(0,-\frac{1}{5}^{(5)},-\frac{1}{4}^{(4)}\right), \left(0,-\frac{1}{6}^{(6)},-\frac{1}{3}^{(3)}\right)}\right\}.
\]
Indeed, by Proposition \ref{Prop_classification of extension by Ext^1}, we have 
\[
\Ext^1\left(\left(0,-\frac{1}{6}^{(6)}\right),\left(-\frac{1}{3}^{(3)}\right)\right)=\bigcup_{\ue\in \Ext^1\left((0),\left(-\frac{1}{3}^{(3)}\right)\right)} \Ext^1\left(\left(-\frac{1}{6}^{(6)}\right),\ue\right).
\]
Using the main result of \cite{BFH} or Corollary \ref{coro_Hansen} above, we have 
\begin{equation}\label{eq:choice-for-ue}
\Ext^1\left((0),\left(-\frac{1}{3}^{(3)}\right)\right)=\left\{\left(-\frac{1}{4}^{(4)}\right),\left(0,-\frac{1}{3}^{(3)}\right)\right\}.
\end{equation}
So it remains to compute, for $\ue$ one of the two elements in \eqref{eq:choice-for-ue}, the set $\Ext^1\left(\left(-\frac{1}{6}^{(6)}\right),\ue\right)$:
\begin{itemize}
    \item if $\ue=\left(-\frac{1}{4}^{(4)}\right)$, again by the main result of \cite{BFH},
    \[
    \Ext^1\left(\left(-\frac{1}{6}^{(6)}\right),\ue\right)=\left\{\left(-\frac{1}{5}^{(10)}\right),\left(-\frac{1}{6}^{(6)},-\frac{1}{4}^{(4)}\right)\right\};
    \]
    \item if $\ue=\left(0,-\frac{1}{3}^{(3)}\right)$, then as $0>-1/6$ any extension of $\cO(-1/6)$ by $\cO$ splits, so we reduce to computing $\Ext^1\left(\left(-\frac{1}{6}^{(6)}\right),\left(-\frac{1}{3}^{(3)}\right)\right)$, and finally we get 
    \[
    \Ext^1\left(\left(-\frac{1}{6}^{(6)}\right),\ue\right)=\left\{\left(0,-\frac{2}{9}^{(9)}\right),\left(0,-\frac{1}{5}^{(5)},-\frac{1}{4}^{(4)}\right),\left(0,-\frac{1}{6}^{(6)},-\frac{1}{3}^{(3)}\right)\right\}.
    \]
\end{itemize}
\end{example}

\begin{remark}\label{rem:direct-computation-of-ext0} As the vector bundle $\cO(-1/3)$ is (semi-)stable, by Proposition \ref{Prop_one_semistable} we have  
\[
\Ext^1\left(\left(0,-\frac{1}{6}^{(6)}\right),\left(-\frac{1}{3}^{(3)}\right)\right)=\widetilde{\Ext}^1\left(\left(0,-\frac{1}{6}^{(6)}\right),\left(-\frac{1}{3}^{(3)}\right)\right).
\]
So we can also compute the extension set in Example \ref{ex:inductive-computation-of-ext0} by exploring directly the combinatorial condition in Definition \ref{Def_tildeExt}.
\end{remark}


\subsection{Applications} As we see at the end of last section, the key for verifying if a single Newton stratum is not contained in the weakly admissible locus is the existence of certain particular extensions of vector bundles on the Fargues-Fontaine curve. In this \S, we shall give further applications of Theorem \ref{thm:criterion-for-a-singule-stratum} for the general linear group $G=\mathrm{GL}_n$. The new input here is our discussions in the previous subsection, which allows us to handle more complicated extensions of vector bundles. In the following, to simplify the notation, for an element 
\[
v=(\lambda_{1}^{(n_1)},\ldots,\lambda_s^{(n_s)})\in \mathcal N(n), \quad \textrm{with} \quad  \lambda_1>\lambda_2>\ldots>\lambda_s, \textrm{ and } n_i\in \mathbb Z_{>0},  
\]
if the integer $\lambda_in_i$ is coprime to $n_i$, then we omit the exponent $(n_i)$ from the notation.

\subsubsection{} We take $G=\mathrm{GL}_{14}$, $\mu=(1^{(6)},0^{(8)})$, and $b\in B(G)_{basic}$ with $\nu_b=\left(\frac{3}{7}^{(14)}\right)\in \mathcal N(14)$. Then $b$ has a reduction $b_M$ to only one proper maximal standard Levi subgroup 
\[
M=\mathrm{GL}_7\times \mathrm{GL}_7\hookrightarrow \mathrm{GL}_{14}.
\] 
The cocharacters $w\mu$, with $w$ an element in the Weyl group, such that $\langle \nu_b-w\mu,\chi\rangle<0$ for some $\chi\in X^*(P_1/Z_G)^+$ are 
\begin{equation}\label{eq:choice-for-wmu}
w\mu\in \left\{\substack{\left((1^{(4)},0^{(3)}),(1^{(2)},0^{(5)})\right), \  \left((1^{(5)},0^{(2)}),(1,0^{(6)})\right)\\  \left((1^{(6)},0),(0^{(7)})\right)}\right\} \subset \mathcal N(M).
\end{equation}
Thus
\[
\nu_{b_{M}}-w\mu \in \left\{\substack{\left(\left(\frac{3}{7}^{(3)},-\frac{4}{7}^{(4)}),(\frac{3}{7}^{(5)},-\frac{4}{7}^{(2)}\right)\right), \  \left(\left(\frac{3}{7}^{(2)},-\frac{4}{7}^{(5)}),(\frac{3}{7}^{(6)},-\frac{4}{7}\right)\right)\\  \left(\left(\frac{3}{7},-\frac{4}{7}^{(6)}\right),\left(\frac{3}{7}^{(7)}\right)\right)}\right\} \subset \mathcal N(M).
\]
We want to describe explicitly the generalized Kottwitz set 
\[
B(M,(k_1,k_2),\nu_{b_M}-w\mu)=B(\mathrm{GL}_7,k_1,v_1)\times B(\mathrm{GL}_7,k_2,v_2)
\]
for the rational cocharacter $\nu_{b}-\mu^w=v_1\times v_2\in \mathcal N(M)=\mathcal N(7)\times \mathcal N(7)$ as above (here $k_i:=|v_i|$). It suffices to do this for $B(\mathrm{GL}_7,k_i,v_i)$ ($i=1,2$) respectively.
\begin{itemize}
\item[(i)] \underline{$w\mu=((1^{(4)},0^{(3)}),(1^{(2)},0^{(5)}))$}. So $v_1=((3/7)^{(3)},(-4/7)^{(4)})$ and $v_2=((3/7)^{(5)},(-4/7)^{(2)})$. The generalized Kottwitz set $B\left(\mathrm{GL}_7,-1,\left((3/7)^{(3)},(-4/7)^{(4)}\right)\right)\subset \mathcal N(7)$ have $7$ elements:
\[
\left(\frac{1}{3}^{(3)}, -\frac{1}{2}^{(4)}\right); \quad \left(0^{(i)}, -\frac{1}{7-i}\right), \ 0\leq i\leq 5,
\]
and $B(\mathrm{GL}_7,1,((3/7)^{(5)},(-4/7)^{(2)}))\subset \mathcal N(7)$ consists of the following $6$ elements:
\[
\left(\frac{2}{5}^{(5)}, -\frac{1}{2}^{(2)}\right); \quad \left(\frac{1}{7-i}^{(7-i)}, 0^{(i)}\right), \ 0\leq i\leq 4. 
\]

\item[(ii)] \underline{$w\mu=((1^{(5)},0^{(2)}),(1,0^{(6)}))$}. So $v_1=((3/7)^{(2)},(-4/7)^{(5)})$ and $v_2=((3/7)^{(6)},-4/7)$. The generalized Kottwitz set $B\left(\mathrm{GL}_7,-2,\left((3/7)^{(2)},(-4/7)^{(5)}\right)\right)\subset \mathcal N(7)$ have $8$ elements:
\[
\left(0^{(i)}, -\frac{2}{7-i}^{(7-i)}\right),\ 0\leq i\leq 3;  \quad \left(0^{(j)}, -\frac{1}{5-j}^{(5-j)},-\frac{1}{2}^{(2)}\right), \ 0\leq j\leq 2; \quad \left(-\frac{1}{4}^{(4)},-\frac{1}{3}^{(3)}\right), 
\]
and $B(\mathrm{GL}_7,2,((3/7)^{(6)},-4/7))\subset \mathcal N(7)$ consists of the following $6$ elements:
\[
\left(\frac{2}{5}^{(5)}, 0^{(2)}\right); \quad \left(\frac{1}{3}^{(6)}, 0\right); \quad \left(\frac{1}{3}^{(3)}, \frac{1}{4}^{(4)}\right); \quad \left(\frac{2}{7}^{(7)}\right). 
\]

\item[(iii)] \underline{$w\mu=((1^{(6)},0),(0^{(7)}))$}. So $v_1=(3/7,(-4/7)^{(6)})$, and $v_2=((3/7)^{(7)})$. The generalized Kottwitz set $B\left(\mathrm{GL}_7,-3,\left(3/7,(-4/7)^{(6)}\right)\right)\subset \mathcal N(7)$ have $4$ elements:
\[
\left(0, -\frac{1}{2}^{(6)}\right), \quad \left(-\frac{1}{3}, -\frac{1}{2}^{(4)}\right);\quad \left(-\frac{2}{5}^{(5)},-\frac{1}{2}^{(2)}\right);\quad \left(-\frac{3}{7}^{(7)}\right),
\]
and $B(\mathrm{GL}_7,3,((3/7)^{(7)}))\subset \mathcal N(7)$ consists of one single element $((3/7)^{(7)})$. 
\end{itemize}

We would like to give a complete list of the non-empty Newton strata
\[
\mathcal F(\mathrm{GL}_{15},\mu,b)^{[b']}, \quad [b']\in B(G,0,\nu_b-w_0\mu)
\]
that are contained in the weakly admissible locus $\mathcal F(\mathrm{GL}_{14},\mu,b)^{wa}$. As before, it is enough to consider those $[b']\in B(G,0,\nu_b-w_0\mu)$ which are HN-indecomposable relative to $\nu_b-w_0\mu$. By Theorem \ref{thm:criterion-for-a-singule-stratum}, the Newton stratum $\mathcal F(G,\mu,b)^{[b']}\not\subset \mathcal F(G,\mu,b)^{wa}$ if and only if the vector bundle $\mathcal E_{b'}$ can be written as an extension of the form 
\begin{equation}\label{eq:extension-for-GL14}
    0\longrightarrow \mathcal E_1\longrightarrow \mathcal E_{b'}\longrightarrow \mathcal E_2\longrightarrow 0,
\end{equation}
with $\mathcal E_1$ and $\mathcal E_2$ two vector bundles, or equivalently, if 
\[
-\nu_{b'}\in \Ext^1(v_{\mathcal E_2},v_{\mathcal E_1})\subset \mathcal N(14), 
\]
such that the pair $(-v_{\mathcal E_1},-v_{\mathcal E_2})\in 
B(M,(k_1,k_2), \nu_{b_M}-w\mu)$ for the cocharacters $w\mu$ in \eqref{eq:choice-for-wmu}. 
Write 
\[
\mathsf E=\bigcup_{(-v_{\mathcal E_1},-v_{\mathcal E_2})\in 
B(M,(k_1,k_2), \nu_{b_M}-w\mu) \atop  \textrm{for a cocharacter } w\mu \textrm{ in } \eqref{eq:choice-for-wmu}. } \Ext^1(v_{\mathcal E_2},v_{\mathcal E_1})\subset \mathcal N(14)
\]

\if false
\begin{example}\label{ex:inductive-computation-of-ext} We have 
\[
\Ext^1\left(\left(0,-\frac{1}{6}\right),\left(\frac{1}{2}^{(2)},-\frac{1}{3}\right)\right)=\left\{\substack{\left(\frac{1}{2}^{(2)},-\frac{1}{5}^{(2)}\right),\left(\frac{1}{2}^{(2)},-\frac{1}{6},-\frac{1}{4}\right), \left(\frac{1}{2}^{(2)},0,-\frac{2}{9}\right)\\ \left(\frac{1}{2}^{(2)},0,-\frac{1}{5},-\frac{1}{4}\right), \left(\frac{1}{2}^{(2)},0,-\frac{1}{6},-\frac{1}{3}\right)}\right\}.
\]
Indeed, as $1/2$ is bigger than $0$ and $-1/6$, there is no non-trivial extension of $\cO\oplus \cO(-1/6)$ by $\cO(1/2)^2$. So, if a vector bundle $\mathcal E$ is an extension of $\cO\oplus \cO(-1/6)$ by $\cO(1/2)^2\oplus \cO(-1/3)$, then $
\mathcal E=\cO(1/2)^2\oplus \mathcal E'$ with $\mathcal E'$ an extension of $\cO\oplus \cO(-1/6)$ by $\cO(-1/3)$. So it suffices to compute \[
\Ext^1\left(\left(0,-\frac{1}{6}\right),\left(-\frac{1}{3}\right)\right).
\] 
Let $\mathcal E'$ be an extension of $\cO\oplus \cO(-1/6)$ by $\cO(-1/3)$. By Proposition \ref{Prop_classification of extension by Ext^1}, we have 
\[
\Ext^1\left(\left(0,-\frac{1}{6}\right),\left(-\frac{1}{3}\right)\right)=\bigcup_{\ue\in \Ext^1\left((0),\left(-\frac{1}{3}\right)\right)} \Ext^1\left(\left(-\frac{1}{6}\right),\ue\right).
\]
Using the main result of \cite{BFH} or Corollary \ref{coro_Hansen} above, we have 
\begin{equation}\label{eq:choice-for-ue}
\Ext^1\left((0),\left(-\frac{1}{3}\right)\right)=\left\{\left(-\frac{1}{4}\right),\left(0,-\frac{1}{3}\right)\right\}.
\end{equation}
So it remains to compute, for $\ue$ one of the two elements in \eqref{eq:choice-for-ue}, the set $\Ext^1\left(\left(-\frac{1}{6}\right),\ue\right)$:
\begin{itemize}
    \item if $\ue=\left(-\frac{1}{4}\right)$, again by the main result of \cite{BFH},
    \[
    \Ext^1\left(\left(-\frac{1}{6}\right),\ue\right)=\left\{\left(-\frac{1}{5}^{(2)}\right),\left(-\frac{1}{6},-\frac{1}{4}\right)\right\};
    \]
    \item if $\ue=\left(0,-\frac{1}{3}\right)$, then as $0>-1/6$ we reduce as above to computing $\Ext^1\left(\left(-\frac{1}{6}\right),\left(-\frac{1}{3}\right)\right)$, and finally we get 
    \[
    \Ext^1\left(\left(-\frac{1}{6}\right),\ue\right)=\left\{\left(0,-\frac{2}{9}\right),\left(0,-\frac{1}{5},-\frac{1}{4}\right),\left(0,-\frac{1}{6},-\frac{1}{3}\right)\right\}.
    \]
\end{itemize}
\end{example}
\fi

\begin{proposition}\label{prop:all-extensions} Let $v\in \mathcal N(14)$. Then $v\in \mathsf E$ if and only if one of the following holds: 
\begin{enumerate}
    \item $v=v_{\mathcal E_1\oplus \mathcal E_2}$, with $(-v_{\mathcal E_1},-v_{\mathcal E_2})\in B(M,(k_1,k_2), \nu_{b_M}-w\mu)$ for a certain cocharacter $w\mu$ in \eqref{eq:choice-for-wmu}; or 
    \item $v\in \mathcal N(14)$ is one of the following $10$ elements:
    \begin{eqnarray*}
    \begin{aligned}
    \left(\frac{1}{2}^{(4)},-\frac{1}{5}^{(10)}\right),& \left(\frac{1}{2}^{(4)},-\frac{1}{6}, -\frac{1}{4}\right), & \left(\frac{1}{2}^{(4)},0,-\frac{2}{9}\right), & \left(\frac{1}{2}^{(4)},0, -\frac{1}{5},-\frac{1}{4}\right), & \left(\frac{1}{2}^{(4)},0^{(2)}, -\frac{1}{4}^{(8)}\right) \\ \left(\frac{1}{3},\frac{1}{6},-\frac{2}{5}\right),& \left(\frac{1}{4},\frac{1}{5},-\frac{2}{5}\right), & \left(\frac{1}{3},\frac{1}{5},0,-\frac{2}{5}\right), & \left(\frac{2}{9},-\frac{2}{5}\right),&  \left(\frac{2}{5},0,-\frac{2}{5}\right).  
    \end{aligned}
    \end{eqnarray*}
\end{enumerate}
Moreover, the last $5$ elements are all contained in $\Ext^1\left(\left(\frac{1}{2},-\frac{2}{5}\right), \left(\frac{1}{6},0\right)\right)$.
\end{proposition}

\begin{proof} Observe that, for most of the pairs $(\mathcal E_1,\mathcal E_2)$ as above, the slopes of $\mathcal E_2$ are less or equal to those of $\mathcal E_1$, so the extension \eqref{eq:extension-for-GL14} is trivial and hence 
\[
\Ext^1(v_{\mathcal E_2},v_{\mathcal E_1})=\{v_{\mathcal E_1\oplus \mathcal E_2}\}\subset \mathcal N(14). 
\]
It remains for us to consider the pairs $(\mathcal E_1,\mathcal E_2)$ for which there exist non-trivial extensions. In other words, we only need to compute the following sets 
\[
\Ext^1\left(\left(\frac{1}{2},-\frac{2}{5}\right),\left(\frac{1}{2}^{(4)},-\frac{1}{3}\right)\right), \quad \Ext^1\left(\left(0^{(i)},-\frac{1}{7-i}\right),\left(\frac{1}{2}^{(4)},-\frac{1}{3}\right)\right), \ 0\leq i\leq 4, 
\]
and 
\[
\Ext^1\left(\left(\frac{1}{2},-\frac{2}{5}\right),\left(\frac{1}{7-j},0^{(j)}\right)\right), \ 0\leq j\leq 5.
\]
Using the inductive criterion in Proposition \ref{Prop_classification of extension by Ext^1}, we can determine explicitly these $12$ sets. For example, we have 
\[
\Ext^1\left(\left(0,-\frac{1}{6}\right),\left(\frac{1}{2}^{(4)},-\frac{1}{3}\right)\right)=\left\{\substack{\left(\frac{1}{2}^{(4)},-\frac{1}{5}^{(10)}\right),\left(\frac{1}{2}^{(4)},-\frac{1}{6},-\frac{1}{4}\right), \left(\frac{1}{2}^{(4)},0,-\frac{2}{9}\right)\\ \left(\frac{1}{2}^{(4)},0,-\frac{1}{5},-\frac{1}{4}\right), \left(\frac{1}{2}^{(4)},0,-\frac{1}{6},-\frac{1}{3}\right)}\right\}.
\]
Indeed, as $1/2$ is bigger than $0$ and $-1/6$, there is no non-trivial extension of $\cO\oplus \cO(-1/6)$ by $\cO(1/2)^2$. So, if a vector bundle $\mathcal E$ is an extension of $\cO\oplus \cO(-1/6)$ by $\cO(1/2)^2\oplus \cO(-1/3)$, then $
\mathcal E=\cO(1/2)^2\oplus \mathcal E'$ with $\mathcal E'$ an extension of $\cO\oplus \cO(-1/6)$ by $\cO(-1/3)$. So it suffices to compute \[
\Ext^1\left(\left(0,-\frac{1}{6}\right),\left(-\frac{1}{3}\right)\right),
\]
which is done in Example \ref{ex:inductive-computation-of-ext0}. The other extension sets can be computed in a similar way. To list all the non-trivial extensions, one just keeps in mind that a vector bundle $\mathcal E$ may be realized as an extension of $\mathcal E_2$ by $\mathcal E_1$ for different pairs of vector bundles $(\mathcal E_1,\mathcal E_2)$ as above. 
\end{proof}

\begin{remark} As in Example \ref{ex:inductive-computation-of-ext0} above (cf. Remark \ref{rem:direct-computation-of-ext0}), we can also prove Proposition \ref{prop:all-extensions} by using Proposition \ref{Prop_one_semistable}. 
\end{remark}

With the help of Proposition \ref{prop:all-extensions}, we can now easily list all the non-empty Newton strata that are contained in the admissible locus. In the sequel, we shall distinguish the following different cases according to the explicit form of $\nu_{b'}$.   
\begin{enumerate}
\item \underline{$\nu_{b'}=(\frac{3}{8},\ldots)$}. So $\nu_{b'}=(\frac{3}{8},-\frac{1}{2}^{(6)})$, and the only Newton stratum here is contained in the weakly admissible locus. 
\item \underline{$\nu_{b'}=(\frac{2}{5},\ldots )$}. We have $16$ non-empty Newton strata in this case, and none of them is contained in the weakly admissible locus.

\item \underline{$\nu_{b'}=(\frac{1}{3}^{(6)},\ldots )$}. We have $11$ non-empty Newton strata in this case, and $3$ of them are contained in the weakly admissible locus: $(\frac{1}{3}^{(6)},-\frac{1}{4}^{(8)})$, $\left(\frac{1}{3}^{(6)},-\frac{1}{5},-\frac{1}{3}\right)$ and $ \left(\frac{1}{3}^{(6)},-\frac{1}{6},-\frac{1}{2}\right)$.

\item \underline{$\nu_{b'}=(\frac{1}{3},\frac{1}{4},\ldots)$}. We find $8$ non-empty Newton strata in this case, and none of them is contained in the weakly admissible locus.   

\item \underline{$\nu_{b'}=(\frac{1}{3},\frac{1}{5},\ldots)$}. We find $5$ non-empty Newton strata in this case, and $4$ of them are contained in the weakly admissible locus: $(\frac{1}{3},\frac{1}{5},-\frac{1}{3}^{(6)})$, $\left(\frac{1}{3},\frac{1}{5},-\frac{1}{4},-\frac{1}{2}\right)$, $\left(\frac{1}{3},\frac{1}{5},0,-\frac{2}{5}\right)$, $ \left(\frac{1}{3},\frac{1}{5},0,-\frac{1}{3},-\frac{1}{2}\right)$.

\item \underline{$\nu_{b'}=(\frac{1}{3},\frac{1}{6},\ldots )$}. We have $3$ non-empty Newton strata in this case, and $2$ of them are in the wa locus: $
\nu_{b'}=\left(\frac{1}{3},\frac{1}{6},-\frac{2}{5}\right)$, $\left(\frac{1}{3},\frac{1}{6},-\frac{1}{3},-\frac{1}{2}\right)$.

\item \underline{$\nu_{b'}=(\frac{1}{3},\frac{1}{7},\ldots)$}. So $\nu_{b'}=(\frac{1}{3},\frac{1}{7},-\frac{1}{2}^{(4)})$ and this Newton stratum is not contained in the weakly admissible locus. 


\item \underline{$\nu_{b'}=(\frac{2}{7},...)$}. We have $8$ non-empty Newton strata in this case, and none of them is contained in the weakly admissible locus.

\item \underline{$\nu_{b'}=(\frac{1}{4}^{(8)},...)$}. We have $5$ non-empty Newton strata in this case, and all of them are contained in the weakly admissible locus. 
 
\item \underline{$\nu_{b'}=(\frac{1}{4},\frac{1}{5},\ldots)$}. We have $3$ non-empty Newton strata in this case, and $2$ of them are contained in the weakly admissible locus: $(\frac{1}{4},\frac{1}{5},-\frac{2}{5})$ or $(\frac{1}{4},\frac{1}{5},-\frac{1}{3},-\frac{1}{2})$. 
\item \underline{$\nu_{b'}=(\frac{1}{4},\frac{1}{6},...)$}. So $\nu_{b'}=(\frac{1}{4},\frac{1}{6},-\frac{1}{2}^{(4)})$, and the corresponding stratum is not contained in the weakly admissible locus. 


\item \underline{$\nu_{b'}=(\frac{2}{9},\ldots)$}. We have $3$ non-empty Newton strata, and $2$ of them are contained in the weakly admissible locus: $(\frac{2}{9},-\frac{2}{5})$ or $(\frac{2}{9},-\frac{1}{3},-\frac{1}{2})$. 

\item \underline{$\nu_{b'}=(\frac{1}{5}^{(2)},...)$}. So that $\nu_{b'}=(\frac{1}{5}^{(2)},-\frac{1}{2}^{(2)})$, and the only one Newton stratum here is not contained in the weakly admissible locus. 




\item \underline{$\nu_{b'}=(\frac{1}{i},0^{(j)},-\frac{1}{14-i-j})$, $3\leq i\leq 12$, $0\leq j\leq 12-i$}. Such a stratum is contained in the weakly admissible locus if and only if $i\geq 8$, or $i\leq 7$ and $0\leq j\leq 6-i$.  
\end{enumerate}
To summarize, we have $121$ non-empty HN-indecomposable Newton strata, and $44$ of them are contained in the weakly admissible locus. 

\subsubsection{} We take $G=\mathrm{GL}_{21}$, $\mu=(1^{(9)},0^{(12)})$ and thus $b\in B(G)_{basic}$ with $
\nu_b=\left(\frac{3}{7}^{(21)}\right)\in \mathcal N(21)$. Then $b$ has a reduction to two proper maximal standard Levi subgroups
\[
M_1=\mathrm{GL}_7\times \mathrm{GL}_{14}, \ M_2=\mathrm{GL}_{14}\times \mathrm{GL}_7\hookrightarrow G.
\]
For the first Levi subgroup, the cocharacters $w\mu$ such that $\langle \nu_b-w\mu,\chi\rangle<0$ for some $\chi\in X^*(P_1/Z_G)^+$ are
\[
w\mu\in \left\{\substack{\left((1^{(4)},0^{(3)}),(1^{(5)},0^{(9)})\right), \  \left((1^{(5)},0^{(2)}),(1^{(4)},0^{(10)})\right)\\  \left((1^{(6)},0),(1^{(3)},0^{(11)})\right), \   \left((1^{(7)}),(1^{(2)},0^{(12)})\right)}\right\} \subset \mathcal N(M_1).
\]
Thus
\[
\nu_{b_{M_1}}\mu^{w,-1} \in \left\{\substack{\left(\left(\frac{3}{7}^{(3)},-\frac{4}{7}^{(4)}),(\frac{3}{7}^{(9)},-\frac{4}{7}^{(5)}\right)\right), \  \left(\left(\frac{3}{7}^{(2)},-\frac{4}{7}^{(5)}),(\frac{3}{7}^{(10)},-\frac{4}{7}^{(4)}\right)\right)\\  \left(\left(\frac{3}{7},-\frac{4}{7}^{(6)}),(\frac{3}{7}^{(11)},-\frac{4}{7}^{(3)}\right)\right), \  \left(\left(-\frac{4}{7}^{(7)}\right),\left(\frac{3}{7}^{(12)},-\frac{4}{7}^{(2)}\right)\right)}\right\} \subset \mathcal N(M_1).
\]
Similarly, for the Levi subgroup $M_2$, , the cocharacters $w\mu$ such that $\langle \nu_b-w\mu,\chi\rangle<0$ for some $\chi\in X^*(P_1/Z_G)^+$ are
\[
w\mu\in \left\{\substack{\left((1^{(7)},0^{(7)}),(1^{(2)},0^{(5)})\right), \  \left((1^{(8)},0^{(6)}),(1,0^{(6)})\right) \\  \left((1^{(9)},0^{(5)}),(0^{(7)})\right)}\right\} \subset \mathcal N(M_2),
\]
and hence
\[
\nu_{b_{M_2}}\mu^{w,-1} \in \left\{\substack{\left(\left(\frac{3}{7}^{(7)},-\frac{4}{7}^{(7)}),(\frac{3}{7}^{(5)},-\frac{4}{7}^{(2)}\right)\right), \  \left(\left(\frac{3}{7}^{(6)},-\frac{4}{7}^{(8)}),(\frac{3}{7}^{(6)},-\frac{4}{7}\right)\right)\\  \left(\left(\frac{3}{7}^{(5)},-\frac{4}{7}^{(9)}),(\frac{3}{7}^{(7)}\right)\right)}\right\} \subset \mathcal N(M_2).
\]
Let $b'\in G(\breve{F})$ with slope vector $\nu_{b'}=\left(\frac{5}{12},-\frac{5}{9}\right)$. So $\mathcal E_{b'}=\cO(\frac{5}{9})\oplus \cO(-\frac{5}{12})$. But for each $w\mu\in \mathcal N(M_i)$ as above, and for each possible pair $(\mathcal E',\mathcal E'')$ with
\[
-w_{0,M_i}\nu_{\mathcal E'\oplus \mathcal E''}\preceq \nu_{b_{M_i}}\mu^{w,-1}
\]
one checks (by a simple drawing for example) that the maximal slope of $\mathcal E'$ is less or equal to $\frac{1}{2}<\frac{5}{9}$. But on the other hand, none of the slope of $\mathcal E''$ is $\frac{5}{9}$. Therefore, $\mathcal E_{b'}$ can not be an extension of $\mathcal E'$ by $\mathcal E''$ by our discussions in the previous \S: see for example Lemma \ref{Prop_tildeExt is Necessary condition} and the combinatorial condition in Definition \ref{Def_tildeExt}. Consequently, by Theorem \ref{thm:criterion-for-a-singule-stratum}, $\mathcal F(G,\mu,b)^{[b']}$ is entirely contained in the weakly admissible locus.


\appendix

\renewcommand{\thesection}{Appendix~\Alph{section}}

\section{Direct description of $\Ext^1$ for $\mathrm{GL}_n$ in some cases}\label{subsection_Ext=tildExt}

\renewcommand{\thesection}{\Alph{section}}

In \S~\ref{Sec_extension}, we describe  $\Ext^1$ in an inductive way. In this appendix, we want to prove $\Ext^1=\widetilde{\Ext}^1$ in some cases. We first need some combinatorial lemmas.

\begin{definition}Let $n\in\N$ and $\epsilon\in \Q\cap [0, 1)$. Recall that
\[
\cN(n)=\{\ua=(a_1,\cdots, a_n)\in \Q^n|a_1\geq a_2\geq \cdots \geq a_n\}.
\]
\begin{enumerate}
\item An element $\ua=(a_1, \cdots, a_n)\in\Q^n$ is said to have \emph{$\epsilon$-breakpoints} if the following two conditions are verified:
\begin{itemize}
\item $|\ua|:=a_1+\cdots+a_n\in\Z+\epsilon$; and
\item for any $0<i<n$ with $a_i\neq a_{i+1}$, we have $\sum_{j=1}^i a_j\in\Z+\epsilon$.
\end{itemize}
We say that $\ua$ has \emph{integral breakpoints} if it has $0$-breakpoints.
\item Set
\[
\cN(n,\epsilon):=\{\ua=(a_1,\cdots, a_n)\in\cN(n)| \ua \text{ has } \epsilon-\text{breakpoints}
\}.
\]
\item For $\ua\in \Q^n$, let $P_{\ua}: [0, n]\rightarrow \R$ be the piecewise linear function such that
\begin{itemize}
\item $P_{\ua}(0)=0$;
\item $P_{\ua}(i)=a_1+a_2+\cdots+a_i$ for $i=1,\cdots, n$; and
\item $P_{\ua}$ is linear on the segment $[i-1, i]$ for $i=1,\cdots, n$.
\end{itemize}

\end{enumerate}
\end{definition}

\begin{remark}
\begin{enumerate}
\item $\cN(n, 0)=\cN(n)$.
\item Let $0<d<n$. For $\ua\in \cN(n)$, let
\[
\tau_{>d}(\ua):=(a_{d+1},\cdots, a_n)\in \cN(n-d).
\]
Then $\tau_{>d}(\ua)\in \cN(n-d, \epsilon)$ for $\epsilon=\{-\sum_{j=1}^d a_j\}\in [0,1)$.
\end{enumerate}
\end{remark}

\begin{lemma}\label{lemma_combinatoric_A}Let $\epsilon\in [0,1)$. Let $\ua\in\cN(n, \epsilon)$ and $\uc\in\cN(n)$, such that $a_i\geq c_i$ for all $1\leq i\leq n$. Then for any $m\in\Z$ with $|\uc|\leq m\leq |\ua|$, there exists $\ub\in\cN(n)$ such that $c_i\leq b_i\leq a_i$ for all $1\leq i\leq n$ and $|\ub|=m$.
\end{lemma}
\begin{proof}We prove by induction on $|\ua|-|\uc|\in\N+\epsilon$. If $|\ua|-|\uc|\leq 1$, there is nothing to prove. So we may assume $|\ua|-|\uc|>1$. Let
\[
\begin{split}
m_0&:=\mathrm{max}\{0\leq m\leq n|m\in\Z,  P_{\ua}(m)\in\Z+\epsilon, P_{\ua}(m)-P_{\uc}(m)<\delta+\epsilon\}\\
m_1&:=\mathrm{min}\{0\leq m\leq n|m\in\Z,  P_{\ua}(m)\in\Z+\epsilon, P_{\ua}(m)-P_{\uc}(m)\geq \delta+\epsilon\}\\
n_0&:=\mathrm{max}\{0\leq m\leq n|m\in\Z,  P_{\uc}(m)\in\Z, P_{\ua}(m)-P_{\uc}(m)<\delta+\epsilon\}\\
n_1&:=\mathrm{min}\{0\leq m\leq n|m\in\Z,  P_{\uc}(m)\in\Z, P_{\ua}(m)-P_{\uc}(m)\geq \delta+\epsilon\}
\end{split}
\]
where
\[
\delta=\begin{cases}0, &\text{if }\epsilon\neq 0\\ 1, &\text{if }\epsilon=0\end{cases}.
\]
Then $a_i$ is constant for $m_0<i\leq m_1$ and $c_i$ is constant for $n_0< i\leq n_1$. Moreover, \[\mathrm{max}(m_0, n_0)<\mathrm{min}(m_1, n_1)\] as $P_{\ua}-P_{\uc}$ is an increasing function. It suffices to find $\ub\in\cN(n)$ such that $|\uc|< |\ub|<|\ua|$.

Let
\[
\tilde{b}_i:=\begin{cases}c_i, & i\leq n_0\\ \frac{P_{\ua}(m_1)-\delta-\epsilon-P_{\uc}(n_0)}{m_1-n_0}, &n_0< i\leq m_1\\
a_i, &i>m_1\end{cases}
\]
By definition, $|\tilde{\ub}|=|\ua|-\delta-\epsilon$ and therefore $|\uc|< \tilde{\ub}< |\ua|$.

\emph{Claim: $c_i\leq \tilde{b}_i\leq a_i$ for any $n_0< i\leq m_1$.}

Then the existence of $\ub$ follows from Lemma \ref{lemma_reorder b}. Now it remains to prove the Claim.

 For any $n_0<i \leq m_1$, it suffices to show \begin{eqnarray}\label{eqn_two inequalities}c_{n_1}=\frac{P_{\uc}(n_1)-P_{\uc}(n_0)}{n_1-n_0}\leq \tilde{b}_i\leq a_{m_1}=\frac{P_{\ua}(m_1)-P_{\ua}(m_0)}{m_1-m_0}.\end{eqnarray}
  We first prove the first inequality. If $m_1=n_1$, then it holds by the definition of $m_1$.

 If $m_1>  n_1$, the first inequality is equivalent to $c_{n_1}\leq  \frac{P_{\ua}(m_1)-\delta-\epsilon-P_{\uc}(n_1)}{m_1-n_1}$. This follows from  \[c_{n_1}\leq a_{n_1}= \frac{P_{\ua}(m_1)-P_{\ua}(n_1)}{m_1-n_1}\leq  \frac{P_{\ua}(m_1)-\delta-\epsilon-P_{\uc}(n_1)}{m_1-n_1}\] by the definition of $n_1$.

 If $m_1< n_1$, the first inequality is equivalent to
 \[\frac{P_{\uc}(n_1)+\delta+\epsilon-P_{\ua}(m_1)}{n_1-m_1}\leq \frac{P_{\ua}(m_1)-\delta-\epsilon-P_{\uc}(n_0)}{m_1-n_0}.\] This follows from  \[\frac{P_{\uc}(n_1)+\delta+\epsilon-P_{\ua}(m_1)}{n_1-m_1}\leq \frac{P_{\ua}(n_1)-P_{\ua}(m_1)}{n_1-m_1}\leq\frac{P_{\ua}(m_1)-P_{\ua}(n_0)}{m_1-n_0}\leq  \frac{P_{\ua}(m_1)-\delta-\epsilon-P_{\uc}(n_0)}{m_1-n_0},\] where the inequality in the middle holds because $\ua$ is decreasing.

 For the second inequality in (\ref{eqn_two  inequalities}), we again distinguish three subcases.

If $m_0=n_0$, then it's obvious by the definition of $m_0$.

If $m_0> n_0$, then the second inequality is equivalent to \[\frac{P_{\ua}(m_0)-\delta-\epsilon-P_{\uc}(n_0)}{m_0-n_0}\leq \frac{P_{\ua}(m_1)-P_{\ua}(m_0)}{m_1-m_0}=a_{n_1}.\] This holds because the left hand side is bounded by $c_{n_1}$ by the defintion of $m_0$.

If $m_0< n_0$, then the second inequality is equivalent to
\[\frac{P_{\ua}(m_1)-\delta-\epsilon-P_{\uc}(n_0)}{m_1-n_0}\leq \frac{P_{\uc}(n_0)+\delta+\epsilon-P_{\ua}(m_0)}{n_0-m_0}.\] This follows from
\[\frac{P_{\ua}(m_1)-\delta-\epsilon-P_{\uc}(n_0)}{m_1-n_0}\leq \frac{P_{\ua}(m_1)-P_{\ua}(n_0)}{m_1-n_0} \leq\frac{P_{\ua}(n_0)-P_{\ua}(m_0)}{n_0-m_0}\leq \frac{P_{\uc}(n_0)+\delta+\epsilon-P_{\ua}(m_0)}{n_0-m_0}.\]

\end{proof}

\begin{lemma}\label{lemma_reorder b}Suppose $\ua\in\cN(n, \epsilon)$, $\uc\in\cN(n)$. Let $\tilde{\ub}=(\tilde{b}_1,\cdots,\tilde{b}_n)\in\Q^n$ has integral breakpoints such that $a_i\geq \tilde{b}_i \geq c_i$ for all $1\leq i\leq n$. Then there exists $\ub:=(b_1,\cdots, b_n)\in S_n\tilde{\ub}$ such that $\ub\in\cN(n)$ and $a_i\geq b_i\geq c_i$ for all $1\leq i\leq n$.
\end{lemma}
\begin{proof}This can be checked directly.
\end{proof}


\begin{lemma}\label{lemma_combinatoric_C} Suppose $\epsilon_1, \epsilon_2\in[0, 1)$. Let $\ua\in\cN(n, \epsilon_1)$ and $\uc\in\cN(n, \epsilon_2)$ such that $a_i\geq c_i$ for all $1\leq i\leq n$. Assume that $c_1=\cdots=c_n$. Then for any $m\in\Z$ such that $|\uc|\leq m\leq |\ua|$, there exists $\ub\in\cN(n)$ such that $c_i\leq b_i\leq a_i$ for all $1\leq i\leq n$ and $|\ub|=m$.
\end{lemma}
\begin{remark}If we do not require all the coordinates in $\uc$ are equal, then Lemma \ref{lemma_combinatoric_C} does not hold in general. For example, take $n=3$, $\ua=(\frac{4}{7}, \frac{4}{7}, 0)$ and $\uc=(\frac{5}{9}, \frac{5}{9}, -1)$. Then $|\ua|=\frac{8}{7}>1>|\uc|=\frac{1}{9}$, but there doesn't exist $\ub\in\cN(3)$ such that $a_i\geq b_i\geq c_i$ for $i=1, 2, 3$ and $|\ub|=1$.
\end{remark}

\begin{proof}[Proof of Lemma \ref{lemma_combinatoric_C}] We prove by induction on $n$. If $n=1$, it's obvious. Now we deal with general $n$. Suppose $|\ua|\geq \lceil|\uc|\rceil=:m_0$. According to Lemma \ref{lemma_combinatoric_A}, it suffices to find $\ub\in \cN(n)$ such that $a_i\geq b_i\geq c_i$ for all $i$ and $|\ub|=m_0$. Let
\[
\uc'=(a_{n},\cdots, a_{n})\in \cN(n, \epsilon_3)
\]
where $\epsilon_3=\{na_n\}$. If $na_n\geq m_0$, then $\ub=(\frac{m_0}{n}, \cdots, \frac{m_0}{n})\in\cN(n)$ is the desired element. If $|\uc'|=na_n< m_0$. Then $|\uc'|<m_0\leq |\ua|$. Write $a_n=\frac{s}{r}$ with $r$ and $s$ coprime and $r\geq 1$. Let
\[
\tau_{\leq n-r}(\ua):=(a_1, \cdots a_{n-r})\in\cN(n-r, \epsilon_1),
\]
and $\tau_{\leq n-r}(\uc')\in\cN(n-r, \epsilon_3)$ is defined in the same way. Then
\[
|\tau_{\leq n-r}(\uc')|=|\uc'|-s< m_0-s\leq |\tau_{\leq n-r}(\ua)|=|\ua|-s.
\]
By the induction hypothesis for $n-r$, we may find $(b_1, \cdots, b_{n-r})\in \cN(n-r)$ such that $a_n\leq b_i\leq a_i$ for $1\leq i\leq n-r$ and $\sum_{i=1}^{n-r}b_i=m_0-s$. Then \[b=(b_1,\cdots, b_{n-r}, \underbrace{a_n,\cdots, a_n}_r)\in\cN(n)\] satisfies the desired properties.
\end{proof}

\begin{lemma}\label{lemma_combinatoric_D} Let $0<d<n$. Suppose $\epsilon_1, \epsilon_2\in[0, 1)$. Let $\ua\in\cN(n, \epsilon_1)$ and $\uc\in\cN(n-d, \epsilon_2)$ such that $a_{i+d}\geq c_i$ for all $1\leq i\leq n-d$. Assume that either $c_1=\cdots=c_{n-d}$ or $\uc\in \Z^{n-d}$. Then for any $m\in\Z$ such that $|\uc|+dc_1\leq m\leq |\ua|$, there exists $\ub\in\cN(n)$ such that $|\ub|=m$, $b_i\leq a_i$ for all $1\leq i\leq n$, $c_i\leq b_{i+d}$ and for $1\leq i\leq n-d$.
\end{lemma}
\begin{proof}Let \[\uc':=\begin{cases}(\underbrace{c_1,\cdots, c_1}_n), &\text{ if } c_1=\cdots=c_{n-d}\\(\underbrace{c_1,\cdots, c_1}_{d}, \uc), &\text{ if } \uc\in\Z^{n-d}\end{cases}.\] In the first case, we apply Lemma \ref{lemma_combinatoric_C} to the pair $(\ua, \uc')$ and in the second case, we apply Lemma \ref{lemma_combinatoric_A}.
\end{proof}

\begin{proposition}\label{Prop_one_semistable}Suppose $0<r<n$. Let $\uc\in \cN(r)$, $\ud\in \cN(n-r)$ and $\ua\in\cN(n)$. Suppose  $\uc$ or $\ud$ is semistable, then $\cO(\ua)$ is an extension of $\cO(\uc)$ by $\cO(\ud)$ if and only if $\ua\in\widetilde{\Ext}^1(\uc, \ud)$. Equivalently, $\Ext^1(\uc, \ud)=\widetilde{\Ext}^1(\uc, \ud)$.
\end{proposition}

\begin{remark} Proposition \ref{Prop_one_semistable} is also obtained independently by Hong in \cite[Theorem 1.1]{Ho2}.
\end{remark}

\begin{proof}[Proof of Proposition \ref{Prop_one_semistable}]By Proposition \ref{Prop_tildeExt is Necessary condition}, it suffices to prove the if part. By duality, we may assume that $\ud=d_1^{(n-r)}$ is semistable.

\textit{Claim: We may assume that $d_1\leq c_r$.}

 Indeed, if $d_1> c_r$, then there exists $1\leq m< r$ such that $c_{m+1}<d_1\leq c_m$. There exists natural bijections:
\[\Ext^1(\uc, \ud)\simeq \Ext^1(\tau_{\leq m}\uc, \ud) \text{ and }\widetilde{\Ext}^1(\uc, \ud)\simeq \widetilde{\Ext}^1(\tau_{\leq m}\uc, \ud).\]
We may replace $\uc$ by $\tau_{\leq m}\uc$. The Claim follows.

We prove by induction on $n$. If $n=1$, then it's trivial. Now assume that the proposition holds for $<n$. We also use induction on $c_1-d_1\in \frac{\N}{(n!)^2}$. If $c_1=d_1$, then $c_1=\ldots=c_r=d_1$, and $\uc$ is also semistable, then the result follows from Corollary \ref{coro_Hansen}. Now we consider the general case $c_1>d_1$. Moreover, we assume $\uc$ non semi-stable: otherwise it follows again from Corollary \ref{coro_Hansen}.

Take $\ua\in \widetilde{\Ext}^1(\uc, \ud)$. Without loss of generality, we assume that the slope $d_1$ of $\ud$ is not bigger than $a_n$. Then there exists a partition of $\{1, 2, \cdots, n\}$ into two disjoint subsets
\[
H=\{h_1< \cdots <h_r\}, \quad  K=\{k_1<\cdots <k_{n-r}\},
\]
and $b\in \mathbb Q^n$ such that
\begin{itemize}
\item $(b_{h_1}, \cdots, b_{h_r})\in S_r \uc$, and $(b_{k_1},\cdots, b_{k_{n-r}})\in S_{n-r}\ud$;
\item $b_i\geq a_i$ if $i\in H$, and $b_i\leq a_i$ if $i\in K$; and
\item $\ub\geq \ua$.
\end{itemize}
Let
\[
m:=\mathrm{min}\{0\leq s\leq n| a_{s+1}\leq c_{r}\}.
\]
In particular, $P_{\ua}(m)\in \mathbb Z$. Moreover, as $a_{m+1}\leq c_r$ and $d_1\leq a_n$, we may assume that
\begin{itemize}
\item[-] $
H\cap (m, n]=\{i\in\N| m+1\leq i\leq h_r\}$;
\item[-]
$b_{m+1}=\cdots= b_{h_r-l}=c_{r}$ for some $0\leq l< h_r-m$ and $b_{i}>c_r$ for $h_r-l<i\leq h_r$.
\end{itemize}
As $a_m>c_r$, we have $c_r\cdot (h_r-l-m)\in \mathbb Z$. Let
\[
\ua':=\tau_{>h_r-l}(\ua)\in \cN(n-h_r+l, \epsilon_1)\quad \textrm{and}\quad \uc':=\tau_{>h_r}(\ub)\in\cN(n-h_r, \epsilon_2)
\]
for some $\epsilon_1,\epsilon_2\in \mathbb Q\cap [0,1)$. So $\uc'=d_1^{(n-h_r)}$. By Lemma \ref{lemma_combinatoric_D}, there exists $\ue'\in \cN(n-h_r+l)$ such that
\[
c'_{i}\leq e'_{i+l}\leq a'_{i+l}
\]
for all $1\leq i\leq n-h_r$, and
\begin{eqnarray*}
 |\ue'|& = & |\ua'|+(a_{m+1}+\ldots+a_{h_r-l})-c_r(h_r-l-m)\\ & =& P_{\ua}(n)-P_{\ua}(m)-c_r(h_r-l-m)\in \mathbb Z.
\end{eqnarray*}
Note that
\[
|\uc'|+ld_1\leq |\ue'|=|\ua'|+(a_{m+1}+\ldots+a_{h_r-l})-c_r(h_r-l-m)\leq |\ua'|.
\]
Here the first inequality holds because
\begin{eqnarray*}
a_{m+1}+\ldots+a_n & \geq &  b_{m+1}+\ldots+b_n \\ & \geq & (b_{m+1}+\ldots+b_{h_r-l})+(b_{h_r-l+1}+\ldots+b_n) \\ & \geq &  (h_r-l-m)c_r+(n-h_r+l)d_1.
\end{eqnarray*}
In particular,
\[
|\ua'|-|\ue'|=P_{\ub}(h_r-l)-P_{\ub}(m)-P_{\ua}(h_r-l)+P_{\ua}(m).
\]

Let
\[
e:=(a_1, \cdots, a_m, e'_{1}, \cdots, e'_{n-h_r+l})\in \cN(n-h_r+l+m).
\]
Note that
\[
\uc=(\tau_{\leq r-h_r+l+m}(\uc), c_r^{(h_r-l-m)}).
\]
By Proposition \ref{Prop_classification of extension by Ext^1}, to complete the proof, it suffices to verify
\[
\ua\in\Ext^{1}(c_{r}^{(h_r-l-m)}, \ue)\text{ and }\ue\in\Ext^1(\tau_{\leq r-h_r+l+m}(\uc), \ud).
\]
But we can check
\[
\ue\in\widetilde{\Ext}^1(\tau_{\leq r-h_r+l+m}(\uc), \ud)=\Ext^1(\tau_{\leq r-h_r+l+m}(\uc), \ud)
\]
where the equality follows from the induction hypothesis, and

\[\ua\in\widetilde{\Ext}^{1}(c_{r}^{(h_r-l-m)}, \ue)\simeq\widetilde{\Ext}^{1}(c_{r}^{(h_r-l-m)}, \ue')=\Ext^{1}(c_{r}^{(h_r-l-m)}, \ue')\simeq\Ext^{1}(c_{r}^{(h_r-l-m)}, \ue)
\]
where the first and third bijection follows from the proof of the Claim and the equality in the middle follows from induction on $n$ if $m>0$ or the fact that
$c_r-e'_{n-h_r+l}< c_1-d_1$ as $\uc$ is non semi-stable and the induction hypothesis on $c_1-d_1$ if $m=0$.
\end{proof}

\begin{proposition}\label{Prop_two is sum of line bundles} Suppose $0<r<n$. Let $\uc\in \cN(r)$, $\ud\in \cN(n-r)$ and $\ua\in\cN(n)$. Suppose that two elements among $\ua$, $\ub$ and $\uc$ are with all coordinates in $\Z$, then $\cO(\ua)$ is an extension of $\cO(\uc)$ by $\cO(\ud)$ if and only if $\ua\in\widetilde{\Ext}^1(\uc, \ud)$. In particular, $\Ext^1(\uc, \ud)=\widetilde{\Ext}^1(\uc, \ud)$ if $\ud\in \Z^{n-r}$ and $\uc\in\Z^{r}$.
\end{proposition}
\begin{proof} 

The proof is similar to that of Proposition \ref{Prop_one_semistable}. By Proposition \ref{Prop_tildeExt is Necessary condition}, it suffices to prove the if part.

Assume first $\ud\in \Z^{n-r}$ and $\uc\in\Z^{r}$.  We prove by induction on $r$. When $r=1$, it's proved in Proposition \ref{Prop_one_semistable}. Next, suppose $r\geq 2$ and take $\ua\in \widetilde{\Ext}^1(\uc, \ud)$. Then there exist a partition
\[
\{1, 2, \cdots, n\}=H\coprod K
\]
with
$
H=\{h_1< \cdots <h_r\}$, and $K=\{k_1<\cdots <k_{n-r}\}$, and $\ub\in \mathbb Q^n$ such that
\begin{itemize}
\item $b_i\geq a_i$ if $i\in H$, $b_i\leq a_i$ if $i\in K$;
\item  $\ub\geq \ua$;
\item $(b_{h_1}, \cdots, b_{h_r})=(c_1,\ldots,c_r)$, and $(b_{k_1},\cdots, b_{k_{n-r}})=(d_1,\ldots,d_{n-r})$.
\end{itemize}
Let
\[
m:=\mathrm{min}\{0\leq s\leq n| a_{s+1}=a_{h_r}\}.
\]
In particular $P_{\ua}(m)\in\Z$. Without loss of generality, assume $H\cap (m, n]=\{i\in\N| m+1\leq i\leq h_r\}$.
Let
\[
\ua':=\tau_{>m+1}(\ua)\in \cN(n-m-1, \epsilon),\quad \textrm{and} \quad \ub':=\tau_{>h_r}(\ub)\in\cN(n-h_r)
\]
for some $\epsilon$. By Lemma \ref{lemma_combinatoric_C}, there exists $\ue'\in\cN(n-m-1)$ such that $e'_i\leq a'_i$ for all $i$, $b_j\leq e'_{h_r-m-1+j}$ for all $1\leq j\leq n-h_r$, and
\[
P_{\ua}(n)-P_{\ua}(m)=|\ue'|+b_{h_r}.
\]
Let $\ue=(a_1,\cdots, a_m, \ue')\in \cN(n-1)$.
Note that $\uc=(\tau_{\leq r-1}(\uc), c_r)$. We can check
\[
\ua\in \widetilde{\Ext}^{1}( (c_r), \ue)=\Ext^{1}( (c_r), \ue)
\]
and
\[
\ue\in \widetilde{\Ext}^{1}(\tau_{\leq r-1}(\uc), \ud)=\Ext^{1}(\tau_{\leq r-1}(\uc), \ud)\]
Indeed, the first equality follows from Proposition \ref{Prop_one_semistable}, while the second equality follows from induction hypothesis. Then we conclude with Proposition \ref{Prop_classification of extension by Ext^1}.

Finally suppose that all the coordinates of $\ua$ and $\ud$ are integers. As above, we prove by induction on $r$. When $r=1$, it's proved in Proposition \ref{Prop_one_semistable}. Next, suppose $r\geq 2$. Then there exist a partition $
\{1, 2, \cdots, n\}=H\coprod K$ with
\[
H=\{h_1< \cdots <h_r\},\quad \textrm{and}\quad K=\{k_1<\cdots <k_{n-r}\},
\]
and $\ub\in \Q^n$ satisfying
\begin{itemize}
\item $(b_{h_1},\ldots, b_{h_r})=(c_1,\ldots, c_r)$; and $(b_{k_1},\ldots, b_{k_{n-r}})=(d_1,\ldots,d_{n-r})$;
\item $b_i\geq a_i$ if $i\in H$, $b_i\leq a_i$ if $i\in K$;
\item $\ub\geq \ua$.
\end{itemize}
Let
\[
m:=\mathrm{min}\{0\leq s\leq r| c_{s+1}=c_{r}\}.
\]
In particular $(r-m)c_r\in\Z$, and $b_{h_{m+1}}=b_{h_{m+2}}=\ldots=b_{h_r}=c_r$. Let $\ua'\in \cN(n-h_{m+1}+1-r+m)$ be the element obtained from
\[
\tau_{\geq h_{m+1}}\ua\in  \cN(n-h_{m+1}+1)
\]
by removing the coordinates $a_{h_{m+1}},a_{h_{m+2}},\ldots, a_{h_r}$, and $\ub'\in \cN(n-h_{m+1}+1-r+m)$ defined from $\ub$ in a similar way. Then $b_i'\leq a_i'$ and all the coordinates of $\ub'$ are integers. As $\ub\geq \ua$,
\[
|\ua'|+(a_{h_{m+1}}+a_{h_{m+2}}+\ldots+a_{h_r})\geq |\ub'|+(r-m)c_r.
\]
So it is easy to see that there exists $\ue'\in \cN(n-h_{m+1}+1-r+m)$ with all the coordinates in $\mathbb Z$, such that $b_i'\leq e_i'\leq a_i'$ and that
\[
|\ua'|-|\ue'|=(r-m)c_r-(a_{h_{m+1}}+a_{h_{m+2}}+\ldots+a_{h_r})\in \mathbb Z.
\]
Let $\ue=(a_1,\cdots, a_{h_{m+1}-1}, \ue')\in \cN(n-r+m)$.
Note that $\uc=(\tau_{\leq m}(\uc), c_r^{(r-m)})$. We can check that
\[
\ua\in \widetilde{\Ext}^{1}( c_r^{(r-m)}, \ue)=\Ext^{1}( c_r^{(r-m)}, \ue)
\]
and
\[
\ue\in \widetilde{\Ext}^{1}(\tau_{\leq m}(\uc), \ud)\cap \mathbb Z^{n-r+m}=\Ext^{1}(\tau_{\leq m}(\uc), \ud)\cap \mathbb Z^{n-r+m}.
\]
Indeed, the first equality follows from Proposition \ref{Prop_one_semistable}, while the second equality follows from induction hypothesis. This concludes the proof by Proposition \ref{Prop_classification of extension by Ext^1}.
\end{proof}

\begin{remark} In Proposition \ref{Prop_two is sum of line bundles}, if all the coordinates of $\ua,\uc$ and $\ud$ are integers, then the same proof as that of the main result of \cite{Schl} shows that $\ua\in \Ext^1(\uc,\ud)$ if and only if $\widetilde{\Ext}^1(\uc,\ud)$. The argument adopted above is inspired from the proof of Schlesinger in \cite{Schl}. This result is also mentioned in the remark after Example 4.5 in \cite{Ho2}.
\end{remark}

\bibliographystyle{amsalpha}

\end{document}